\documentclass{amsart}
\usepackage{graphicx}
\usepackage{verbatim}
\usepackage{color}
\usepackage{amssymb}
\usepackage{amsmath}
\usepackage{here}
\theoremstyle{plain}
\usepackage{mathrsfs}
\newcommand\no[1]{}

\newtheorem{theorem}{Theorem}[section]
\newtheorem*{thm}{Main Theorem}
\newtheorem{lemma}[theorem]{Lemma}

\newtheorem{proposition}[theorem]{Proposition}
\newtheorem{conjecture}{Conjecture}

\theoremstyle{definition}
\newtheorem{remark}[theorem]{Remark}
\newtheorem{example}[theorem]{Example}
\newtheorem{definition}[theorem]{Definition}
\newtheorem*{ack}{Acknowledgments}

\renewcommand{\Im}{\operatorname{Im}}

\newcommand{\C}{\mathbb{C}}
\newcommand{\R}{\mathbb{R}}
\newcommand{\Z}{\mathbb{Z}}
\newcommand{\Q}{\mathbb{Q}}
\newcommand{\SL}{\mathrm{SL}}
\newcommand{\SU}{\mathrm{SU}}
\newcommand{\Id}{\operatorname{Id}}
\newcommand{\Int}{\operatorname{Int}}
\newcommand{\Res}{\operatorname{Res}}
\newcommand{\Vol}{\operatorname{Vol}}
\newcommand{\tr}{\operatorname{tr}}
\newcommand{\sign}{\operatorname{sign}}
\newcommand{\ad}{\operatorname{ad}}
\newcommand{\modtwo}{\underset{(2)}{\equiv}}
\newcommand{\nmodtwo}{\underset{(2)}{\not\equiv}}
\newcommand{\tGamma}{\tilde{\Gamma}}
\newcommand{\tTheta}{\tilde{\Theta}}
\newcommand{\trho}{\tilde{\rho}}
\newcommand{\htau}{\hat{\tau}}

\newcommand{\CS}{\operatorname{CS}}
\newcommand{\Tor}{\operatorname{Tor}}
\newcommand{\ChernSimons}{\mathcal{CS}}
\newcommand{\Reidemeister}{\mathcal{T}}
\newcommand{\RD}{\mathcal{R}^{\Delta}}
\newcommand{\Rp}{\mathcal{R}_{+}}
\newcommand{\Rm}{\mathcal{R}_{-}}
\newcommand{\Rpm}{\mathcal{R}_{+}}
\newcommand{\RpD}{\mathcal{R}_{+}^{\Delta}}
\newcommand{\RmD}{\mathcal{R}_{-}^{\Delta}}
\newcommand{\RpmD}{\mathcal{R}_{\pm}^{\Delta}}
\newcommand{\RN}{\mathcal{R}^{\nabla}}
\newcommand{\RpN}{\mathcal{R}_{+}^{\nabla}}
\newcommand{\RmN}{\mathcal{R}_{-}^{\nabla}}
\newcommand{\RpmN}{\mathcal{R}_{\pm}^{\nabla}}
\newcommand{\HpD}{\mathcal{H}_{+}^{\Delta}}
\newcommand{\HmD}{\mathcal{H}_{-}^{\Delta}}
\newcommand{\HpmD}{\mathcal{H}_{\pm}^{\Delta}}
\newcommand{\HpN}{\mathcal{H}_{+}^{\nabla}}
\newcommand{\HmN}{\mathcal{H}_{-}^{\nabla}}
\newcommand{\HpmN}{\mathcal{H}_{\pm}^{\nabla}}
\renewcommand{\i}{\sqrt{-1}}
\begin{document}

\title[Quantum invarians of three-manifold]{Quantum invariants of three-manifolds obtained by surgeries along torus knots}

\author{Hitoshi Murakami}
\address{
Graduate School of Information Sciences,
Tohoku University,
Aramaki-aza-Aoba 6-3-09, Aoba-ku,
Sendai 980-8579, Japan}
\email{hitoshi@tohoku.ac.jp}
\author{Anh T.~Tran}
\address{
Department of Mathematical Sciences, The University of Texas at Dallas, Richardson,
TX 75080, USA}
\email{att140830@utdallas.edu}
\date{\today}
\begin{abstract}
We study the asymptotic behavior of the Witten--Reshetikhin--Turaev invariant associated with the square of the $n$-th root of unity with odd $n$ for a Seifert fibered space obtained by an integral Dehn surgery along a torus knot.
We show that it can be described as a sum of the Chern--Simons invariants and the twisted Reidemeister torsions both associated with representations of the fundamental group to the two-dimensional complex special linear group.
\end{abstract}
\thanks{H.M. was supported by JSPS KAKENHI Grant Number JP17K05239.
A.T. was supported by a JSPS postdoctoral fellowship and a grant from the Simons Foundation (\#354595)}
\maketitle
\section{Introduction}
For a closed three-manifold $M$ and an integer $r\ge2$, we denote by $\tau_r(M;\exp(2\pi\i/r))$ the Witten--Reshetikhin--Turaev quantum $\mathrm{SU}(2)$ invariant \cite{Reshetikhin/Turaev:INVEM1991,Witten:COMMP1989}.
Here E.~Witten introduced, in a physical way, the invariant by using the Chern--Simons action and the path integral, and N.~Reshetikhin and V.~Turaev defined it, in a mathematical way motivated by Witten's paper, by using quantum groups.
\par
In \cite[(2.23)]{Witten:COMMP1989}, Witten suggests that when $r\to\infty$, $\tau_r(M;\exp(2\pi\i/r))$ splits into sums of terms of $\SU(2)$ representations of $\pi_1(M)$, and each term can be expressed in terms of the associated Chern--Simons invariant and the Reidemeister torsion.
The asymptotic behavior of $\tau_r(M;\exp(2\pi\i/r))$ are studied in \cite{Freed/Gompf:COMMP1991,Jeffrey:COMMP1992,Rozansky:COMMP1995,Andersen:JREIA2013,Andersen/Hansen:JKNOT2006}.
In \cite[(1.36)]{Freed/Gompf:COMMP1991}, D.~Freed and R.~Gompf gave a precise formula as a speculation and did computer calculation for the asymptotic behaviors of the invariants of some Seifert fibered three-manidolds including lens spaces.
L.~Jeffrey \cite{Jeffrey:COMMP1992} confirmed the formula for lens spaces and for torus bundles over circles.
L.~Rozansky \cite{Rozansky:COMMP1995,Rozansky:COMMP1996_2} obtained the asymptotic expansion of $\tau_r(M;\exp(2\pi\i/r))$ for Seifert fibered spaces $(O,g;0\mid0;\alpha_1,\beta_1;\alpha_2,\beta_2,\dots,\alpha_k,\beta_k)$.
See also \cite{HikamiCOMMP2006}.
\par
The following conjecture is a part of the Asymptotic Expansion Conjecture \cite{Andersen:problem,Andersen:JREIA2013,Andersen/Himpel:QT2012} by J.~Andersen.
\begin{conjecture}[Asymptotic Expansion Conjecture]\label{conj:AEC}
There exist constants $b_j\in\C$ and $d_j\in\Q$ such that
\begin{equation*}
  \tau_r(M;\exp(2\pi\i/r))
  =
  \sum_{j=1}^{n}b_je^{2\pi\i rq_j}r^{d_j}
  +
  O(r^{-1})
\end{equation*}
for $r\to\infty$, where $0=q_0<q_1<q_2<\dots<q_n$ are different values of the Chern--Simons invariants of $M$ associated with representations of $\pi_1(M)$ to $\mathrm{SU}(2)$.
See \cite{Andersen:problem} for the original conjecture.
\end{conjecture}
The Asymptotic Expansion Conjecture is proved for all finite order mapping tori in \cite{Andersen:JREIA2013} (see also \cite{Andersen/Himpel:QT2012}),
for three-manifold obtained by rational Dehn surgeries along the figure-eight knot \cite{Andersen/Hansen:JKNOT2006}.
See \cite{Hansen2005} for more general Seifert fibered spaces.
\par
Note that the Asymptotic Expansion Conjecture states that $\tau_r(M;\exp(2\pi\i/r))$ grows at most polynomially when $r\to\infty$, which is true by topological quantum field theory.
\par
For an odd integer $n\ge3$, we can define another quantum invariant denoted by $\htau_n(M;\exp(4\pi\i/n))$.
See Subsection~\ref{subsec:WRT} for the combinatorial definition together with that of $\tau_r(M;\exp(2\pi\i/r))$.
In \cite{Chen/Yang:QT2018}, Q.~Chen and T.~Yang studies the asymptotic behavior of $\htau_n(M;\exp(4\pi\i/n))$ by using computer for closed three-manifolds obtained by integral Dehn surgeries along the knots $4_1$ and $5_2$, and proposed the following conjecture:
\begin{conjecture}[{\cite[Conjecture~1.2]{Chen/Yang:QT2018}}]\label{conj:CY}
For a closed, hyperbolic three-manifold $M$, $\htau_n(M;\exp(4\pi\sqrt{-1}/n))$ grows exponentially and
\begin{equation}\label{eq:CY}
  4\pi\i\lim_{n\to\infty}\frac{\log\htau_n(M;\exp(4\pi\sqrt{-1}/n))}{n}=\CS(M)+\i\Vol(M),
\end{equation}
where $\Vol(M)$ is the hyperbolic volume of $M$ and $\CS(M):=2\pi\operatorname{cs}(M)$ with $\operatorname{cs}(M)$ the Chern--Simons invariant of the complete hyperbolic metric of $M$ \cite{Chern/Simons:ANNMA21974,Cheeger/Simons:LNM1167}.
\end{conjecture}
\begin{remark}
The quantity $\Vol(M)+\i\CS(M)$ is often called the complex volume.
See \cite[Theorem~2.8]{Garoufalidis/Thurston/Zickert:DUKMJ2015}.
\end{remark}
\begin{remark}
Note that this conjecture says that $\htau_n(M;\exp(4\pi\i/n))$ grows exponentially when $n\to\infty$ if $M$ is hyperbolic.
Note also that the first author calculated the asymptotic behavior of $\tau_n(M;\exp(2\pi\sqrt{-1}/n))$ by using computer for three-manifold obtained by integral Dehn surgeries of the figure-eight knot $4_1$, and observed that, possibly because of lack of precision, it grows exponentially and \eqref{eq:CY} holds if we replace $4\pi\sqrt{-1}$ with $2\pi\sqrt{-1}$ \cite{Murakami:SURIK2000}.
\end{remark}
\par
In \cite{Ohtsuki:AGT2018}, T.~Ohtsuki proved Conjecture~\ref{conj:CY} in the case of three-manifolds obtained from the figure-eight knot by integral surgeries.
He also generalized the conjecture above as follows.
\begin{conjecture}\label{conj:Ohtsuki}
For a closed, hyperbolic three-manifold $M$, we have
\begin{equation*}
  \htau_n(M;\exp(4\pi\sqrt{-1}/n))
  \sim
  (\text{some root of unity})\times\omega(M)n^{3/2}
  \exp\left(\frac{n}{4\pi\i}(\CS(M)+\i\Vol(M))\right)
\end{equation*}
for $n\to\infty$, where $\omega(M)$ involves the square root of the twisted Reidemeister torsion associated with the holonomy representation \rm{(}see for example \cite{Porti:MAMCAU1997}\rm{)}, and we write $f(n)\sim g(n)$ for $n\to\infty$ if $\lim_{n\to\infty}\frac{f(n)}{g(n)}=1$.
\end{conjecture}
Note that the conjecture above implies Conjecture~\ref{conj:CY}.
\par
Compare this with the complexified version \cite{Murakami/Murakami/Okamoto/Takata/Yokota:EXPMA02} of Kashaev's conjecture for hyperbolic knots \cite{Kashaev:LETMP97} in terms of the colored Jones polynomial \cite{Murakami/Murakami:ACTAM12001}, where a knot is called hyperbolic if its complement $S^3\setminus{K}$ possesses a unique complete hyperbolic structure with finite volume.
\begin{conjecture}[Complexification of Kashaev's Conjecture]\label{conj:cKc}
Let $K\subset S^3$ be a hyperbolic knot.
Then we have
\begin{equation}\label{eq:cKc}
  2\pi\lim_{n\to\infty}
  \frac{\log J_n(K;\exp(2\pi\i/n))}{n}
  =
  \Vol(S^3\setminus{K})+\i\CS(S^3\setminus{K}),
\end{equation}
where $J_n(K;q)$ is the colored Jones polynomial of $K$ associated with the $n$-dimensional representation of $\mathfrak{sl}(2;\C)$ and $\CS$ is the Chern--Simons invariant for a knot \cite{Meyerhoff:LMSLN112}.
\end{conjecture}
\begin{remark}
R.~Kashaev's original conjecture is that for any \emph{hyperbolic} knot $K$, $\lim_{N\to\infty}\frac{2\pi\log\left|\langle K\rangle_{N}\right|}{N}=\Vol(S^3\setminus{K})$, where $\langle K\rangle_{N}$ is his invariant defined in \cite{Kashaev:MODPLA95} depending on an integer $N\ge2$.
\end{remark}
Conjecture~\ref{conj:cKc} is proved for $4_1$ by T.~Ekholm, $5_2$ by T.~Ohtsuki \cite{Ohtsuki:QT2016}, $6_1,6_2,6_3$ by Ohtsuki and Y.~Yokota \cite{Ohtsuki/Yokota:MATPC2018}.
\par
For general knots, J.~Murakami and the first author proposed the following conjecture generalizing Kashaev's conjecture, which was also complexified later.
\begin{conjecture}[Complexification of the Volume Conjecture]\label{conj:cvc}
For any knot $K$, we have
\begin{equation}\label{eq:cvc}
  2\pi\lim_{n\to\infty}
  \frac{\log J_n(K;\exp(2\pi\i/n))}{n}
  =
  \Vol(S^3\setminus{K})+\i\CS(S^3\setminus{K}).
\end{equation}
Here we define $\Vol(S^3\setminus{K}):=v_3\left\|S^3\setminus{K}\right\|$ with $v_3$ the volume of the regular, ideal, hyperbolic tetrahedron and $||S^3\setminus{K}||$ the simplicial volume \rm{(}or Gromov's invariant\rm{)} \rm{(}\cite{Gromov:INSHE82}, \cite[Chapter~6]{Thurston:GT3M}\rm{)}, and $\CS$ is a topological Chern--Simons invariant (defined by the left hand side), which coincides with the Chern--Simons invariant when $K$ is hyperbolic.
\end{conjecture}
As for non-hyperbolic knots, R.~Kashaev and O.~Tirkkonen proved the Volume Conjecture for torus knots.
Note that since the simplicial volume equals the sum of those of hyperbolic pieces of the knot complement after the torus decomposition (also known as the Jaco--Shalen--Johannson decomposition) \cite{Jaco/Shalen:MEMAM1979,Johannson:1979}, the simplicial volume of a torus knot vanishes.
See \cite{Zheng:CHIAM22007} and \cite{Le/Tran:JKNOT2010} for other cases, that is, for non-hyperbolic knots with non-zero volume.
\par
Now it would be natural to study the asymptotic behavior of $\htau_n(M;\exp(4\pi\sqrt{-1}/n))$ in the case where $M$ is not hyperbolic, and derive a formula similar to Conjecture~\ref{conj:AEC}.
In this paper, we calculate it for a certain family of Seifert fibered spaces with three singular fibers.
\par
Let $T(a,b)$ be the torus knot of type $(a,b)$ in the three-sphere $S^3$ for positive coprime integers $a$ and $b$.
Put $X:=S^3\setminus\Int{N(T(a,b))}$, where $N(T(a,b))$ is the regular neighborhood of $T(a,b)$ in $S^3$ and $\Int$ is the interior.
Note that $X$ is a compact three-manifold with boundary $\partial{X}$ a torus $S^1\times S^1$.
For an integer $p$ we denote by $X_p$ the closed three-manifold obtained from $S^3$ by $p$-Dehn surgery.
\par
It can be shown that $X_p$ is the Seifert fibered three-manifold of type $S(-a/c,b/d,p-ab)$, where $c$ and $d$ are integers such that $ad-bc=1$ \cite{Moser:PACJM1971} (see Figure~\ref{fig:link} for a rational surgery description for $S(r_1,r_2,r_3)$ ($r_1,r_2,r_3\in\Q$).
Note that $S(r_1,r_2,r_3)$ is the Seifert fibered space $(O,o;0\mid0;\alpha_1,\beta_1;\alpha_2,\beta_2,\alpha_3,\beta_3)$ with $r_i=\alpha_i/\beta_i$ ($i=1,2,3$) in Seifert's notation \cite[Satz~ 5]{Seifert:ACTAM11933} (see \cite{Seifert/Threlfall:Topology} for an English translation).
We give a proof in Subsection~\ref{subsec:surgery_description} because we need to carefully choose the signs of the surgery coefficients.
\begin{figure}[h]
\includegraphics[scale=0.3]{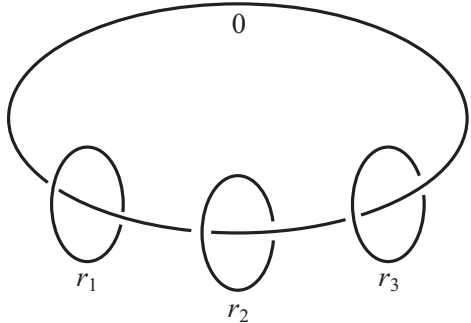}
\caption{A surgery description for $X_p$, where $0$, $r_1$, $r_2$, and $r_3$ are surgery coefficients.}
\label{fig:link}
\end{figure}
\par
Assume that $n$ is odd, $p-ab>0$, and $\gcd(p,ab)=1$.
Then we can express the asymptotic behavior of $\htau_n(X_p;\exp(4\pi\sqrt{-1}/n))$ in terms of the Chern--Simons invariants and the Reidemeister torsions associated with representations of $\pi_1(X_p)$ to $\SL(2;\C)$.
In fact we prove
\begin{thm}[Theorem~\ref{thm:main}]
\begin{equation*}
  \htau_n(X_p;\exp(4\pi\sqrt{-1}/n))
  =
  \frac{(-1)^{p+1}n^{3/2}}{2\pi}
  \left(A(n)+B(n)n^{-1/2}+O(n^{-1})\right)
\end{equation*}
with
\begin{equation*}
\begin{split}
  A(n)
  :=&
  2e^{\frac{n+1}{4}\pi\sqrt{-1}}
  \left(
    \sum_{\substack{(h,k,l)\in\mathcal{H}_{+}^{\Delta}\\ h\modtwo k\modtwo l\modtwo1}}
    -
    \sum_{\substack{(h,k,l)\in\mathcal{H}_{+}^{\nabla}\\ h\modtwo k\modtwo l\modtwo0}}
  \right)
  \Reidemeister_{+}^{\rm{Irr}}(h,k,l)e^{n\ChernSimons_{+}^{\rm{Irr}}(h,k,l)\pi\i}
  \\
  &+
  2e^{\frac{n+1}{4}\pi\sqrt{-1}}
  \left(
    \sum_{\substack{(h,k,l)\in\mathcal{H}_{-}^{\Delta}\\ h\modtwo k\modtwo l\modtwo1}}
    -
    \sum_{\substack{(h,k,l)\in\mathcal{H}_{-}^{\nabla}\\ h\modtwo k\modtwo l\modtwo0}}
  \right)
  \Reidemeister_{-}^{\rm{Irr}}(h,k,l)e^{n\ChernSimons_{-}^{\rm{Irr}}(h,k,l)\pi\i}
\end{split}
\end{equation*}
and
\begin{equation*}
  B(n)
  :=
  \frac{1}{2}\i(-1)^{a+b+ab}e^{n(1-p)\pi\i/4}
  \sum_{0<l<(p-1)/2}
  \Reidemeister^{\rm{Abel}}(l)e^{n\ChernSimons^{\rm{Abel}}(l)\pi\i},
\end{equation*}
where $\Reidemeister_{\pm}^{\rm{Irr}}(h,k,l)$ and $\Reidemeister^{\rm{Abel}}(l)$ are related to the twisted Reidemeister torsions, and $\ChernSimons_{\pm}^{\rm{Irr}}(h,k,l)$ and $\ChernSimons^{\rm{Abel}}(l)$ are related to the Chern--Simons invariant as described in the following.
\begin{itemize}
\item
$\left(\Reidemeister^{\rm{Abel}}(l)\right)^{-2}=\pm\Tor(X_{p};\trho^{\rm{Abel}}_{l})$, where $\Tor(X_{p};\trho^{\rm{Abel}}_{l})$ is the homological Reidemeister torsion of $X_p$ twisted by the Abelian representation $\trho^{\rm{Abel}}_{l}\colon\pi_1(X_p)\to\SL(2;\C)$,
\item
$\ChernSimons^{\rm{Abel}}(l)=\CS(X_{p};\trho^{\rm{Abel}}_{l})\pmod\Z$, where $\CS(X_{p};\trho^{\rm{Abel}}_{l})$ is the Chern--Simons invariant of $X_p$ associated with $\trho^{\rm{Abel}}_{l}$,
\item
$\left(\Reidemeister_{\pm}^{\rm{Irr}}(h,k,l)\right)^{-2}=\left|\mathbb{T}(X_{p};\trho_{p-ab-h,a-k,b-l}^{\rm{Irr}})\right|$, where $\Tor(X_p;\trho_{p-ab-h,a-k,b-l}^{\rm{Irr}})$ is the homological Reidemeister torsion of $X_p$ twisted by the irreducible representation $\trho_{p-ab-h,a-k,b-l}^{\rm{Irr}}\colon\pi_1(X_p)\to\SL(2;\C)$,
\item
$\ChernSimons(h,k,l)\equiv\CS(X_p;\trho_{p-ab-h,a-k,b-l}^{\rm{Irr}})\pmod{2\Z}$, where $\CS(X_p;\trho_{p-ab-h,a-k,b-l}^{\rm{Irr}})$ is the Chern--Simons invariant of $X_p$ associated with $\trho_{p-ab-h,a-k,b-l}^{\rm{Irr}}$,
\item
$\mathcal{H}^{\Delta}_{\pm}\subset\mathcal{H}$ and $\mathcal{H}^{\nabla}_{\pm}\subset\mathcal{H}$ are certain index sets, where $\mathcal{H}:=\{(h,k,l)\in\Z^3\mid0<h<p-ab,0<k<a,0<l<b,h\equiv k\equiv l\pmod{2}\}$.
\end{itemize}
\end{thm}
\begin{remark}
Let $\operatorname{TV}_{n}(M)$ be the Turaev--Viro invariant of a three-manifold $M$ \cite{Turaev/Viro:TOPOL1992}.
It is known that $\operatorname{TV}_n(M)=|\htau_n(M;\exp(4\pi\i/n))|^2$ \cite{Roberts:TOPOL1995,Benedetti/Petronio:JKNOT1996}, where we choose the parameter of $\operatorname{TV}_n(M)$ so that this equality holds.
\par
Chen and Yang proposed the following conjecture:
\begin{conjecture}[{\cite[Conjecture~1.1]{Chen/Yang:QT2018}}]
\label{conj:Chen_Yang}
For a hyperbolic three-manifold $M$ with possibly non-empty boundary, we have
\begin{equation*}
  2\pi\lim_{n\to\infty}
  \frac{\log\operatorname{TV}_n(M)}{n}
  =
  \Vol(M),
\end{equation*}
where $n$ runs over all odd integers.
\end{conjecture}
\par
For a link complement, the following conjecture was proposed by R.~Detcherry, E.~Kalfagianni and Yang.
\begin{conjecture}[{\cite[Conjecture~5.1]{Detcherry/Kalfagianni/Yang:QT2018}}]
\label{conj:link}
For any link $L$ in $S^3$, we have
\begin{equation*}
  2\pi\lim_{n\to\infty}
  \frac{\log\operatorname{TV}_n(S^3\setminus{L})}{n}
  =
  v_3\|S^3\setminus{L}\|,
\end{equation*}
where $n$ runs over all odd integers.
\end{conjecture}
Note that they also proved that the Turaev--Viro invariant of a link complement can be calculated as a sum of the squares of the absolute values of the colored Jones polynomials.
They also proved Conjecture~\ref{conj:link} in the case where $L$ is a knot with volume zero.
\par
In \cite{Detcherry/Kalfagianni:ANNSE12019}, Detcherry and Kalfagianni proposed a similar conjecture for the Turaev--Viro invariant:
\begin{conjecture}[{\cite[Conjecture~8.1]{Detcherry/Kalfagianni:ANNSE12019}}]
\label{conj:DK}
For any compact orientable three-manifold $M$ with empty or toroidal boundary, we have
\begin{equation*}
  2\pi\limsup_{n\to\infty}\frac{\log\operatorname{TV}_n(M)}{n}
  =
  v_3\|M\|,
\end{equation*}
where $n$ runs over all odd integers.
\end{conjecture}
Compare this with Conjecture~\ref{conj:Chen_Yang}; they replace $\lim$ with $\limsup$.
\par
See \cite{Maria/Rouille:arXiv2020} for computer calculations of the asymptotic behaviors of the Turaev--Viro invariants.
\par
Finally, note that if we could show that $A(n)$ or $B(n)$ does not vanish, then Conjecture~\ref{conj:DK} holds for $X_p$.
\end{remark}
For the asymptotic behavior of $\htau_n(X_p;\exp(2\pi\sqrt{-1}/n))$ evaluated at $\exp(2\pi\i/n)$, see \cite{Jeffrey:COMMP1992}, \cite{Freed/Gompf:COMMP1991}, \cite[7.2 The asymptotic expansion conjecture]{Ohtsuki:2002}, and \cite{Andersen:JREIA2013}.
\par
For coprime, odd integers $p_1$ and $p_2$ greater than or equal to three, put $S(-2,p_1,p_2)$ the Seifert fibered space with three singular fibers with index $-2$, $p_1$ and $p_2$.
Then the fundamental group $\pi_1(S(-2,p_1,p_2))$ has the following presentation:
\begin{equation*}
  \pi_1(M_{p_1,p_2})
  =
  \langle
    \alpha,\beta,\gamma,f
    \mid
    [\alpha,f]=[\beta,f]=[\gamma,f]=\alpha^{2}f=\beta^{p_1}f=\gamma^{p_2}f
    =\alpha\beta\gamma=1
  \rangle
\end{equation*}
For coprime, odd integers $k_1$ and $k_2$ with $0<k_1<p_1$ and $0<k_2<p_2$, let $\hat{\rho}_{k_1,k_2}\colon\pi_1(S(-2,p_1,p_2))\to\SL(2;\C)$ be an irreducible representation such that the eigenvalues of $\hat{\rho}_{k_1,k_2}(\beta)$ ($\hat{\rho}_{k_1,k_2}(\gamma)$, respectively) are $\exp\left(\pm\frac{k_1\pi\i}{p_1}\right)$ ($\exp\left(\pm\frac{k_2\pi\i}{p_1}\right)$, respectively).
Then in \cite[Proposition~4.17]{Ohtsuki/Takata::COMMP2019}, it is shown that the conjugacy class of $\hat{\rho}_{k_1,k_2}$ is unique.
\par
Put $\widehat{\ChernSimons}(S(-2,p_1,p_2);\hat{\rho}_{k_1,k_2}):=\frac{1}{4}\left(\frac{1}{2}-\frac{k_1^2}{p_1}-\frac{k_2^2}{p_2}\right)$ and $\hat{\Reidemeister}(S(-2,p_1,p_2);\hat{\rho}_{k_1,k_2}):=\frac{8\sin\left(\frac{k_1\pi}{p_1}\right)\sin\left(\frac{\pi}{p_2}\right)}{\sqrt{2p_1p_2}}$.
Then Ohtsuki and Takata proved the following theorem.
\begin{theorem}[{\cite[Theorem~1.3]{Ohtsuki/Takata::COMMP2019}}]\label{thm:Ohtsuki_Takata}
For odd integers $n$, we have the following asymptotic expansion:
\begin{equation*}
\begin{split}
  &\htau_n\left(S(-2,p_1,p_2);e^{4\pi\i/n}\right)
  \\
  =&
  \frac{(-1)^{(n-1)/2}e^{\pi\i/n}e^{-n(p_1+p_2)\pi\i/4}n^{3/2}}{8\pi}
  \\
  &\times
  \left(
    \sum_{\mathrm{SU}(2)}
    +
    2\sum_{\substack{\SL(2;\R)\\\frac{k_1}{p_1}+\frac{k_2}{p_2}<\frac{1}{2}}}
  \right)
  (-1)^{(k_1+k_2)/2}\hat{\Reidemeister}\left(S(-2,p_1,p_2);\hat{\rho}_{k_1,k_2}\right)
  e^{n\widehat{\ChernSimons}\left(S(-2,p_1,p_2);\hat{\rho}_{k_1,k_2}\right)\pi\i}
  +O(n).
\end{split}
\end{equation*}
Here
\begin{itemize}
\item
$\sum_{\mathrm{SU}(2)}$ means that the summation is over all $\mathrm{SU}(2)$ representations.
In this case the corresponding range is
\begin{equation}\label{eq:SU2_OT}
  \left\{
    (k_1,k_2)\in\Z^2
    \Biggm|
    \left|\frac{k_1}{p_1}-\frac{k_2}{p_2}\right|<\frac{1}{2}<
    \frac{k_1}{p_1}+\frac{k_2}{p_2}<\frac{3}{2},
    k_1\modtwo k_2\modtwo1
  \right\},
\end{equation}
\item
$\sum_{\SL(2;\R)}$ means that the summation is over all $\SL(2;\R)$ representations with $\frac{k_1}{p_1}+\frac{k_2}{p_2}<\frac{1}{2}$.
In this case the corresponding range is
\begin{equation*}
  \left\{
    (k_1,k_2)\in\Z^2
    \Biggm|
    \frac{k_1}{p_1}+\frac{k_2}{p_2}<\frac{1}{2},
    k_1\modtwo k_2\modtwo1
  \right\},
\end{equation*}
\item
$\widehat{\ChernSimons}\left(S(-2,p_1,p_2);\hat{\rho}_{k_1,k_2}\right)$ is the Chern--Simons invariant of $S(-2,p_1,p_2)$ associated with $\hat{\rho}_{k_1,k_2}$,
\item
$\hat{\Reidemeister}\left(S(-2,p_1,p_2);\hat{\rho}_{k_1,k_2}\right)^{-2}$ is the homological Reidemeister torsion of $S(-2,p_1,p_2)$ twisted by $\hat{\rho}_{k_1,k_2}$.
\end{itemize}
\end{theorem}
Compare this formula with ours when $a=2$.
If $a=2$, then our formula can be simplified as follows (see Example~\ref{ex:Main_Theorem_a_2}).
\begin{equation*}
\begin{split}
  &\htau_n\left(X_p;e^{4\pi\i/n}\right)
  \\
  =&
  \frac{e^{(n+1)\pi\i/4}n^{3/2}}{8\pi}
  \left(
    2\sum_{\substack{(h,1,l)\in\HmD\\ b-l\equiv2\pmod4}}+
    \sum_{(h,1,l)\in\HpD\setminus\HmD}
  \right)
  \Reidemeister_{+}^{\rm{Irr}}(h,1,l)e^{n\ChernSimons_{+}^{\rm{Irr}}(h,1,l)\pi\i}
  \\
  &-
  \frac{\i e^{n(1-p)\pi\i/4}n}{4\pi}
  \sum_{0<l<(p-1)/2}
  \Reidemeister^{\rm{Abel}}(l)e^{n\ChernSimons^{\rm{Abel}}(l)\pi\i}
  +
  O(n^{1/2})
  \\
  =&
  \frac{e^{(n+1)\pi\i/4}n^{3/2}}{8\pi}
  \left(
    2\sum_{
      \substack{\text{$\trho_{h,1,l}^{\rm{Irr}}$: $\SU(2)$-representation}
      \\
      \frac{l}{2b}+\frac{h}{p-2b}<1, b-l\equiv2\pmod4}}
    +
    \sum_{\text{$\trho_{h,1,l}^{\rm{Irr}}$: $\SU(1,1)$-representation}}
  \right)
  \Reidemeister_{+}^{\rm{Irr}}(h,1,l)e^{n\ChernSimons_{+}^{\rm{Irr}}(h,1,l)\pi\i}
  \\
  &-
  \frac{\i e^{n(1-p)\pi\i/4}n}{4\pi}
  \sum_{0<l<(p-1)/2}
  \Reidemeister^{\rm{Abel}}(l)e^{n\ChernSimons^{\rm{Abel}}(l)\pi\i}
  +
  O(n^{1/2}).
\end{split}
\end{equation*}
\begin{ack}
The authors would like to thank Professor E.~Kalfagianni and Professor T.~Ohtsuki for helpful discussions and suggestions.
\end{ack}

\section{Preliminaries}
In this section, we describe some basic facts that are necessary for this paper.
\subsection{Quantum three-manifold invariants}\label{subsec:WRT}
\par
We will explain how to compute the $SU(2)$ Witten--Reshetikhin--Turaev invariant \cite{Witten:COMMP1989,Reshetikhin/Turaev:INVEM1991} following \cite{Lickorish:JKNOT1993}.
See also \cite{Blanchet/Habegger/Masbaum/Vogel:TOPOL1992} and \cite[Chapter~13]{Lickorish:1997}.
\par
Let $M$ be a closed, oriented three-manifold.
Suppose that $M$ is obtained from the three-sphere $S^3$ by Dehn surgery on a framed link presented by a link diagram $D:=D_1\cup D_2\cup\dots\cup D_m$.
Note that the surgery coefficients are integers because they are given by the framings.
\par
If $C_i$ ($i=1,2,\dots,m$) is a link diagram in an annulus, $\langle C_1,C_2,\dots,C_m\rangle_D$ means the Kauffman bracket \cite{Kauffman:TOPOL1987} of the link diagram obtained by inserting $C_i$ in the regular neighborhood of $D_i$ in the plane, respecting the under/over crossing information.
We extend the definition to linear combinations of link diagrams multilinearly.
\par
For a non-negative integer $k$, let $S_k(x)$ be the $k$-th Chebyshev polynomial defined by $S_0(x)=1$, $S_1(x)=x$, and
\begin{equation*}
  S_{k+1}(x)=xS_k(x)-S_{k-1}(x).
\end{equation*}
Denoting by $\alpha$ the core of an annulus, let $S_k(\alpha)$ be the linear combination of $\alpha^j$ determined by $S_k(x)$, where $\alpha^j$ means the $j$ parallels of $\alpha$.
So $S_k(\alpha)$ is a linear combination of link diagram in the annulus.
It is known that $S_k(x)$ is obtained from the $k$-th Jones--Wenzl idempotent \cite{Wenzl:CRMAR1987} by closing it along the annulus.
Put
\begin{equation*}
  \omega
  :=
  \sum_{k=0}^{n-2}
  \Delta_k\times S_k(\alpha),
\end{equation*}
with $\Delta_k:=(-1)^k\frac{A^{2(k+1)}-A^{-2(k+1)}}{A^2-A^{-2}}$.
\par
Let $r\ge3$ be an integer.
Suppose that either $A$ is either a primitive $4r$-th root of unity, or $A$ is a primitive $2r$-th root of unity and $r$ is odd.
Then
\begin{equation*}
  \frac{\langle\omega,\omega,\dots,\omega\rangle_{D}}
  {\langle\omega\rangle_{U_+}^{b_+}\langle\omega\rangle_{U_-}^{b_-}}
\end{equation*}
is an invariant of $M$, where
\begin{itemize}
\item
$b_{+}$ ($b_{-}$, respectively) is the number of positive (negative, respectively) eigenvalues of the linking matrix of $D$, where the diagonal entries are the framings and the off diagonal entries are the linking numbers \cite[Theorem~5]{Lickorish:JKNOT1993}.
\item
$U_{\pm}$ is a knot diagram of the unknot with framing $\pm1$.
\end{itemize}
\par
Now we define
\begin{equation*}
  \tau_r(M;\exp(2\pi\i/r))
  :=
  \left.
    \frac{\langle\omega,\omega,\dots,\omega\rangle_{D}}
    {\langle\omega\rangle_{U_+}^{b_+}\langle\omega\rangle_{U_-}^{b_-}}
  \right|_{A=\exp\left(\frac{\pi\i}{2r}\right)}
\end{equation*}
for any integer $r\ge3$, and
\begin{equation}\label{def:htau}
  \htau_n(M;\exp(4\pi\i/n))
  :=
  \left.
    \frac{\langle\omega,\omega,\dots,\omega\rangle_{D}}
    {\langle\omega\rangle_{U_+}^{b_+}\langle\omega\rangle_{U_-}^{b_-}}
  \right|_{A=\exp\left(\frac{\pi\i}{n}\right)}
\end{equation}
Here we use $\exp(2\pi\i/r)$ and $\exp(4\pi\i/n)$ because we want to mention the value of $A^{4}$ that is usually used for the parameter of the Jones polynomial (see Section~\ref{sec:U}).
\subsection{Reidemeister torsion}
Here we explain the Reidemeister torsion twisted by the adjoint action of a representation from the fundamental group of a closed three-manifold to the Lie group $\SL(2;\C)$.
\par
For a closed, oriented, connected three-manifold $M$, let $\rho\colon\pi_1(M)\to\SL(2;\C)$ be a representation, where we take a basepoint of $M$ appropriately.
Denoting by $\tilde{M}$ be universal cover of $M$, the $i$-th chain group $C_{i}(\tilde{M};\C)$ and the Lie algebra $\mathfrak{sl}(2;\C)$ can be regarded as a $\Z(\pi_1(M))$.
Here an element in $\pi_1(M)$ acts on $C_{i}(\tilde{M};\Z)$ by a deck transformation, and acts on $\mathfrak{sl}(2;\C)$ by $x\cdot g:=\ad_{\rho(x)}(x)$, where $x\in\pi_1(M)$, $g\in\mathfrak{sl}(2;\C)$ and $\ad_{\rho(x)}(g):=\rho(x)^{-1}g\rho(x)$ is the adjoint action.
Then the tensor product $C_i(M;\rho):=C_i(\tilde{M};\Z)\otimes_{\Z(\pi_1(M))}\mathfrak{sl}(2;\C)$ ($i=0,1,2$) forms a chain complex
\begin{equation*}
  C_{\bullet}
  :
  \{0\}
  \to
  C_2(M;\rho)
  \xrightarrow{\partial_2}
  C_1(M;\rho)
  \xrightarrow{\partial_1}
  C_0(M;\rho)
  \to
  \{0\}.
\end{equation*}
Let $H_i(M;\rho)$ be the homology group of the chain complex $C_{\bullet}$.
Let $\mathbf{c}_i$ be a basis of $C_i(M;\rho)$, $\mathbf{h}_j$ be a basis of $H_i(M;\rho)$, and $\mathbf{b}_j$ be a set of vectors in $C_i(M;\rho)$ such that the set $\partial_i(\mathbf{b}_i)$ forms a basis of $\Im(\partial_i)$.
Then we define
\begin{equation*}
  \Tor(M;\rho)
  :=
  \frac{\left[
          \partial_2(\mathbf{b}_2)\cup\tilde{\mathbf{h}}_1\cup\bf{b}
          \bigm|\mathbf{c}_1
        \right]}
       {\left[
          \partial_1(\mathbf{b}_1)\cup\tilde{\mathbf{h}}\bigm|\mathbf{c}_0
        \right]
        \left[
          \tilde{\mathbf{h}_2}\cup\mathbf{b}_2\bigm|\mathbf{c}_2
        \right]}
\end{equation*}
and call it the (homological) Reidemeister torsion twisted by the adjoint action of $\rho$.
Here if $\mathbf{u}$ and $\mathbf{v}$ are bases of a vector space, then $\mathbf{x}\bigm|\mathbf{y}$ is the determinant of the base-change-matrix from $\mathbf{x}$ to $\mathbf{y}$.
\par
If $X$ is a three-manifold with torus boundary, then we can also define the Reidemeister torsion $\Tor_{\gamma}(M;\rho)$ if one fixes a simple closed curve $\gamma$ in $\partial{X}$.
The following facts are known to calculate the Reidemeister torsions.
\par
Let $\mu$ be the meridian and $\lambda$ be the preferred longitude.
\begin{itemize}
\item
If $\rho$ is a representation in the irreducible component indexed by $(k,l)$, then $\Tor_{\lambda}(X;\rho)$ is given by
\begin{equation}\label{eq:torsion_lambda}
  \Tor_{\lambda}(X;\rho)
  =
  \pm\frac{a^2b^2}{16\sin^2\left(\frac{k\pi}{a}\right)\sin^2\left(\frac{l\pi}{b}\right)}.
\end{equation}
Note that one can also determine the sign.
See \cite[6.2]{Dubois:CANMB2006}.
\item
Suppose that $\rho(\mu)=\begin{pmatrix}e^{u/2}&\ast\\0&e^{-u/2}\end{pmatrix}$ and $\rho(\gamma)=\begin{pmatrix}e^{w(u)/2}&\ast\\0&e^{-w(u)/2}\end{pmatrix}$ after a certain conjugation.
Then we have
\begin{equation}\label{eq:torsion_gamma}
  \Tor_{\gamma}(X;\rho)
  =
  \pm
  \frac{d\,w(u)}{d\,u}
  \Tor_{\mu}(X,\rho)
\end{equation}
from \cite[Th{\'e}or{\`e}me~4.7]{Porti:MAMCAU1997}.
\item
Let $M$ a closed three-manifold obtained from $X$ by Dehn surgery.
We assume that $M=X\cup_{i}D$, where $i\colon\partial{D}\to\partial{X}$ is a homeomorphism.
From \cite[Proposition~4.10]{Porti:MAMCAU1997}, $\Tor(M;\trho)$ is given as
\begin{equation}\label{eq:torsion_surgery}
  \Tor(M;\trho)
  =
  \pm
  \frac{\Tor_{i(\mu_D)}(X;\rho)}{(\tr\rho(i(\lambda_D)))^2-4},
\end{equation}
where $\mu_D$ and $\lambda_D$ are the meridian and the longitude of $D$, respectively.
\end{itemize}
\subsection{Chern--Simons invariant}
For a representation $\rho\colon\pi_1(M)\to\SL(2;\C)$, let $A$ be a flat connection  on $M\times\SL(2;\C)$ that induces $\rho$ as the holonomy representation.
Then the following integral is called the $\SL(2;\C)$ Chern--Simons invariant $\CS(M;\rho)$ of $M$ associated with $\rho$ \cite{Chern/Simons:ANNMA21974}.
\begin{equation*}
  \CS(M;\rho)
  :=
  \frac{1}{8\pi^2}
  \int_{M}\tr\left(A\wedge dA+\frac{2}{3}A\wedge A\wedge A\right)
  \in
  \C/\Z.
\end{equation*}
\par
Suppose that a closed three-manifold $M$ is obtained by Dehn surgery along a knot $K\subset S^3$.
Put $D:=D^2\times S^1$ and we assume that $M$ is obtained from $E:=S^3\setminus\Int{N(K)}$ and $D$ by identifying $\mu_D\subset\partial{E}$ with the meridian of $D$.
Here the meridian of $D$ is $\partial{D^2}\times\{\text{point}\}$.
We also assume that the longitude $\{\text{point in $\partial{D}^2$}\}\times S^1$ of $D$ is identified with $\lambda_D\subset\partial{E}$.
Then the following theorem is known \cite[Theorem~4.2]{Kirk/Klassen:MATHA1990}.
\begin{theorem}[Kirk--Klassen]\label{thm:Kirk_Klassen}
Let $\tilde{\rho}_0$ and $\tilde{\rho}_1$ be representations of $\pi_1(M)\to\SL(2;\C)$.
Assume that there exists a path of representations $\rho_t\colon\pi_1(E)\to\SL(2;\C)$ \rm{(}$0\le t\le1$\rm{)} avoiding parabolic representations such that $\rho_i=\tilde{\rho}_i\Bigm|_{E}$ for $i=0,1$.
We assume that after conjugation, the images of $\mu_D$ and $\lambda_D$ are as follows:
\begin{align*}
  \rho_t(\mu_D)
  &=
  \begin{pmatrix}e^{2\pi\i\alpha(t)}&0\\0&e^{-2\pi\i\alpha(t)}\end{pmatrix},
  \\
  \rho_t(\lambda_D)
  &=
  \begin{pmatrix}e^{2\pi\i\beta(t)}&0\\0&e^{-2\pi\i\beta(t)}\end{pmatrix}.
\end{align*}
\par
Then we have
\begin{equation*}
  \CS(M;\tilde{\rho}_1)-\CS(M;\tilde{\rho}_0)
  =
  2
  \int_{0}^{1}\beta(t)\frac{d\,\alpha(t)}{d\,t}\,dt
\end{equation*}
as an element in $\C/\Z$.
\end{theorem}
\begin{remark}\label{rem:CS_sign}
Our sign convention is different from that in \cite{Kirk/Klassen:MATHA1990}.
See the footnote of Page~98 in \cite{Freed/Gompf:COMMP1991}.
\end{remark}
\subsection{Surgery description}\label{subsec:surgery_description}
In this subsection, we show that $X_{p}$ is homeomorphic to the Seifert fibered space $S(O,0,-a/c,b/d,p-ab)$ by using techniques described in \cite{Rolfsen:1990}.
See also \cite{Rolfsen:PACJM1984}.
Here we assume that $b>a>0$ and $(a,b)=1$, but we do not assume that $b$ is odd for simplicity.
\par
Let $S\cong D^2\times S^1$ be an unknotted solid torus, where $D^2$ is a $2$-disk and $S^1$ is a circle.
Let $\left(T(q,r)\cup L\cup M\right)_{(s,t,u)}$ be the $3$-component link  with rational surgery coefficients, where $T(q,r)$ is the torus knot with surgery coefficient $s$ that lies on $\partial{S}$, $L$ is the core $\{0\}\times S^1$ of $S$ with coefficient $t$, and $M$ is an unknotted circle in $S^3\setminus{S}$ that is parallel to the meridian $\partial{D^2}\times\text{point}$ of $S$ with surgery coefficient $u$.
Here we assume that $s,t,u\in\Q\cup\{\infty\}$, and that the knot $T(q,r)$ presents the homology class $[q\times\text{longitude}+r\times\text{meridian}]\in H_1(\partial{S};\Z)$.
See Figure~\ref{fig:S_L_M}.
\begin{figure}[H]
\includegraphics[scale=0.3]{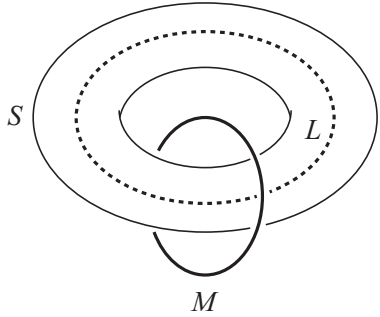}
\caption{A solid torus $S$, $L$, and $M$. The torus knot $T(q,r)$ is on $\partial{S}$.  It passes through $M$ $q$ times and goes around $L$ $r$ times.}
\label{fig:S_L_M}.
\end{figure}
\par
Putting $r_0=:b$ and $r_1:=a$, we have the following $k$ equalities from the Euclidean algorithm:
\begin{equation}\label{eq:Euclidean}
\begin{split}
  r_0=&q_1r_1+r_2,
  \\
  r_1=&q_2r_2+r_3,
  \\
  &\vdots
  \\
  r_{k-3}=&q_{k-2}r_{k-2}+r_{k-1},
  \\
  r_{k-2}=&q_{k-1}r_{k-1}+1,
  \\
  r_{k-1}=&q_k
\end{split}
\end{equation}
since $(a,b)=1$, where $r_i$ ($i=0,1,\dots,k-1$) and $q_j$ ($j=1,w,\dots,k$) are positive integers with $r_{i+1}<r_i$.
\par
We start with $\left(T(r_0,r_1)\cup L\cup M\right)_{(p,\infty,\infty)}$.
\par
If we apply $-q_1$ times twists about $L$, $\left(T(r_0,r_1)\cup L\cup M\right)_{(p,\infty,\infty)}$ is changed into $\left(T(r_2,r_1)\cup L\cup M\right)_{(p-r_1^2q_1,-1/q_1,\infty)}$ without changing the associated three-manifold.
Similarly, we see that $-q_2$ times twists about $M$ changes $\left(T(r_2,r_1)\cup M\cup L\right)_{(p-r_1^2q_1,\infty,-1/q_1)}$ into $\left(T(r_2,r_3)\cup L\cup M\right)_{(p-r_1^2q_1-r_2^2q_2,-1/q_2,-1/q_1-q_2)}$ without changing the associated three-manifold.
\par
Continuing these twists, we get the following sequence of links with rational surgery coefficient without changing the associated three-manifold:
\begin{equation*}
\begin{split}
  &\left(T(r_0,r_1)\cup L\cup M\right)_{(p,\infty,\infty)}
  \quad\xrightarrow{\text{$-q_1$ twists about $L$}}
  \\
  &\left(T(r_2,r_1)\cup L\cup M\right)_{(p-r_1^2q_1,-1/q_1,\infty)}
  \quad\xrightarrow{\text{$-q_2$ twists about $M$}}
  \\
  &\left(T(r_2,r_3)\cup L\cup M\right)
  _{(p-r_1^2q_1-r_2^2q_2,-[q_2,q_1],-1/q_2)}
  \quad\xrightarrow{\text{$-q_3$ twists about $L$}}
  \\
  &\left(T(r_4,r_3)\cup L\cup M\right)
  _{(p-r_1^  2q_1-r_2^2q_2-r_3^2q_3,-1/[q_3,q_2,q_1],-[q_3,q_2])}
  \quad\xrightarrow{\text{$-q_4$ twists about $M$}}
  \\
  &\hspace{20mm}\vdots
  \\
  &
  \begin{cases}
    \left(T(1,r_{k-1})\cup L\cup M\right)
    _{(p-\sum_{j=1}^{k-1}r_j^2q_j,
      -1/[q_{k-1},q_{k-2},\cdots,q_1],-[q_{k-1},q_{k-2},\dots,q_2])}
    &\quad\text{(if $k$ is even)},
    \\[3mm]
    \left(T(r_{k-1},1)\cup L\cup M\right)
    _{(p-\sum_{j=1}^{k-1}r_j^2q_j,
      -[q_{k-1},q_{k-2},\dots,q_1],-1/[q_{k-1},q_{k-2},\dots,q_2])}
    &\quad\text{(if $k$ is odd)},
  \end{cases}
  \\[5mm]
  &\xrightarrow[\text{if $k$ is even (odd, respectively)}]
               {\text{$-q_r=-r_{k-1}$ twists about $M$ ($L$, respectively)}}
  \\[5mm]
  &
  \begin{cases}
    \left(T(1,0)\cup L\cup M\right)
    _{(p-\sum_{j=1}^{k-1}r_j^2q_j-r_{k-1},
      -[q_k,q_{k-1},\cdots,q_1],-1/[q_{k},q_{k-1},\dots,q_2])}
    &\quad\text{(if $k$ is even)},
    \\[3mm]
    \left(T(0,1)\cup L\cup M\right)
    _{(p-\sum_{j=1}^{k-1}r_j^2q_j-r_{k-1},
      -1/[q_k,q_{k-1},\dots,q_1],-[q_k,q_{k-1},\dots,q_2])}
    &\quad\text{(if $k$ is odd)},
  \end{cases}
\end{split}
\end{equation*}
from \eqref{eq:Euclidean}, where $[m_0,m_1,m_2,\dots,m_k]$ means the continued fraction $m_0+\frac{1}{m_1+\frac{1}{m_2+\dots+\frac{1}{m_k}}}$.
Ignoring the surgery coefficients, the links $\bigl(T(1,0)\cup L\cup M\bigr)$ and $\bigl(T(0,1)\cup L\cup M\bigr)$ are depicted in Figures~\ref{fig:T_1_0_T_0_1}.
\begin{figure}[H]
\includegraphics[scale=0.3]{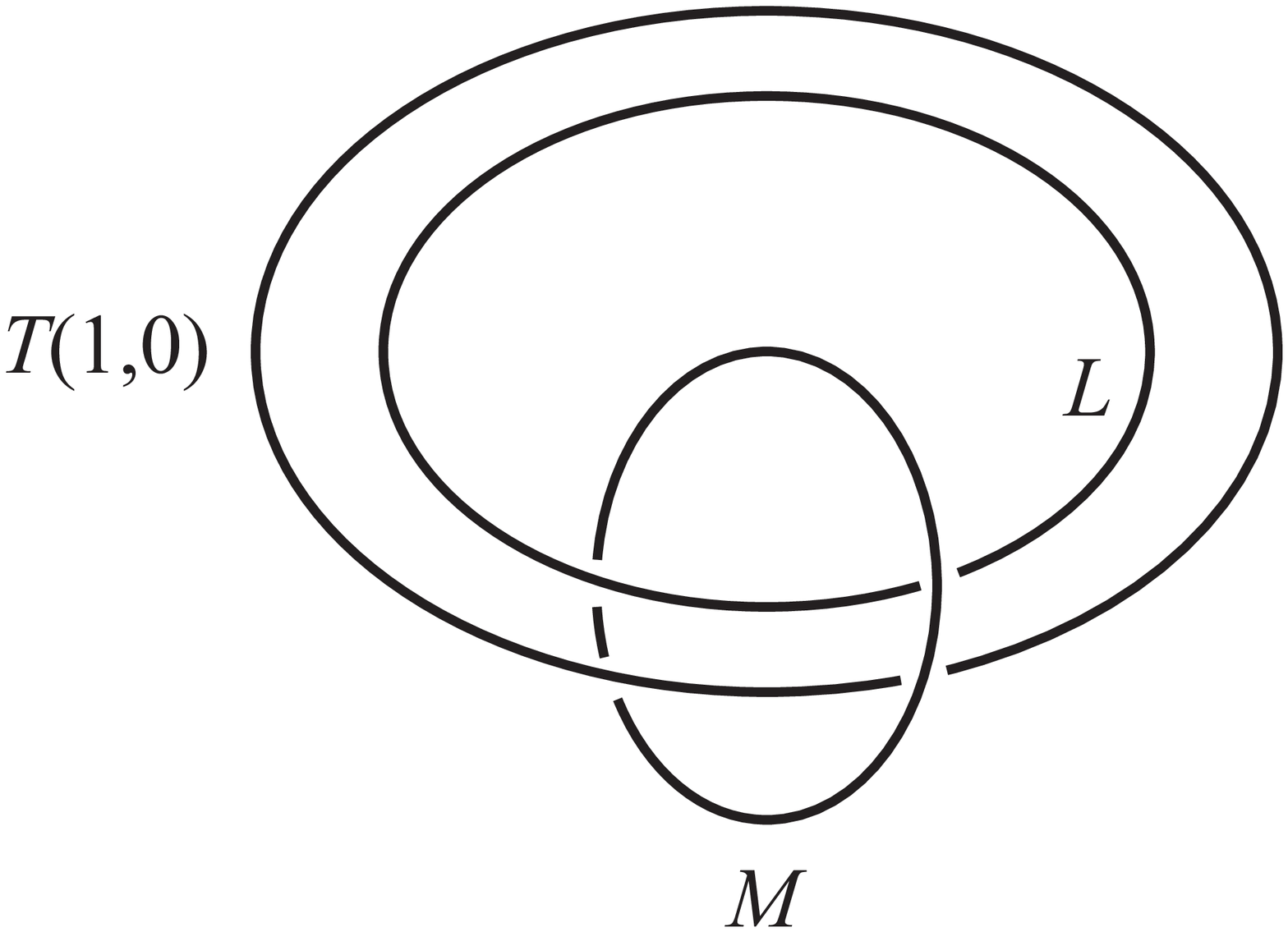}
\hspace{10mm}
\includegraphics[scale=0.3]{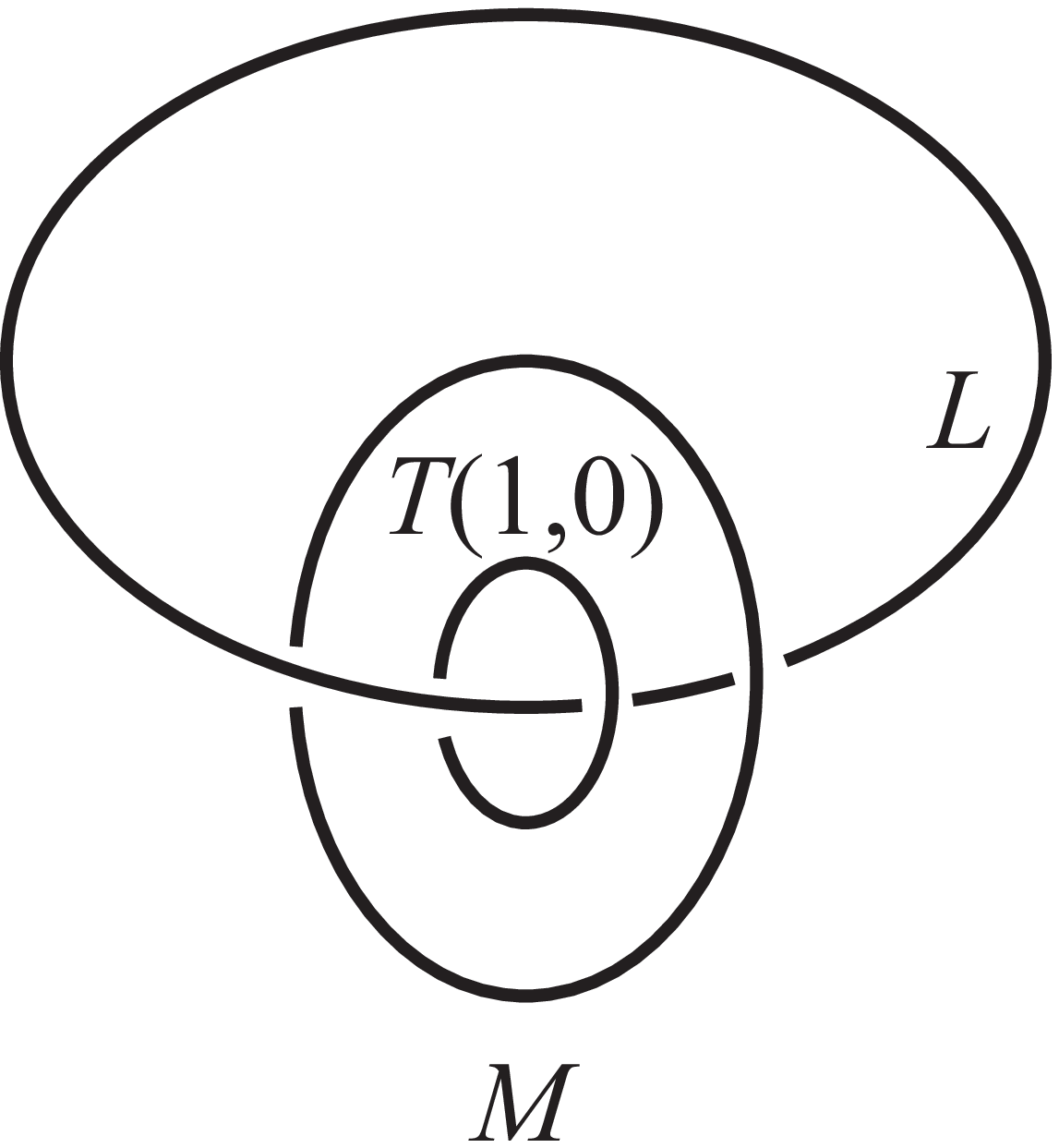}
\caption{The links $\bigl(T(1,0)\cup L\cup M\bigr)$ (left) and $\bigl(T(0,1)\cup L\cup M\bigr)$ (right).}
\label{fig:T_1_0_T_0_1}
\end{figure}
\par
Now we study the complicated coefficients.
\par
First of all, we have
\begin{equation*}
\begin{split}
  p-r_1^2q_1-\dots-r_{k-1}^2q_{k-1}-r_{k-1}
  &=
  p-r_1^2q_1-\dots-r_{k-1}r_{k-2}
  \\
  &=
  p-r_1^2q_1-\dots-r_{k-2}^2q_{k-2}-r_{k-1}r_{k-2}
  \\
  &=
  p-r_1^2q_1-\dots-r_{k-2}r_{k-3}
  \\
  &=
  \cdots
  \\
  &=
  p-r_1r_0
  =
  p-ab.
\end{split}
\end{equation*}
\par
Next, we consider the continued fractions.
\par
Dividing $j$-th equation of \eqref{eq:Euclidean} by $r_j$ ($j=1,2,\dots,k-1$), we obtain
\begin{align*}
  r_0/r_1=&q_1+r_2/r_1,
  \\
  r_1/r_2=&q_2+r_3/r_2,
  \\
  &\vdots
  \\
  r_{k-2}/r_{k-1}=&q_{k-1}+1/r_{k-1},
  \\
  r_{k-1}=&q_k.
\end{align*}
From these equations, we can see that $b/a=r_0/r_1$ is expressed as the continued fraction $[q_1,q_2,\dots,q_k]$.
We also see that $r_1/r_2=[q_2,q_3,\dots,q_k]$.
\par
Put
\begin{equation*}
  M(q)
  :=
  \begin{pmatrix}
    q&1 \\
    1&0
  \end{pmatrix}.
\end{equation*}
If we write $M(q_k)M(q_{k-1})\cdots M(q_j)=\begin{pmatrix}s_j&t_j\\u_j&v_j\end{pmatrix}$, then it is easy to prove
\begin{align*}
  \frac{s_j}{t_j}
  &=
  [q_j,q_{j+1},\dots,q_{k}],
  \\
  \frac{s_j}{u_j}
  &=
  [q_k,q_{k-1},\dots,q_{j}],
  \\
  \frac{t_j}{v_j}
  &=
  [q_{k},q_{k-1},\dots,q_{j+1}].
\end{align*}
In particular, we have
\begin{align*}
  \frac{s_1}{t_1}
  &=
  [q_1,q_2,\dots,q_{k}],
  \\
  \frac{s_1}{u_1}
  &=
  [q_k,q_{k-1},\dots,q_1],
  \\
  \frac{t_1}{v_1}
  &=
  [q_k,q_{k-1},\dots,q_{2}].
\end{align*}
Since $b/a=[q_1,q_2,\dots,q_k]=s_1/t_1$, $s_1>0$, $t_1>0$, and $(s_1,t_1)=1$, we have $b=s_1$ and $a=t_1$.
Therefore we have $[q_k,q_{k-2},\dots,q_1]=b/u_1$ and $[q_k,q_{k-1},\dots,q_2]=a/v_1$.
Now since $\det\begin{pmatrix}s_1&t_1\\u_1&v_1\end{pmatrix}=(-1)^k$, we have $bv_1-au_1=(-1)^k$.
Therefore if $k$ is even, putting $c:=-v_1$ and $d:=-u_1$, we have $[q_k,q_{k-1},\dots,q_1]=-b/d$ and $[q_k,q_{k-1},\dots,q_2]=-a/c$ with $ad-bc=1$.
If $k$ is odd, putting $c:=v_1$ and $d:=u_1$, we have $[q_k,q_{k-1},\dots,q_1]=b/d$, $[q_k,q_{k-1},\dots,q_2]=a/c$ with $ad-bc=1$.
\par
Therefore $\bigl(T(b,a)\cup L\cup M\bigr)_{(p,\infty,\infty)}$ and the link in Figure~\ref{fig:link_even_odd} give the same three-manifold.
\begin{figure}[H]
\includegraphics[scale=0.3]{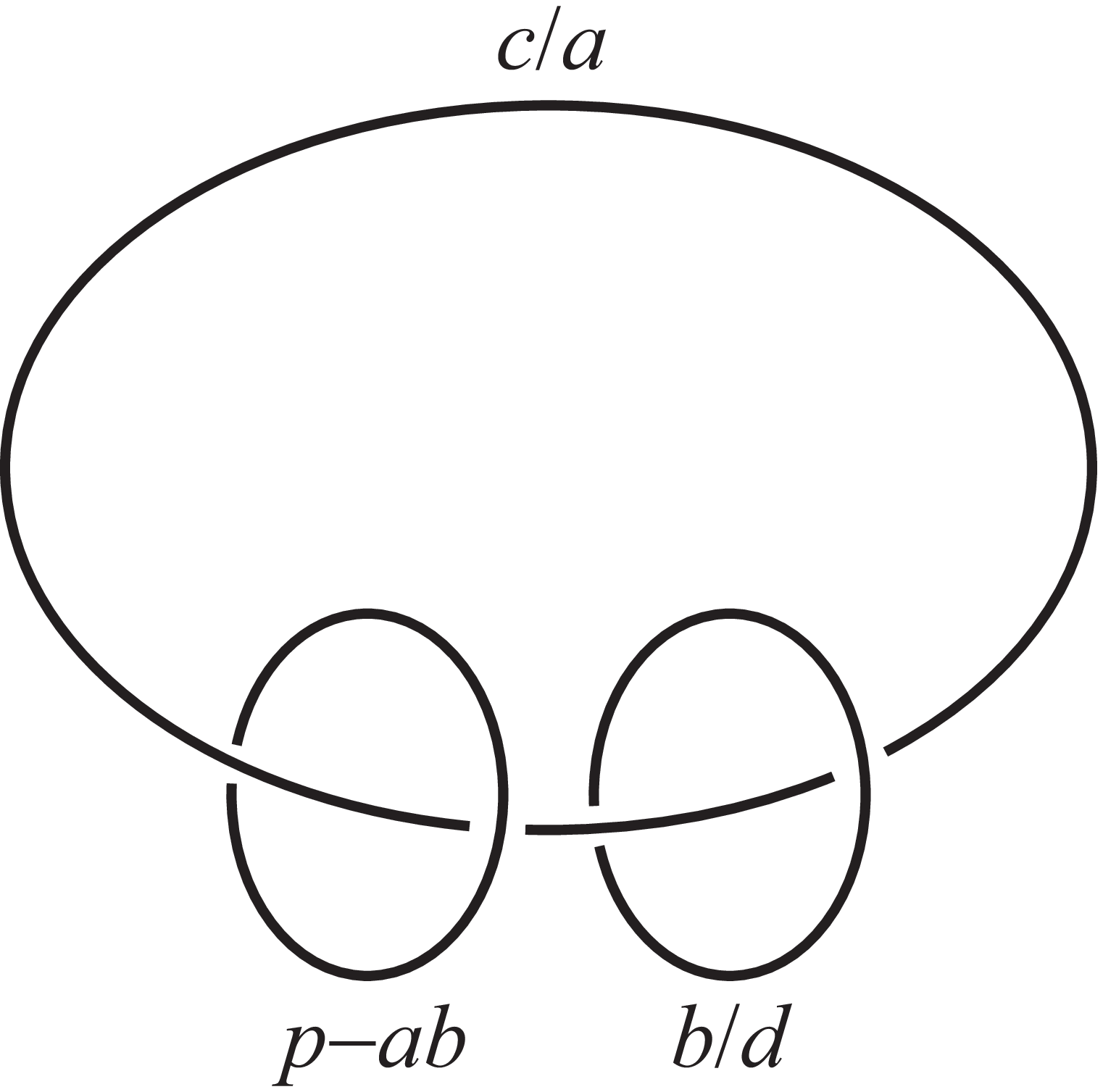}
\hspace{10mm}
\includegraphics[scale=0.3]{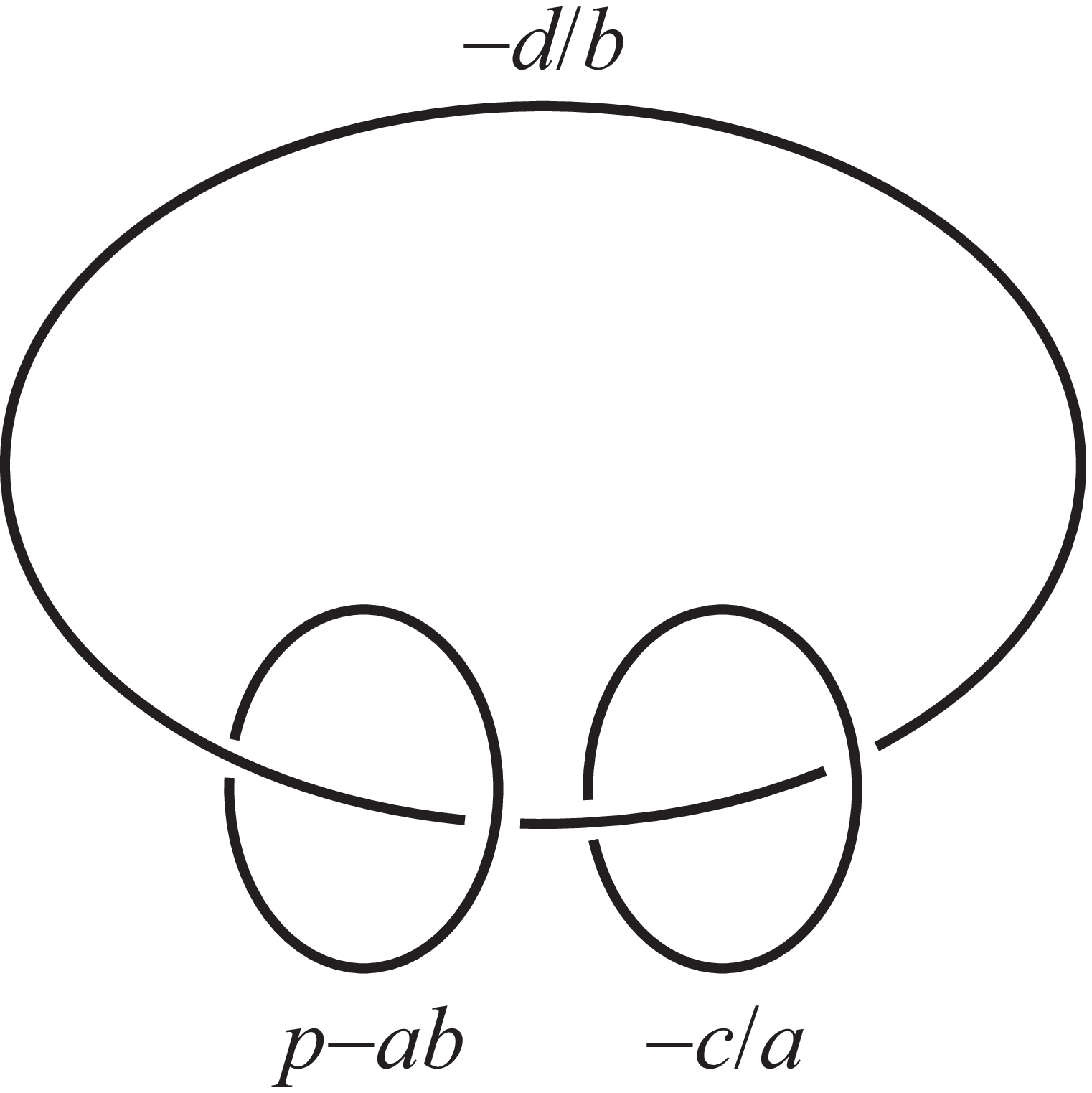}
\caption{$\bigl(T(b,a)\cup L\cup M\bigr)_{(p,\infty,\infty)}$ is equivalent to the link to the left (right, respectively) if $k$ is even (odd, respectively).}
\label{fig:link_even_odd}
\end{figure}
Here numbers beside link components are surgery coefficient.
\par
By the ``slam-dunk'' move \cite[Fig.~6]{Cochran/Gompf:TOPOL1988}, both links in Figure~\ref{fig:link_even_odd} and the link depicted in Figure~\ref{fig:link} give the same three-manifold.
So we conclude that $X_p$ is nothing but the Seifert fibered space $S(O,0,-a/c,b/d,p-ab)$.

\section{Calculation of $\htau_n(X_{p};\exp(4\pi\i/n))$}\label{sec:U}
In this section, we show that the summation in $\htau_n(X_{p};\exp(4\pi\i/n))$ can be expressed as the difference of two double integrals.
\par
In this paper we are only interested in three-manifolds obtained by Dehn surgery along a knot.
Let $K$ be a knot in the three-sphere $S^3$ and put $X:=S^3\setminus{\Int N(K)}$, where $N(K)$ is the tubular neighborhood of $K$ in $S^3$ and $\Int$ means the interior.
For an integer $p$, denote by $X_p$ the closed, oriented three-manifold obtained from $S^3$ by Dehn surgery along $K$ with coefficient $p$.
More precisely, we obtain $X_p$ from the disjoint union $X\sqcup(D^2\times S^1)$ by identifying $\partial{X}$ and $\partial(D^2\times S^1)$ so that $\partial{D^2}\times\{\ast\}\subset\partial(D^2\times S^1)$ is identified with the simple closed curve in $\partial{X}$ that goes once along $K$ and $p$ times around $K$, where $D^2$ is a disk, $S^1$ is a circle, and $\ast$ is a point in $S^1$.
So the first homology group $H_1(X_p;\Z)$ is isomorphic to $\Z/p\Z$.
\par
We denote by $\htau_n(X_p;\exp(4\pi\i/n))$ the invariant given in \eqref{def:htau}.
Let $E$ be a knot diagram presenting $K$ with framing $p$.
Then we have
\begin{equation*}
\begin{split}
  \langle S_{k}(\alpha)\rangle_{E}
  &=
  \left((-1)^kA^{(k+1)^2-1}\right)^{p}\Delta_{k}J_{k+1}(K;A^{4})
  \\
  &=
  \left(-e^{\pi\i/n}\right)^{p((k+1)^2-1)}
  (-1)^{k+1}\frac{\sin(2(k+1)\pi/n)}{\sin(2\pi/n)}J_{k+1}(K;e^{4\pi\i/n}),
\end{split}
\end{equation*}
where $J_{k+1}(K;q)$ is the $(k+1)$-dimensional colored Jones polynomial of $K$, normalized so that $J_{k+1}(\text{unknot};q)=1$.
\begin{remark}
We need to multiply by $\Delta_{k}$ since $\langle S_k(\alpha)\rangle_{U}=\Delta_{k}$ with $U$ a diagram of the unknot with no crossing.
\end{remark}
\begin{remark}
If we start with the Kauffman bracket defined as in \cite{Kauffman:TOPOL1987} and replace $A$ with $t^{-1/4}$, we obtain the original Jones polynomial $V(L;t)$ \cite{Jones:BULAM31985}.
In the formula above, we replace $q$ in the colored Jones polynomial $J_{k+1}(K;q)$ with $A^4$, and so our $2$-dimensional colored Jones polynomial $J_2(L;q)$ equals $V(L;q^{-1})$.
\end{remark}
We have the following lemma.
\begin{lemma}
The Witten--Reshetikhin-Turaev invariant $\htau_n(X_p;\exp(4\pi\i/n))$ is given by the following formula:
\begin{equation*}
  \htau_n(X_p;\exp(4\pi\i/n))
  =
  \frac{1}{\sqrt{n}\sin(2\pi/n)}\,
  e^{\sign(p)(\frac{3}{n}+\frac{n+1}{4})\pi\sqrt{-1}}\,
  \sum_{k=1}^{n-1}
  \sin^2\left(\frac{2k\pi}{n}\right)
  \left(-e^{\frac{\pi\sqrt{-1}}{n}}\right)^{p(k^2-1)}
  J_{k}\left(K;e^{4\pi\sqrt{-1}/n}\right)
\end{equation*}
where $\sign(p)$ is the sign of $p$.
\end{lemma}
\begin{proof}
When $A=e^{\pi\i/n}$, we have
\begin{equation*}
  \Delta_{k-1}
  =
  (-1)^{k-1}\frac{\sin(2k\pi/n)}{\sin(2\pi/n)}.
\end{equation*}
So we have
\begin{equation*}
  \langle\omega\rangle_{E}
  =
  \sum_{k=1}^{n-1}\Delta_{k-1}\times\langle S_{k-1}(\alpha)\rangle_{E}
  =
  \frac{1}{\sin^2(2\pi/n)}
  \sum_{k=1}^{n-1}
  \left(-e^{\pi\i/n}\right)^{p(k^2-1)}
  \sin^2(2k\pi/n)
  J_k\left(K;e^{4\pi\i/n}\right).
\end{equation*}
\par
Next we calculate $\langle\omega\rangle_{U_{\pm}}$.
Since $U_{-}$ is the mirror image of $U_{+}$, $\langle\omega\rangle_{U_-}$ can be obtained from $\langle\omega\rangle_{U_+}$ by replacing $A$ with $A^{-1}$, that is, $\langle\omega\rangle_{U_-}=\overline{\langle\omega\rangle_{U_+}}$ (complex conjugate).
So we will only calculate $\langle\omega\rangle_{U_+}$.
Since $J_k(\text{unknot};q)=1$, we have
\begin{equation*}
\begin{split}
  \langle\omega\rangle_{U_{+}}
  &=
  \frac{1}{(A^2-A^{-2})^2}
  \sum_{k=1}^{n-1}
  \left(-A\right)^{k^2-1}
  (A^{2k}-A^{-2k})^2
  \\
  &=
  \frac{1}{(A^2-A^{-2})^2}
  \sum_{k=0}^{n-1}
  \left(
    -A^{-5}(-A)^{(k+2)^2}
    +
    2A^{-1}(-A)^{k^2}
    -
    A^{-5}(-A)^{(k-2)^2}
  \right)
  \\
  &=
  \frac{2\left(-A^{-5}+A^{-1}\right)}{(A^2-A^{-2})^2}
  \sum_{l=0}^{n-1}(-A)^{l^2}
  \\
  &=
  \frac{2A^{-3}}{A^2-A^{-2}}G_n(-A),
\end{split}
\end{equation*}
where the third equality follows since $(-A)^{(l+n)^2}=(-A)^{l^2}$, and $G_n(\zeta):=\sum_{l=0}^{n-1}\zeta^{l^2}$ is the quadratic Gaussian sum.
\par
Denoting by $\left(\frac{c}{n}\right)\in\{\pm1,0\}$ the Jacobi symbol, a generalization of the Legendre symbol, it is well known that
\begin{align*}
  G_n(e^{2c\pi\i/n})
  &=
  \left(\frac{c}{n}\right)G_n(e^{2\pi\i/n}),
  \\
  \left(\frac{2}{n}\right)
  &=
  (-1)^{(n^2-1)/8},
  \\
  \left(\frac{cd}{n}\right)
  &=
  \left(\frac{c}{n}\right)\left(\frac{d}{n}\right).
\end{align*}
See for example \cite{Ireland/Rosen:GTM84}.
So we have
\begin{equation*}
\begin{split}
  G_n(-A)
  &=
  G_n\left(-e^{\pi\i/n}\right)
  \\
  &=
  G_n\left(\exp\left(\frac{\frac{n+1}{2}\times2\pi\i}{n}\right)\right)
  \\
  &=
  \left(\frac{(n+1)/2}{n}\right)G_n(e^{2\pi\i/n})
  \\
  &=
  \frac{\left(\frac{n+1}{n}\right)}{\left(\frac{2}{n}\right)}G_n(e^{2\pi\i/n})
  \\
  &=
  (-1)^{(n^2-1)/8}G_n(e^{2\pi\i/n}),
\end{split}
\end{equation*}
where the last equality holds since $\left(\frac{c}{n}\right)=\left(\frac{c'}{n}\right)$ if $c\equiv c'\pmod{n}$.
Since it is also well known that
\begin{equation*}
  G_n(e^{2\pi\i})
  =
  \begin{cases}
    \sqrt{n}&\quad\text{if $n\equiv1\pmod{4}$,}
    \\
    \i\sqrt{n}&\quad\text{if $n\equiv3\pmod{4}$,}
  \end{cases}
\end{equation*}
we have
\begin{equation*}
\begin{split}
  G_n(-A)
  &=
  \begin{cases}
    (-1)^{(n^2-1)/8}\sqrt{n}&\quad\text{if $n\equiv1\pmod{4}$,}
    \\
    (-1)^{(n^2-1)/8}\i\sqrt{n}&\quad\text{if $n\equiv3\pmod{4}$}
  \end{cases}
  \\
  &=
  (-\i)^{(n-1)/2}\sqrt{n}.
\end{split}
\end{equation*}
Hence we have
\begin{equation*}
  \langle\omega\rangle_{U_{+}}
  =
  \frac{e^{-3\pi\i/n}}{\i\sin(2\pi/n)}(-\i)^{(n-1)/2}\sqrt{n}
  =
  \frac{\sqrt{n}}{\sin(2\pi/n)}
  \times e^{-\left(\frac{3}{n}+\frac{n+1}{4}\right)\pi\i}.
\end{equation*}
So we conclude that
\begin{equation*}
  \langle\omega\rangle_{U_+}^{-b_+}\langle\omega\rangle_{U_-}^{-b_-}
  =
  \frac{\sin(2\pi/n)}{\sqrt{n}}
  \times e^{\sign(p)\left(\frac{3}{n}+\frac{n+1}{4}\right)\pi\i}
\end{equation*}
and the required formula follows.
\end{proof}
\begin{remark}
Equation~(4.1) in \cite{Chen/Yang:QT2018} is not correct; they should have used $\langle\omega\rangle_{U_-}$ and $\langle\omega\rangle_{U}$ instead of $\langle\mu\omega\rangle_{U_-}$ and $\langle\mu\omega\rangle_{U}$.
\end{remark}
\begin{remark}
This is the so-called the quantum $SU(2)$-invariant.
See \cite[P.~930]{Gang/Romo/Yamazaki:COMMP2018}.
See also \cite{Kirby/Melvin:INVEM1991} and \cite{Lickorish:1997}.
\end{remark}
\par
Let $K:=T(a,b)$ be the torus knot of type $(a,b)$.
Our convention is that $T(2,3)$ is the right-handed trefoil.
Note that the knot $3_1$ in Rolfsen's table \cite{Rolfsen:1990} (see also \cite{Lickorish:1997}) is $T(2,-3)$.
Then we have
\begin{equation*}
\begin{split}
  J_k(K;q)
  &=
  \frac{q^{1/2}-q^{-1/2}}{q^{k/2}-q^{-k/2}}\,q^{-ab(k^2-1)/4}
  \sum_{j=-(k-1)/2}^{(k-1)/2}
  q^{bj(aj+1)}
  \frac{q^{aj+1/2}-q^{-(aj+1/2)}}{q^{1/2}-q^{-1/2}}
  \\
  &=
  q^{-ab(k^2-1)/4}
  \sum_{j=-(k-1)/2}^{(k-1)/2}
  q^{bj(aj+1)}
  \frac{q^{aj+1/2}-q^{-(aj+1/2)}}{q^{k/2}-q^{-k/2}}
\end{split}
\end{equation*}
from \cite{Morton:MATPC1995,Rosso/Jones:JKNOT1993}.
Hence we have
\begin{equation*}
  J_{k}(K; e^{\frac{4\pi\sqrt{-1}}{n}})
  =
  e^{-ab(k^2-1)\frac{\pi\sqrt{-1}}{n}}
  \sum_{j=-(k-1)/2}^{(k-1)/2}
  e^{4bj(aj+1)\frac{\pi \sqrt{-1}}{n}}
  \frac{\sin\left(\frac{(4aj+2)\pi}{n}\right)}{\sin\left(\frac{2k\pi}{n}\right)}.
\end{equation*}
With $l=2j$ we have
\begin{equation*}
\begin{split}
  \htau_n\left(X_p;\exp(4\pi\i/n)\right)
  &=
  \frac{1}{\sqrt{n}\sin(2\pi/n)}\,
  e^{\sign(p)(\frac{3}{n}+\frac{n+1}{4})\pi\sqrt{-1}}\,
  \sum_{k=1}^{n-1}
  \sin\left(\frac{2k\pi}{n}\right)
  (-1)^{p(k^2-1)}e^{(p-ab)(k^2-1)\frac{\pi\sqrt{-1}}{n}}
  \\
  &
  \quad\times
  \sum_{\substack{-k+1\le l\le k-1\\l\equiv{k+1}\pmod{2}}}
  e^{bl(al+2)\frac{\pi\sqrt{-1}}{n}}\sin\left(\frac{2(al+1)\pi}{n}\right).
\end{split}
\end{equation*}
Putting
\begin{equation*}
\begin{split}
  S
  &:=
  \sum_{k=1}^{n-1}
  (-1)^{pk}e^{(p-ab)k^2 h}\sinh(2kh)
  \sum_{\substack{-k+1\le l\le k-1\\l\equiv{k+1}\pmod{2}}}
  e^{bl(al+2)h}\sinh(2(al+1)h)
  \\
  &=
  e^{-\frac{b}{a}h}\sum_{k=1}^{n-1}
  (-1)^{pk}e^{(p-ab)k^2 h}\sinh(2kh)
  \sum_{\substack{-k+1\le l\le k-1\\l\equiv{k+1}\pmod{2}}}
  e^{\frac{b}{a}(al+1)^2h}\sinh(2(al+1)h)
\end{split}
\end{equation*}
with $h:=\frac{\pi\sqrt{-1}}{n}$, we have
\begin{equation*}
  \htau_n(X_p;\exp(4\pi\i/n))
  =
  -\frac{1}{\sqrt{n}\sin(2\pi/n)}\,
  e^{\sign(p)(\frac{3}{n}+\frac{n+1}{4})\pi\sqrt{-1}}\,
  (-1)^p e^{-(p-ab)\frac{\pi\sqrt{-1}}{n}}\times S.
\end{equation*}
From now on we consider the case where $p>ab>0$.
\par
We use the following formula:
\begin{equation}\label{eq:integral_phi}
  e^{\lambda w^2}
  =
  \frac{1}{\sqrt{\pi\lambda}}\displaystyle\int_{C_\theta}e^{-\frac{z^2}{\lambda}-2zw}\,dz,
\end{equation}
where  $C_{\theta}$ is the path $\{te^{\theta\sqrt{-1}}\mid t\in\R\}$.
Here $\theta$ satisfies $-\pi/2 + \arg\lambda<2\theta<\pi/2 + \arg\lambda$ so that the integral converges.
Applying \eqref{eq:integral_phi} with $\lambda=bh/a$ and $w=al+1$, we have
\begin{multline*}
  \sum_{\substack{-k+1\le l\le k-1\\l\equiv{k+1}\pmod{2}}}
  e^{\frac{b}{a}(al+1)^2h}
  \sinh(2(al+1)h)
  \\
  =
  \sqrt{\frac{a}{bh\pi}}
  \int_{C_\theta}e^{-\frac{a}{bh}z^2}
  \left(
    \sum_{\substack{-k+1\le l\le k-1\\l\equiv{k+1}\pmod{2}}}
    e^{-2(al+1)z} \sinh(2(al+1)h)
  \right)
  \,dz.
\end{multline*}
Since $\arg(\frac{b}{a}h)=\pi/2$, $\theta$ should satisfy $0<\theta<\pi$.
\par
We first calculate the summation.
\begin{lemma}\label{sum_l}
We have
\begin{multline*}
  \sum_{\substack{-k+1\le l\le k-1\\l\equiv{k+1}\pmod{2}}}
  e^{-2(al+1)z}\sinh(2(al+1)h)
  \\
  =
  \frac{1}{2}e^{2(h-z)}\frac{\sinh(2ka(z-h))}{\sinh(2a(z-h))}
  -
  \frac{1}{2}e^{-2(h+z)}\frac{\sinh(2ka(z+h))}{\sinh(2a(z+h))}.
\end{multline*}
\end{lemma}
\begin{proof}
Since the range of summation $\{l\in\Z\mid-k+1\le l\le k-1,l\equiv k+1\pmod{2}\}$ is invariant under $l\leftrightarrow-l$, for any discrete function $f$ we have
\begin{equation*}
  \sum_{\substack{-k+1\le l\le k-1\\l\equiv{k+1}\pmod{2}}}f(l)
  =
  \sum_{\substack{-k+1\le l\le k-1\\l\equiv{k+1}\pmod{2}}}\frac{f(l)+f(-l)}{2}.
\end{equation*}
\par
Put $R(z):=\sum_{\substack{-k+1\le l\le k-1\\l\equiv{k+1}\pmod{2}}}e^{-2(al+1)z}\sinh(2(al+1)h)$.
We have 
\begin{equation*}
\begin{split}
  2R(z)
  &=
  \sum_{\substack{-k+1\le l\le k-1\\l\equiv{k+1}\pmod{2}}}
  e^{-2(al+1)z}\sinh(2(al+1)h)
  -
  \sum_{\substack{-k+1\le l\le k-1\\l\equiv{k+1}\pmod{2}}}
  e^{2(al-1)z}\sinh(2(al-1)h)
  \\
  &=
  \sum_{\substack{-k+1\le l\le k-1\\l\equiv{k+1}\pmod{2}}}
  e^{-2(al+1)z}(\sinh(2alh)\cosh(2h)+\cosh(2alh)\sinh(2h))
  \\
  &\quad
  -
  \sum_{\substack{-k+1\le l\le k-1\\l\equiv{k+1}\pmod{2}}}
  e^{2(al-1)z}(\sinh(2alh)\cosh(2h)-\cosh(2alh)\sinh(2h))
  \\        
  &=
  -(e^{2h}+e^{-2h})e^{-2z}
  \sum_{\substack{-k+1\le l\le k-1\\l\equiv{k+1}\pmod{2}}}\sinh(2alz)\sinh(2alh)
  \\
  &\quad
  +(e^{2h} - e^{-2h}) e^{-2z}
  \sum_{\substack{-k+1\le l\le k-1\\l\equiv{k+1}\pmod{2}}}
  \cosh(2alz)\cosh(2alh)
  \\
  &=
  e^{2(h-z)}
  \sum_{\substack{-k+1\le l\le k-1\\l\equiv{k+1}\pmod{2}}}\cosh(2al(z-h))
  -
  e^{-2(h+z)}
  \sum_{\substack{-k+1\le l\le k-1\\l\equiv{k+1}\pmod{2}}}\cosh(2al(z+h)).
\end{split}
\end{equation*}
Noting that
\begin{equation*}
  \sum_{\substack{-k+1\le l\le k-1\\l\equiv{k+1}\pmod{2}}}
  \cosh(xl)
  =
  \sum_{\substack{-k+1\le l\le k-1\\l\equiv{k+1}\pmod{2}}}
  \frac{\sinh((l+1)x)-\sinh((l-1)x)}{2\sinh(x)}
  =
  \frac{\sinh(kx)}{\sinh{x}},
\end{equation*}
the lemma then follows.
\end{proof}
We calculate the integral.
\begin{proposition}\label{intsum_l}
We have
\begin{equation*}
\begin{split}
  &
  \int_{C_\theta}
  e^{-\frac{a}{bh}z^2}
  \left(
    \sum_{\substack{-k+1\le l\le k-1\\l\equiv{k+1}\pmod{2}}}e^{-2(al+1)z}\sinh(2(al+1)h)
  \right)
  \,dz
  \\
  =&
  e^{-\frac{ah}{b}}
  \int_{C_\theta}
  e^{-\frac{a}{bh}z^2}
  \frac{\sinh(\frac{2a}{b}z)\sinh(2z)\sinh(2akz)}{\sinh(2az)}\,dz.
\end{split}
\end{equation*}
\end{proposition}
\begin{proof}
Since $C_\theta\leftrightarrow -C_\theta$ under $z\leftrightarrow -z$, for any function $f$ we have $$\int_{C_\theta} f(z) dz  = \int_{C_\theta} f(-z) dz.$$. 
\par
By Lemma \ref{sum_l} the left-hand side becomes
\begin{equation*}
\begin{split}
  &
  \frac{1}{2}
  \int_{C_\theta}
  e^{-\frac{a}{bh}z^2}e^{2(h-z)}
  \frac{\sinh(2ka(z-h))}{\sinh(2a(z-h))}\,dz
  -
  \frac{1}{2}
  \int_{C_\theta}
  e^{-\frac{a}{bh}z^2}e^{-2(h+z)}
  \frac{\sinh(2ka(z+h))}{\sinh(2a(z+h))}\,dz
  \\
  =&
  \frac{1}{2}
  \int_{C_\theta}
  e^{-\frac{a}{bh}z^2}
  e^{2(h-z)} \frac{\sinh(2ka(z-h))}{\sinh(2a(z-h))}\,dz
  -
  \frac{1}{2}
  \int_{C_\theta}
  e^{-\frac{a}{bh}z^2}
  e^{-2(h-z)}\frac{\sinh(2ka(z-h))}{\sinh(2a(z-h))}\,dz
  \\
  =&
  \int_{C_\theta}
  e^{-\frac{a}{bh}z^2}
  \sinh(2(h-z))\frac{\sinh(2ka(z-h))}{\sinh(2a(z-h))}\,dz.
\end{split}
\end{equation*}
By shifting the path of integral from $C_{\theta}$ to $C_{\theta}+h$, since $\frac{\sinh(2kaz)}{\sinh(2az)}$ is analytic, and applying change of variable $z \mapsto z+h$ we get
\begin{equation*}
\begin{split}
  &
  -\int_{C_\theta}
  e^{-\frac{a}{bh}(z+h)^2}
   \sinh(2z)\frac{\sinh(2kaz)}{\sinh(2az)}\,dz
  \\
  =&
  -\frac{1}{2}
  \int_{C_\theta}e^{-\frac{a}{bh}(z+h)^2}
  \sinh (2z)\frac{\sinh(2kaz)}{\sinh(2az)}\,dz
  +
  \frac{1}{2}
  \int_{C_\theta}
  e^{-\frac{a}{bh}(z-h)^2}
  \sinh(2z)\frac{\sinh(2kaz)}{\sinh(2az)}\,dz
  \\
  =&
  e^{-\frac{ah}{b}}
  \int_{C_\theta}
  e^{-\frac{a}{bh}z^2}
  \frac{\sinh(\frac{2a}{b}z)\sinh(2z)\sinh(2akz)}{\sinh(2az)}\,dz
\end{split}
\end{equation*}
as required.
\end{proof}
By Proposition \ref{intsum_l} we have
\begin{equation*}
\begin{split}
  &
  \sum_{\substack{-k+1\le l\le k-1\\l\equiv{k+1}\pmod{2}}}
  e^{\frac{b}{a}(al+1)^2h}\sinh(2(al+1)h)
  \\
  =&
  \sqrt{\frac{a}{bh\pi}}
  e^{-\frac{ah}{b}}
  \int_{C_\theta}
  e^{-\frac{a}{bh}z^2}
  \frac{\sinh(\frac{2a}{b}z)\sinh(2z)\sinh(2akz)}{\sinh(2az)}\,dz
  \\
  &{(az\mapsto z)}
  \\
  =&
  \frac{1}{\sqrt{abh\pi}}
  e^{-\frac{ah}{b}}
  \int_{C_\theta}
  e^{-\frac{z^2}{abh}} \ \frac{\sinh (\frac{2z}{b}) \sinh(\frac{2z}{a}) \sinh (2kz)}{\sinh (2z)}\,dz
  \\
  =&
  \frac{1}{\sqrt{abh\pi}}
  e^{-\frac{ah}{b}}
  \int_{C_\theta}
  e^{-\frac{z^2}{abh}}\varphi(z)\sinh(2kz)\,dz,
\end{split}
\end{equation*}
where we put $\varphi(z):=\frac{\sinh(\frac{2z}{b})\sinh(\frac{2z}{a})}{\sinh(2z)}$.
Then we have
\begin{equation*}
  S
  =
  \frac{1}{\sqrt{abh\pi}}
  e^{-(a/b+b/a)h}
  \sum_{k=1}^{n-1}
  (-1)^{pk}e^{(p-ab)k^2h}\sinh(2kh)
  \int_{C_\theta}e^{-\frac{z^2}{abh}}\varphi(z)\sinh(2kz)\,dz.
\end{equation*}
Hence we have
\begin{equation*}
  \htau_n(X_p;\exp(4\pi\i/n))
  =
  -\frac{1}{\sqrt{n}\sin(2\pi/n)}
  e^{(\frac{3}{n}+\frac{n+1}{4})\pi\sqrt{-1}}
  (-1)^p
  e^{-(p-ab)\frac{\pi\sqrt{-1}}{n}}\frac{1}{\sqrt{abh\pi}}
  e^{-(a/b+b/a)h}\,T,
\end{equation*}
where
\begin{equation*}
  T
  :=
  \sum_{k=1}^{n-1}
  (-1)^{pk}e^{(p-ab)k^2 h}\sinh(2kh)
  \int_{C_\theta}e^{-\frac{z^2}{abh}}\varphi(z)\sinh(2kz)\,dz.
\end{equation*}
\par
Applying \eqref{eq:integral_phi} with $\lambda=(p-ab)h$ and $w=k$, we have
\begin{equation*}
\begin{split}
  T
  &=
  \frac{1}{\sqrt{(p-ab)h\pi}}
  \\
  &\quad\times
  \sum_{k=1}^{n-1}
  (-1)^{pk}
  \left(
    \int_{C_\theta}e^{-\frac{z_1^2}{(p-ab)h}-2kz_1}\,dz_1
  \right)
  \sinh(2kh)
  \int_{C_\theta}e^{-\frac{z_2^2}{abh}}\varphi(z_2)\sinh(2kz_2)\,dz_2
\end{split}
\end{equation*}
Note that $\arg((p-ab)h)=\pi/2$ and $\theta$ should satisfy $0<\theta<\pi$ as before.
Applying $\int_{C_\theta}f(z)dz =\int_{C_\theta}\frac{f(z)+f(-z)}{2}\,dz$, we have
\begin{equation*}
\begin{split}
  T
  &=
  \frac{1}{\sqrt{(p-ab)h\pi}}
  \\
  &\quad\times
  \iint_{C_{\theta}\times C_{\theta}}
  e^{-\frac{z_1^2}{(p-ab)h}}
  e^{-\frac{z_2^2}{abh}}
  \varphi(z_2)
  \sum_{k=1}^{n-1}
  (-1)^{pk}\cosh (2kz_1)\sinh (2kh)\sinh (2kz_2)
  \,dz_1\,dz_2.
\end{split}
\end{equation*}
Hence we have
\begin{equation*}
\begin{split}
  \htau_n(X_p;\exp(4\pi\i/n))
  &=
  -\frac{1}{\sqrt{n}\sin(2\pi/n)}
  e^{(\frac{3}{n}+\frac{n+1}{4})\pi\sqrt{-1}}
  (-1)^p e^{-(p-ab)\frac{\pi\sqrt{-1}}{n}}
  \frac{1}{\sqrt{abh\pi}}
  e^{-(a/b+b/a)h} \, \frac{1}{\sqrt{(p-ab)h\pi}}U
  \\
  &=
  \frac{\sqrt{n}}{\sin(2\pi/n)}
  \frac{\sqrt{-1}}{\pi^2}
  e^{(\frac{n+1}{4})\pi\sqrt{-1}}
  (-1)^p e^{-(p-ab+a/b+b/a - 3 )\frac{\pi\sqrt{-1}}{n}}\frac{1}{\sqrt{ab(p-ab)}}U,
\end{split}
\end{equation*}
where
\begin{equation*}
\begin{split}
  U
  &:=
  \iint_{C_{\theta}\times C_{\theta}}
  e^{-\frac{z_1^2}{(p-ab)h}}e^{-\frac{z_2^2}{abh}}\varphi(z_2)
  \sum_{k=1}^{n-1}
  (-1)^{pk}\cosh(2kz_1)\sinh(2kh)\sinh(2kz_2)
  \,dz_1\,dz_2
  \\
  &=
  \iint_{C_{\theta}\times C_{\theta}}
  e^{-\frac{z_1^2}{(p-ab)h}}e^{-\frac{z_2^2}{abh}}\varphi(z_2)
  \sum_{k=1}^{n-1}
  \cosh (2kz_1)\sinh (2k\tilde{h})\sinh (2kz_2)
  \,dz_1\,dz_2,
\end{split}
\end{equation*}
where $\tilde{h}:=h-p\pi\sqrt{-1}/2=(1/n-p/2)\pi\sqrt{-1}$.
\par
Note that the double integral $\iint_{C_{\theta}\times C_{\theta}}g(z_1, z_2)\,dz_1\,dz_2$ is unchanged if the integrand $g(z_1,z_2)$ is replaced with $g(-z_1,z_2)$, $g(z_1,-z_2)$ or $g(-z_1,-z_2)$.
Note also that $\varphi(z_2)$ is an odd function.
\par
Since 
\begin{align*}
  \cosh (2kz_1) \sinh (2kz_2)
  &=
  \frac{\sinh(2k(z_1+z_2))+\sinh (2k(-z_1+z_2))}{2},
  \\
  \sinh(2k(z_1+z_2))\sinh (2k\tilde{h})
  &=
  \frac{\cosh (2k(z_1+z_2+\tilde{h}))-\cosh (2k(-z_1-z_2+\tilde{h}))}{2},
\end{align*}
we have 
\begin{equation*}
\begin{split}
  U
  &=
  \iint_{C_{\theta}\times C_{\theta}}
  e^{-\frac{z_1^2}{(p-ab)h}}
  e^{-\frac{z_2^2}{abh}}\varphi(z_2)
  \sum_{k=1}^{n-1}\cosh (2k(z_1+z_2+\tilde{h}))
  \,dz_1\,dz_2
  \\
  &=
  \frac{1}{2}
  \iint_{C_{\theta}\times C_{\theta}}
  e^{-\frac{z_1^2}{(p-ab)h}}
  e^{-\frac{z_2^2}{abh}}
  \varphi(z_2)
  \left\{ \frac{\sinh ((2n-1)(z_1+z_2+\tilde{h}))}{\sinh (z_1+z_2+\tilde{h})} -1 \right\}
  \,dz_1\,dz_2
  \\
  &=
  \frac{1}{2}
  \iint_{C_{\theta}\times C_{\theta}}
  e^{-\frac{z_1^2}{(p-ab)h}}
  e^{-\frac{z_2^2}{abh}}
  \varphi(z_2)\frac{\sinh ((2n-1)(z_1+z_2+\tilde{h}))}{\sinh (z_1+z_2+\tilde{h})}
  \,dz_1\,dz_2.
\end{split}
\end{equation*}
Here we use the formula $\sum_{k=1}^{n-1}\cosh(2kx)=\frac{1}{2}\left(\frac{\sinh((2n-1)x)}{\sinh(x)}-1\right)$ in the second equality, and use the fact that $\varphi(z_2)$ is an odd function in the third.
\par
Since 
\begin{equation*}
\begin{split}
  \frac{\sinh ((2n-1)(x+\tilde{h}))}{\sinh (x+\tilde{h})}
  &=
  \frac{\sinh(2n(x+\tilde{h}))\cosh(x+\tilde{h})}{\sinh (x+\tilde{h})}
  -
  \cosh(2n(x+\tilde{h}))
  \\
  &=
  (-1)^{np}
  \left\{
    \sinh(2nx)\coth(x+\tilde{h})
    -
    \cosh(2nx)
  \right\}
\end{split}
\end{equation*}
and $\varphi(z_2)$ is an odd function, we have
\begin{equation*}
\begin{split}
  U
  &=
  \frac{(-1)^{np}}{2}
  \iint_{C_{\theta}\times C_{\theta}}
  e^{-\frac{z_1^2}{(p-ab)h}}
  e^{-\frac{z_2^2}{abh}}
  \varphi(z_2)
  \sinh(2n(z_1+z_2))\coth(z_1+z_2+\tilde{h})
  \,dz_1\,dz_2
  \\
  &=
  \frac{(-1)^{np}}{2}
  \iint_{C_{\theta}\times C_{\theta}}
  e^{-\frac{(z_1-z_2)^2}{(p-ab)h}}
  e^{-\frac{z_2^2}{abh}}\varphi(z_2)
  \sinh(2nz_1)\coth(z_1+\tilde{h})
  \,dz_1\,dz_2.
\end{split}
\end{equation*}
By using the symmetry of the integral, we have
\begin{equation*}
\begin{split}
  &
  \iint_{C_{\theta}\times C_{\theta}}
  e^{-\frac{(z_1-z_2)^2}{(p-ab)h}}
  e^{-\frac{z_2^2}{abh}}\varphi(z_2)
  \sinh(2nz_1)\coth(z_1+\tilde{h})
  \,dz_1\,dz_2
  \\
  =&
  -
  \iint_{C_{\theta}\times C_{\theta}}
  e^{-\frac{(z_1-z_2)^2}{(p-ab)h}}
  e^{-\frac{z_2^2}{abh}}\varphi(z_2)
  \sinh(2nz_1)\coth(-z_1+\tilde{h})
  \,dz_1\,dz_2
  \\
  =&
  \iint_{C_{\theta}\times C_{\theta}}
  e^{-\frac{(z_1-z_2)^2}{(p-ab)h}}
  e^{-\frac{z_2^2}{abh}}\varphi(z_2)
  \sinh(2nz_1)\coth(z_1-\tilde{h})
  \,dz_1\,dz_2.
\end{split}
\end{equation*}
Therefore we have
\begin{equation*}
\begin{split}
  U
  &=
  \frac{(-1)^{np}}{4}
  \iint_{C_{\theta}\times C_{\theta}}
  e^{-\frac{(z_1-z_2)^2}{(p-ab)h}}
  e^{-\frac{z_2^2}{abh}}\varphi(z_2)
  \sinh (2nz_1)
  (\coth(z_1+\tilde{h})-\coth(z_1-\tilde{h}))
  \,dz_1\,dz_2
  \\
  &=
  \frac{(-1)^{np}}{4}
  \iint_{C_{\theta}\times C_{\theta}}
  e^{-\frac{(z_1-z_2)^2}{(p-ab)h}}
  e^{-\frac{z_2^2}{abh}}\varphi(z_2)e^{2n z_1}
  (\coth(z_1+\tilde{h})-\coth(z_1-\tilde{h}))
  \,dz_1\,dz_2
\end{split}
\end{equation*}
where we use the symmetry $(z_1,z_2)\leftrightarrow(-z_1,-z_2)$ of the double integral again.
Hence we have
\begin{equation}\label{eq:tau_V1_V2}
  \htau_n(X_p;\exp(4\pi\i/n))
  =
  \frac{\sqrt{n}}{\sin(2\pi/n)}
  \frac{(-1)^{np+p}e^{\frac{3}{4}\pi\sqrt{-1}}}{4\pi^2\sqrt{ab(p-ab)}}
  e^{-(p-ab+a/b+b/a-3)
  \frac{\pi\sqrt{-1}}{n}}e^{\frac{n}{4}\pi\sqrt{-1}}(V_1-V_2)
\end{equation}
where
\begin{align*}
  V_1
  &=
  \iint_{C_{\theta}\times C_{\theta}}
  \psi_{1}(z_1)\varphi(z_2)
  e^{nF(z_1,z_2)}
  \,dz_1\,dz_2,
  \\
  V_2
  &=
  \iint_{C_{\theta}\times C_{\theta}}
  \psi_{2}(z_1)\varphi(z_2)
  e^{nF(z_1,z_2)}
  \,dz_1\,dz_2
\end{align*}
with
\begin{align*}
  \psi_{1}(z_1)
  &:=
  \coth(z_1+\tilde{h}),
  \\
  \psi_{2}(z_1)
  &:=
  \coth(z_1-\tilde{h}'),
  \\
  F(z_1,z_2)
  &:=
  -\frac{(z_1-z_2)^2}{(p-ab)\pi\sqrt{-1}}-\frac{z_2^2}{ab\pi\sqrt{-1}}+2z_1.
\end{align*}
Here we put $\tilde{h}'=\left(\frac{1}{n}+\frac{p}{2}\right)\pi\sqrt{-1}$, noting that $\coth(z_1-\tilde{h})=\coth(z_1-\tilde{h}')$.
We note the following:
\begin{itemize}
\item
$F(z_1,z_2)$ has a unique critical point $(w_1,w_2):=(p\pi\sqrt{-1},ab\pi\sqrt{-1})$,
\item
$\psi_{1}(z_1)$ has poles $\xi_{1,l}=l\pi\sqrt{-1}-\tilde{h}=(l+p/2-1/n)\pi\sqrt{-1}$ with $l \in\Z$.
Moreover, the residue of $\psi_{1}$ at $\xi_{1,l}$ is given by $\Res(\psi_{1};\xi_{1,l})=1$,
\item
$\psi_{2}(z_1)$ has poles $\xi_{2,l}:=l\pi\sqrt{-1}+\tilde{h}'=(l+p/2+1/n)\pi\sqrt{-1}$ with $l \in\Z$.
Moreover, the residue of $\psi_{2}$ at $\xi_{2,l}$ is given by $\Res(\psi_{2};\xi_{2,l}) = 1$,
\item
$\varphi(z_2)$ has poles $\eta_m=\frac{m}{2}\pi\sqrt{-1}$, where $m\in\Z$ such that $a\nmid m$, and $b\nmid m$.
Moreover, the residue is given by $\Res(\varphi;\eta_m)=-\frac{(-1)^m}{2}\sin(m\pi/a)\sin(m\pi/b)$.
\end{itemize}

\section{Asymptotic expansion of $V_1$}
In this section we consider the asymptotic behavior of $V_1$ as $n\to\infty$.
\par
In the double integral in the definition of $V_1$, we shift $C_{\theta}\times C_{\theta}$ to $(C_{\theta}+w_1)\times(C_{\theta}+w_2)$ by using the residue theorem.
Then we have
\begin{equation*}
  V_1
  =
  I_{1,0}+2\pi\sqrt{-1}I_{1,1}+2\pi\sqrt{-1}I_{1,2}+(2\pi\sqrt{-1})^2I_{1,3},
\end{equation*}
where we put
\begin{align*}
  I_{1,0}
  &:=
  \iint_{(C_{\theta}+w_1)\times(C_{\theta}+w_2)}
  \psi_{1}(z_1)\varphi(z_2)e^{nF(z_1,z_2)}
  \,dz_1\,dz_2
  \\
  I_{1,1}
  &:=
  \sum_{-p/2+1/n<l<p/2+1/n}
  \Res(\psi_{1};\xi_{1,l})
  \int_{C_{\theta}+w_2}\varphi(z_2)e^{nF(\xi_{1,l},z_2)}
  \,dz_2
  \\
  I_{1,2}
  &=
  \sum_{m=1}^{2ab-1}
  \Res(\varphi;\eta_m)
  \int_{C_{\theta}+w_1}
  \psi_{1}(z_1)e^{nF(z_1,\eta_m)}
  \,dz_1
  \\
  I_{1,3}
  &=
  \sum_{\substack{-p/2+1/n<l<p/2+1/n\\1\le m\le2ab-1}}
  \Res(\psi_{1};\xi_{1,l})\Res(\varphi;\eta_m)e^{nF(\xi_{1,l},\eta_m)}.
\end{align*}
Here $l$ runs over all integers such that $l+p/2-1/n$ is between $0$ and $p$, i.e., $|l-1/n|<p/2$, and $m$ runs over all integers such that $0<m<2ab$,
\begin{remark}
If $p$ is odd, then $\frac{1}{2}(1-p)\le l\le\frac{1}{2}(p-1)$.
If $p$ is even, then $-\frac{1}{2}p+1\le l\le\frac{1}{2}p.$
\end{remark}
\subsection{The double integral in $V_1$}
In this subsection we calculate the double integral $I_{1,0}$.
We have
\begin{equation*}
\begin{split}
  I_{1,0}
  &=
  \iint_{(C_{\theta}+w_1)\times(C_{\theta}+w_2)}
  \psi_{1}(z_1)\varphi(z_2)e^{nF(z_1,z_2)}
  \,dz_1\,dz_2
  \\
  &=
  \iint_{C_{\theta}\times C_{\theta}}
  \psi_{1}(z_1+w_1)\varphi(z_2+w_2)e^{nF(z_1+w_1,z_2+w_2)}
  \,dz_1\,dz_2.
\end{split}
\end{equation*}
Note that $\psi_{1}(z_1+w_1)=\psi_{1}(z_1)$ and $\varphi(z_2+w_2)=\varphi(z_2)$.
Moreover, we have
\begin{equation*}
  F(z_1+w_1,z_2+w_2)
  =
  -\frac{(z_1-z_2)^2}{(p-ab)\pi\sqrt{-1}}
  -\frac{z_2^2}{ab\pi\sqrt{-1}}
  +p\pi\sqrt{-1}.
\end{equation*}
Hence we have
\begin{equation*}
  I_{1,0}
  =
  \iint_{C_{\theta}\times C_{\theta}}
  \psi_{1}(z_1)\varphi(z_2)
  \exp
  \left[
    n
    \left(
      -\frac{(z_1-z_2)^2}{(p-ab)\pi\sqrt{-1}}
      -\frac{z_2^2}{ab\pi\sqrt{-1}}
      +p\pi\sqrt{-1}
    \right)
  \right]
  \,dz_1\,dz_2.
\end{equation*}
Later, we will show this will cancel the counterpart of $V_2$, and so we leave it as it is.
\subsection{The double sum in $V_1$}
Since $\xi_{1,l}=(l+p/2-1/n)\pi\sqrt{-1}$ and $\eta_m=m\pi\sqrt{-1}/2$, we have 
\begin{equation*}
  F(\xi_{1,l},\eta_m)
  =
  \left(
    -\frac{(l-m/2+p/2-1/n)^2}{p-ab}-\frac{m^2}{4ab}+2(l+p/2-1/n)
  \right) 
  \pi\sqrt{-1}.
\end{equation*}
Hence we have
\begin{multline*}
  nF(\xi_{1,l},\eta_m)
  \\
  =
  \left[
    \left(
      -\frac{(l-m/2+p/2)^2}{p-ab}-\frac{m^2}{4ab}+2l+p
    \right)n
    +
    \frac{2l-m+p}{p-ab}-2-\frac{1}{(p-ab)n}
  \right]
  \pi\sqrt{-1}.
\end{multline*}
Since $\Res(\psi_{1};\xi_{1,l})\Res(\varphi;\eta_m)=-\frac{(-1)^{m}}{2}\sin(m\pi/a)\sin(m\pi/b)$ we obtain
\begin{equation}\label{eq:I_3n}
\begin{split}
  I_{1,3}
  &=
  \sum_{\substack{-p/2+1/n<l<p/2+1/n\\1\le m\le2ab-1}}
  \frac{(-1)^{m+1}}{2}\sin\left(\frac{m\pi}{a}\right)\sin\left(\frac{m\pi}{b}\right)
  e^{\frac{2l-m + p}{p-ab}\pi \sqrt{-1}}\,e^{-\frac{1}{p-ab} \frac{\pi \sqrt{-1}}{n} }
  \\
  &\quad
  (-1)^{pn} \exp \left[ n \left( - \frac{(l - m/2 + p/2  )^2}{p-ab} - \frac{m^2}{4ab} \right) \pi \sqrt{-1} \right].
\end{split}
\end{equation}
\subsection{The sum over $l$ in $V_1$}
Consider the integral
\begin{equation*}
  \int_{C_{\theta_2+w_2}}\varphi(z_2)e^{nF(\xi_{1,l},z_2)}\,dz_2.
\end{equation*}
The polynomial $F(\xi_{1,l},z_2)=-\frac{(\xi_{1,l}-z_2)^2}{(p-ab)\pi\sqrt{-1}}-\frac{z_2^2}{ab\pi\sqrt{-1}}+2 \xi_{1,l}$ has a unique critical point
\begin{equation*}
  \frac{ab}{p}\xi_{1,l}
  =
  ab\left(\frac{l}{p}+\frac{1}{2}\right)\pi\sqrt{-1}-\frac{ab}{pn}\pi\sqrt{-1}.
\end{equation*}
Put
\begin{align*}
  \alpha&:=ab\left(\frac{l}{p}+\frac{1}{2}\right)\pi\sqrt{-1}
  \\
  \beta&:=\frac{ab}{p}\pi\sqrt{-1}
\end{align*}
so that the critical point becomes $\alpha-\beta/n$.
Then we can write
\begin{equation}\label{eq:F_1}
  F(\xi_{1,l},z_2)
  =
  F\left(\xi_{1,l},\alpha-\beta/n\right)
  -
  \frac{p}{ab(p-ab)\pi\sqrt{-1}}(z_2-\alpha+\beta/n)^2.
\end{equation}
Note that $\alpha-\beta/n$ is not a pole of $\varphi(z_2)$.
\par
We will shift the line of integration from $C_{\theta}+w_2$ to $C_{\theta}+\alpha-\beta/n$.
Since $0<l+p/2-1/n<p$, we see that $\alpha-\beta/n$ is below $w_2=ab\pi\sqrt{-1}$ in the complex plane.
So by the residue theorem we have
\begin{equation*}
\begin{split}
  &\int_{C_{\theta}+w_2}
  \varphi(z_2)e^{nF(\xi_{1,l},z_2)}\,dz_2
  \\
  =&
  \int_{C_{\theta}+\alpha-\beta/n}
  \varphi(z_2)
  e^{nF(\xi_{1,l},z_2)}
  -
  2\pi\sqrt{-1}
  \sum_{\frac{2ab}{p}(l+p/2-1/n)<m'<2ab}
  \Res(\varphi;\eta_{m'})e^{nF(\xi_{1,l},\eta_{m'})}.
\end{split}
\end{equation*}
Note that from \eqref{eq:F_1} we have
\begin{equation*}
\begin{split}
  &\int_{C_{\theta}+\alpha-\beta/n}\varphi(z_2)e^{nF(\xi_{1,l}, z_2)}\,dz_2
  \\
  =&
  e^{nF(\xi_{1,l},\alpha-\beta/n)}
  \int_{C_{\theta}-\beta_1/n}
  \varphi(z_2+\alpha)
  e^{\frac{-pn}{ab(p-ab)\pi\sqrt{-1}}(z_2+\beta/n)^2}\,dz_2
  \\
  =&
  e^{nF(\xi_{1,l},\alpha-\beta/n)}
  e^{\frac{ab}{p(p-ab)n\pi\sqrt{-1}}}
  \\
  &\times
  \int_{C_{\theta}-\beta/n}
  \varphi(z_2+\alpha)
  e^{-\frac{2}{p-ab}z_2}
  e^{\frac{-pn}{ab(p-ab)\pi\sqrt{-1}}z_2^2}\,dz_2
  \\
  =&
  e^{nF(\xi_{1,l},\alpha-\beta/n)}
  e^{\frac{ab\pi\sqrt{-1}}{p(p-ab)n}}
  \\
  &\times
  \int_{C_{\theta}}
  \varphi(z_2+\alpha)
  e^{-\frac{2}{p-ab}z_2}
  e^{\frac{-pn}{ab(p-ab)\pi\sqrt{-1}}z_2^2}\,dz_2,
\end{split}
\end{equation*}
where we change the variable $z_2\mapsto z_2+\alpha$ at the first equality, and use the fact that $\varphi(z_2+\alpha)$ has no poles between $C_{\theta}$ and $C_{\theta}-\beta/n$ for sufficiently large $n$.
This can be shown as follows:
Suppose for a contradiction that there is a pole between $C_{\theta}$ and $C_{\theta}-\beta/n$.
It is of the form $k\pi\sqrt{-1}/2-ab(l/p+1/2)\pi\sqrt{-1}$ with $k\in\Z$ not divisible by $a$ or $b$.
So it should satisfy the inequality $-ab/(pn)\le k/2-ab(l/p+1/2)\le0$.
If $n$ is sufficiently large, then this implies that $2abl/p=k-ab$.
Since $p$ and $ab$ are coprime, this means that $2l/p$ is an integer.
However since $-p/2\le l\le p/2$ we see that $l=-p/2$, $0$, or $p/2$.
Then $k$ equals $0$, $ab$, or $2ab$ respectively, which is divisible by $a$ and $b$, a contradiction.
\par
By the saddle point method, we have
\begin{equation*}
  \int_{C_{\theta}}
  \varphi(z_2+\alpha)
  e^{-\frac{2}{p-ab}z_2}
  e^{\frac{-pn}{ab(p-ab)\pi\sqrt{-1}}z_2^2}\,dz_2
  =
  \sqrt{\frac{ab(p-ab)\pi^2\sqrt{-1}}{pn}}\varphi(\alpha)+O(n^{-3/2}).
\end{equation*}
Hence we have
\begin{multline*}
  \int_{C_{\theta_2+\alpha-\beta/n}}
  \varphi(z_2)e^{nF(\xi_{1,l},z_2)}\,dz_2
  \\
  =
  e^{nF(\xi_{1,l},\alpha-\beta/n)}
  e^{\frac{ab\pi\sqrt{-1}}{p(p-ab)n}}
  \sqrt{\frac{ab(p-ab)\pi^2\sqrt{-1}}{pn}}\varphi(\alpha)+O(n^{-3/2}),
\end{multline*}
where 
\begin{equation*}
  \varphi(\alpha) 
  =
  (-1)^{ab+a+b}\sqrt{-1}
  \frac{\sin\left(\frac{2al\pi}{p}\right)\sin\left(\frac{2bl\pi}{l}\right)}
       {\sin\left(\frac{2abl}{p}\right)}.
\end{equation*}
\par
Since $\xi_{1,l}=(l+p/2-1/n)\pi\sqrt{-1}$ and $\alpha-\beta/n=\frac{ab}{p}\xi_{1,l}$, we have
\begin{equation*}
\begin{split}
  F(\xi_{1,l},\alpha-\beta/n)
  &=
  -\frac{(\xi_{1,l}-\frac{ab}{p}\xi_{1,l})^2}{(p-ab)\pi\sqrt{-1}}
  -\frac{(\frac{ab}{p}\xi_{1,l})^2}{ab\pi\sqrt{-1}}+2\xi_{1,l}
  \\
  &=
  \left(
    -\frac{p-ab}{p^2\pi\sqrt{-1}}-\frac{ab}{p^2\pi\sqrt{-1}}
  \right)
  \xi_{1,l}^2
  +
  2\xi_{1,l}
  \\
  &=
  -\frac{1}{p\pi\sqrt{-1}}\xi_{1,l}^2+2\xi_{1,l}
  \\
  &=
  \left(-\frac{(l+p/2-1/n)^2}{p}+2(l+p/2-1/n)\right)\pi\sqrt{-1}
  \\
  &=
  \left(-\frac{(l+p/2)^2}{p}+2(l+p/2)+(\frac{l+p/2}{p}-1)\frac{2}{n}-\frac{1}{pn^2}\right)
  \pi\sqrt{-1}.
\end{split}
\end{equation*}
Hence we have
\begin{equation*}
\begin{split}
  &
  \int_{C_{\theta+\alpha-\beta/n}}\varphi(z_2)e^{nF(\xi_{1,l},z_2)}\,dz_2
  \\
  =&
  \sqrt{\frac{ab(p-ab)\pi^2\sqrt{-1}}{pn}}
  \sqrt{-1}(-1)^{ab+a+b}
  \frac{\sin\left(\frac{2al\pi}{p}\right)\sin\left(\frac{2bl\pi}{p}\right)}
       {\sin\left(\frac{2abl\pi}{p}\right)}
  e^{-(\frac{ab}{p(p-ab)}+\frac{1}{p})\frac{\pi\sqrt{-1}}{n}}
  \\
  &\times
  (-1)^{np-1}
  e^{\frac{2l}{p}\pi\sqrt{-1}}\exp\left[ n \left(-\frac{(l+p/2)^2}{p}\right)\pi\sqrt{-1}\right]
  +O(n^{-3/2}).
\end{split}
\end{equation*}
\par
We finally have
\begin{equation}\label{eq:I_1n}
\begin{split}
  I_{1,1}
  &=
  \sum_{-p/2+1/n<l<p/2+1/n}
  \Res(\psi_{1};\xi_{1,l})
  \int_{C_{\theta}+w_2}\varphi(z_2)e^{nF(\xi_{1,l},z_2)}\,dz_2
  \\
  &=
  -2\pi\sqrt{-1}
  \sum_{\substack{-p/2+1/n<l<p/2+1/n\\ \frac{2ab}{p}(l+p/2-1/n)<m'<2ab}}
  \Res(\psi_{1};\xi_{1,l})\Res(\varphi;\eta_{m'})e^{nF(\xi_{1,l}),\eta_{m'}}
  \\
  &\quad+
  \sum_{-p/2+1/n<l<p/2+1/n}
  \Res(\psi_{1};\xi_{1,l})\int_{C_{\theta}+\alpha-\beta/n}
  \varphi(z_2)e^{nF(\xi_{1,l},z_2)}\,dz_2
  \\
  &=
  -2\pi\sqrt{-1}
  \sum_{\substack{-p/2+1/n<l<p/2+1/n\\ \frac{2ab}{p}(l+p/2-1/n)<m'<2ab}}
  \Res(\psi_{1};\xi_{1,l})\Res(\varphi;\eta_{m'})e^{nF(\xi_{1,l}),\eta_{m'}}
  \\
  &\quad+
  \sum_{-p/2+1/n<l<p/2+1/n}
  \sqrt{\frac{ab(p-ab)\pi^2\sqrt{-1}}{pn}}\sqrt{-1}(-1)^{ab+a+b}
  \frac{\sin\left(\frac{2al\pi}{p}\right)\sin\left(\frac{2bl\pi}{p}\right)}
       {\sin\left(\frac{2abl\pi}{p}\right)}
  e^{-\frac{1}{p-ab}\frac{\pi\sqrt{-1}}{n}}
  \\
  &\quad\times
  (-1)^{np-1}e^{\frac{2l}{p}\pi\sqrt{-1}}
  \exp\left[n\left(-\frac{(l+p/2)^2}{p}\right)\pi\sqrt{-1}\right]
  +O(n^{-3/2}).
\end{split}
\end{equation}
\subsection{The sum over $m$ in $V_1$}
In this subsection we study the asymptotic behavior of $I_{1,2}$.
\par
Consider the integral
\begin{equation*}
  \int_{C_{\theta}+w_1}\psi_{1}(z_1)e^{nF(z_1,\eta_m)}\,dz_1.
\end{equation*}
The polynomial $F(z_1,\eta_m)=-\frac{(z_1-\eta_m)^2}{(p-ab)\pi\sqrt{-1}}-\frac{\eta_m^2}{ab\pi\sqrt{-1}}+2z_1$ has a unique critical point
\begin{equation*}
  \gamma
  :=
  (m/2+p-ab)\pi\sqrt{-1}
\end{equation*}
and we can write
\begin{equation*}
  F(z_1,\eta_m)
  =
  F(\gamma,\eta_m)-\frac{1}{(p-ab)\pi\sqrt{-1}}(z_1-\gamma)^2.
\end{equation*}
Note that $\gamma$ is not a pole of $\psi_{1}(z_1)$.
Since $\gamma$ is below $w_1=p\pi\sqrt{-1}$ in the complex plane, we have
\begin{equation*}
\begin{split}
  &\int_{C_{\theta_1+w_1}}\psi_{1}(z_1)e^{nF(z_1,\eta_m)}\,dz_1
  \\
  =&
  \int_{C_{\theta}+\gamma}\psi_{1}(z_1)e^{nF(z_1,\eta_m)}\,dz_1
  -
  2\pi\sqrt{-1}
  \sum_{m/2+p/2-ab+1/n<l'<p/2+1/n}
  \Res(\psi_{1};\xi_{1,l'})e^{nF(\xi_{1,l'},\eta_m)}.
\end{split}
\end{equation*}
Note that 
\begin{equation*}
\begin{split}
  \int_{C_{\theta}+\gamma}\psi_{1}(z_1)e^{nF(z_1,\eta_m)}\,dz_1
  &=
  e^{nF(\gamma,\eta_m)}
  \int_{C_{\theta}+\gamma}
  \psi_{1}(z_1)e^{-\frac{n(z_1-\gamma)^2}{(p-ab)\pi\sqrt{-1}}}\,dz_1
  \\
  &=
  e^{nF(\gamma,\eta_m)}
  \int_{C_{\theta}}\psi_{1}(z_1+\gamma)e^{-\frac{nz_1^2}{(p-ab)\pi \sqrt{-1}}}\,dz_1.
\end{split}
\end{equation*}
Since $\gamma=(m/2+p-ab)\pi\sqrt{-1}$ and $\eta_m=m\pi\sqrt{-1}/2$ we have
\begin{equation*}
  F(\gamma,\eta_m)
  =
  \left(p-ab-\frac{m^2}{4ab}+m\right)\pi\sqrt{-1}.
\end{equation*}
\par
Note that $\psi_{1}(z_1 + \gamma)=\coth(z_1+\gamma+\tilde{h})$.
Put $\delta:=(m/2+p/2-ab)\pi\sqrt{-1}$.
Then $\psi_{1}(z_1+\gamma) = \coth(z_1+\delta+h).$ 
\par
The integral above becomes
\begin{equation*}
\begin{split}
  &
  \int_{C_{\theta}}\coth(z_1+\delta+h)e^{-\frac{nz_1^2}{(p-ab)\pi\sqrt{-1}}}\,dz_1
  \\
  =&
  \int_{C_{\theta}+h}\coth(z_1+\delta)e^{-\frac{n(z_1-h)^2}{(p-ab)\pi\sqrt{-1}}}\,dz_1
  \\
  =&
  e^{-\frac{\pi\sqrt{-1}}{(p-ab)n}}
  \int_{C_{\theta}+h}
  e^{\frac{2z_1}{p-ab}}\coth(z_1+\delta)
  e^{-\frac{nz_1^2}{(p-ab)\pi\sqrt{-1}}}\,dz_1.
\end{split}
\end{equation*}
There are two cases to consider.
\begin{itemize}
\item[Case 1:]
If $m \not\equiv p \pmod{2}$, then $\delta\in(\Z+1/2)\pi\sqrt{-1}$ and so $\coth(z_1+\delta)=\coth(z_1+\pi\sqrt{-1}/2)$.
Hence we have
\begin{equation*}
\begin{split}
  &
  \int_{C_{\theta}+h}
  e^{\frac{2z_1}{p-ab}}
  \coth(z_1+\delta)e^{-\frac{nz_1^2}{(p-ab)\pi\sqrt{-1}}}\,dz_1
  \\
  =&
  \int_{C_{\theta_1} }
  e^{\frac{2z_1}{p-ab}}\coth(z_1+\pi\sqrt{-1}/2 )e^{-\frac{nz_1^2}{(p-ab)\pi\sqrt{-1}}}\,dz_1
  \\
  =&
  \sqrt{\frac{(p-ab)\pi^2\sqrt{-1}}{n}}\coth(\pi\sqrt{-1}/2)+O(n^{-3/2}) 
  \\
  =&
  O(n^{-3/2})
\end{split}
\end{equation*}
since there are no poles of $\coth(z_1+\pi\sqrt{-1}/2)$ between $C_{\theta}$ and $C_{\theta}+h$.
\item[Case 2:]
If $m = p \pmod{2}$, then $\delta\in\Z\pi\sqrt{-1}$ and $\coth(z_1+\delta)=\coth(z_1)$.
Hence we have
\begin{equation*}
\begin{split}
  &
  \int_{C_{\theta}+h}
  e^{\frac{2z_1}{p-ab}}\coth(z_1+\delta)e^{-\frac{nz_1^2}{(p-ab)\pi\sqrt{-1}}}\,dz_1
  \\
  =&
  \int_{C_{\theta}+h}
  \left(e^{\frac{2z_1}{p-ab}}\coth(z_1)-\frac{1}{z_1}\right)
  e^{-\frac{nz_1^2}{(p-ab)\pi\sqrt{-1}}}\,dz_1
  +
  \int_{C_{\theta}+h}\frac{1}{z_1}e^{-\frac{nz_1^2}{(p-ab)\pi\sqrt{-1}}}\,dz_1
  \\
  =&
  \int_{C_{\theta}}
  \left(e^{\frac{2z_1}{p-ab}}\coth(z_1)-\frac{1}{z_1}\right)
  e^{-\frac{nz_1^2}{(p-ab)\pi\sqrt{-1}}}\,dz_1
  +
  \int_{C_{\theta}+h}\frac{1}{z_1}e^{-\frac{nz_1^2}{(p-ab)\pi\sqrt{-1}}}\,dz_1
  \\
  =&
  \sqrt{\frac{(p-ab)\pi^2\sqrt{-1}}{n}}\frac{2}{p-ab}+ O(n^{-3/2})
  +
  \int_{C_{\theta}+h}\frac{1}{z_1}e^{-\frac{nz_1^2}{(p-ab)\pi\sqrt{-1}}}\,dz_1
\end{split}
\end{equation*}
since there are no poles of $e^{\frac{2z_1}{p-ab}}\coth(z_1)-\frac{1}{z_1}$ between $C_{\theta}$ and $C_{\theta_1}+h$.
Note that $\left(e^{\frac{2z_1}{p-ab}}\coth(z_1)-\frac{1}{z_1}\right)\Bigr|_{z_1=0}=\frac{2}{p-ab}$.
\par
Replacing $z_1+h$ with $-z_1+h$, we have
\begin{equation*}
  \int_{C_{\theta}+h}\frac{1}{z_1}e^{-\frac{nz_1^2}{(p-ab)\pi\sqrt{-1}}}\,dz_1
  =
  -\int_{C_{\theta}-h}\frac{1}{z_1}e^{-\frac{nz_1^2}{(p-ab)\pi\sqrt{-1}}}\,dz_1.
\end{equation*}
Hence we have
\begin{equation*}
\begin{split}
  \int_{C_{\theta}+h}
  \frac{1}{z_1}e^{-\frac{nz_1^2}{(p-ab)\pi\sqrt{-1}}}\,dz_1
  &=
  \frac{1}{2}
  \left(
    \int_{C_{\theta}+\beta/n}\frac{1}{z_1}e^{-\frac{nz_1^2}{(p-ab)\pi\sqrt{-1}}}\,dz_1
    -
    \int_{C_{\theta}-\beta/n}\frac{1}{z_1}e^{-\frac{nz_1^2}{(p-ab)\pi\sqrt{-1}}}\,dz_1
  \right)
  \\
  &=
  -\pi\sqrt{-1}
  \Res\left(\frac{1}{z_1}e^{-\frac{nz_1^2}{(p-ab)\pi\sqrt{-1}}};0\right)
  =
  -\pi\sqrt{-1}.
\end{split}
\end{equation*}
So in this case we have
\begin{equation*}
  \int_{C_{\theta}+h}
  e^{\frac{2z_1}{p-ab}}\coth(z_1+\delta)e^{-\frac{nz_1^2}{(p-ab)\pi\sqrt{-1}}}\,dz_1
  =
  -\pi\sqrt{-1}+\frac{2}{p-ab}\sqrt{\frac{(p-ab)\pi^2\sqrt{-1}}{n}}
  +O(n^{-3/2})
\end{equation*}
and 
\begin{equation*}
\begin{split}
  &
  \int_{C_{\theta}+\gamma}
  \psi_{1}(z_1)e^{nF(z_1,\eta_m)}\,dz_1
  \\
  &=
  e^{nF(\zeta,\eta_m)}e^{-\frac{\pi\sqrt{-1}}{(p-ab)n}}
  \int_{C_{\theta}+h}e^{\frac{2z_1}{p-ab}}\coth(z_1+\delta)e^{-\frac{nz_1^2}{(p-ab)\pi\sqrt{-1}}}\,dz_1
  \\
  &=
  \left(-\pi\sqrt{-1}+2\sqrt{\frac{\pi^2\sqrt{-1}}{(p-ab)n}}\right)
  e^{-\frac{\pi\sqrt{-1}}{(p-ab)n}}
  \exp\left[n\left(p-ab-\frac{m^2}{4ab}+m\right)\pi\sqrt{-1}\right]
  +O(n^{-3/2}).
\end{split}
\end{equation*}
\end{itemize}
\par
Hence we have
\begin{equation}\label{eq:I_2n}
\begin{split}
  I_{1,2}
  &=
  \sum_{m=1}^{2ab-1}
  \Res(\varphi;\eta_m)
  \int_{C_{\theta}+w_1}\psi_{1}(z_1)e^{nF(z_1,\eta_m)}\,dz_1
  \\
  &=
  -
  2\pi\sqrt{-1}
  \sum_{\substack{m/2+p/2-ab+1/n<l'<p/2+1/n\\1\le m\le2ab-1}}
  \Res(\varphi;\eta_m)\Res(\psi_{1};\xi_{1,l'})e^{nF(\xi_{1,l'})\eta_m}
  \\&\quad
  +
  \sum_{m=1}^{2ab-1}
  \Res(\varphi;\eta_m)
  \int_{C_{\theta}+\gamma}\psi_{1}(z_1)e^{nF(z_1,\eta_m)}\,dz_1
  \\
  &=
  -
  2\pi\sqrt{-1}
  \sum_{\substack{m/2+p/2-ab+1/n<l'<p/2+1/n\\1\le m\le2ab-1}}
  \Res(\varphi;\eta_m)\Res(\psi_{1};\xi_{1,l'})e^{nF(\xi_{1,l'})\eta_m}
  \\
  &\quad
  +
  \sum_{\substack{1\le m\le2ab-1\\\text{$m-p$: even}}}
  \frac{(-1)^{m+1}}{2}\sin\left(\frac{m\pi}{a}\right)\sin\left(\frac{m\pi}{b}\right)
  \\
  &\quad
  \times
  \left(-\pi\sqrt{-1}+2\sqrt{\frac{\pi^2\sqrt{-1}}{(p-ab)n}}\right)
  e^{-\frac{\pi\sqrt{-1}}{(p-ab)n}}
  \exp\left[n\left(-ab-\frac{m^2}{4ab}\right)\pi\sqrt{-1}\right]
  +O(n^{-3/2}),
\end{split}
\end{equation}
where $l'$ is an integer such that $m/2+p-ab<l'+p/2-1/n<p$.
\subsection{Summary}
Now we combine the results so far.
\par
Note that $I_{1,3}$ is the sum of $\Res(\psi_{1};\xi_{1,l})\Res(\varphi;\eta_m)e^{nF(\xi_{1,l},\eta_m)}$ over $(l,m)$.
The same summand appears both in \eqref{eq:I_1n} and in \eqref{eq:I_2n}.
Taking cancellation into account, the range of summation turns out to be
\begin{equation*}
\begin{split}
  &
  \mathcal{S}_1
  \\
  :=&
  \left\{(l,m)\in\Z^2\bigm|
  -\frac{p}{2}<l-\frac{1}{n}<\frac{p}{2},
  0<m<2ab,
  a\nmid m, b\nmid m\right\}
  \\
  &\setminus
  \left\{
    (l,m)\in\Z^2\bigm|
    -\frac{p}{2}<l-\frac{1}{n}<\frac{p}{2},\frac{2ab}{p}(l+\frac{p}{2}-\frac{1}{n})<m<2ab
  \right\}
  \\
  &\setminus
  \left\{
    (l,m)\in\Z^2\Bigm|\frac{m}{2}+p-ab<l+\frac{p}{2}-\frac{1}{n}<p,0<m<2ab,
  \right\}
  \\
  =&
  \left\{
    (l,m)\in\Z^2\bigm|
    -\frac{p}{2}<l-\frac{1}{n}<\frac{p}{2},m\le\frac{2ab}{p}(l-\frac{1}{n})+ab,
    l-\frac{1}{n}\le\frac{m}{2}+\frac{p}{2}-ab,
    0<m<2ab,
    a\nmid m, b\nmid m
  \right\}
  \\
  =&
  \left\{
    (l,m)\in\Z^2\bigm|
    -\frac{p}{2}<l\le\frac{p}{2},
    m<\frac{2abl}{p}+ab,
    l\le\frac{m}{2}+\frac{p}{2}-ab,
    0<m<2ab,
    a\nmid m, b\nmid m
  \right\}.
\end{split}
\end{equation*}
for large $n$.
Note that the range $\mathcal{S}_1$ is the integer points in the interior of the triangle with vertices $(p/2,2ab)$, $(-p/2,0)$, and $(p/2-ab,0)$ on the $lm$-plane.
See Figure~\ref{fig:S_1}.
\begin{figure}[h]
  \raisebox{-0.5\height}{\includegraphics[scale=0.4]{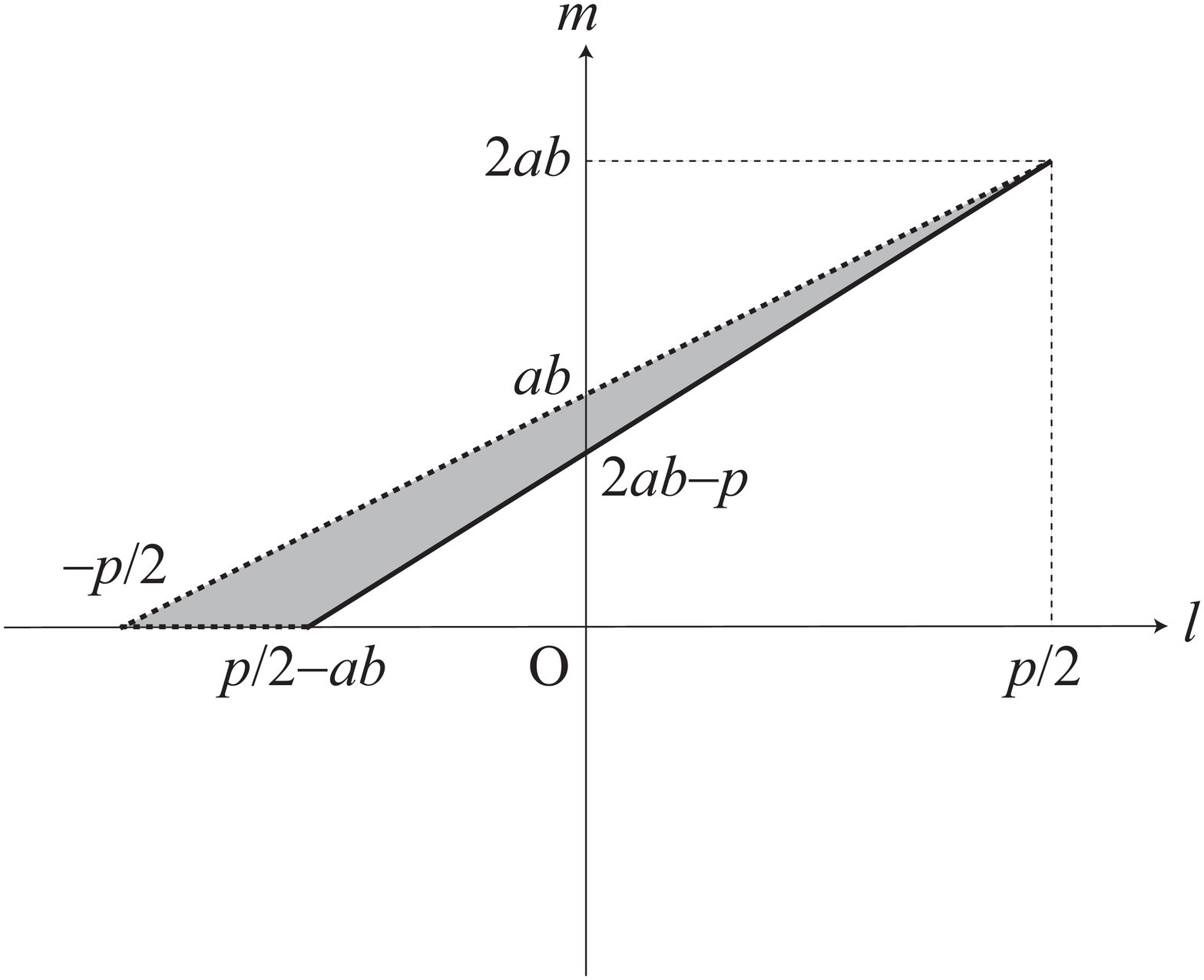}}
  \caption{The gray area indicates $\mathcal{S}_1$. The dotted lines are not included.}
  \label{fig:S_1}
\end{figure}
\par
So we have
\begin{equation*}
\begin{split}
  &V_1
  \\
  =&
  I_{1,0}
  \\
  &+
  \sum_{(l,m)\in\mathcal{S}_1}
  2\pi^2(-1)^{m}
  \sin\left(\frac{m\pi}{a}\right)\sin\left(\frac{m\pi}{b}\right)
  e^{\frac{2l-m+p}{p-ab}\pi\sqrt{-1}}
  e^{-\frac{1}{p-ab}\frac{\pi\sqrt{-1}}{n}}
  \\
  &\quad
  \times
  (-1)^{pn}\exp\left[n\left(-\frac{(l-m/2+p/2)^2}{p-ab}-\frac{m^2}{4ab} \right)\pi\sqrt{-1}\right]
  \\
  &+
  \sum_{-p/2+1/n<l<p/2+1/n}
  (2\pi\sqrt{-1})\sqrt{\frac{ab(p-ab)\pi^2\sqrt{-1}}{pn}}
  \sqrt{-1}(-1)^{ab+a+b}\
  \frac{\sin\left(\frac{2al\pi}{p}\right)\sin\left(\frac{2bl\pi}{p}\right)}
      {\sin\left(\frac{2abl\pi}{p}\right)}
  e^{-\frac{1}{p-ab}\frac{\pi\sqrt{-1}}{n}}
  \\
  &\quad\times
  (-1)^{np-1}e^{\frac{2l}{p}\pi\sqrt{-1}}
  \exp\left[n\left(-\frac{(l+p/2)^2}{p}\right)\pi\sqrt{-1}\right]
  \\
  &+
  \sum_{\substack{1\le m\le2ab-1\\\text{$m-p$: even}}}
  (-1)(2\pi\sqrt{-1})\frac{(-1)^m}{2}\sin\left(\frac{m\pi}{a}\right)\sin\left(\frac{m\pi}{b}\right)
  \\
  &\quad
  \times
  \left(-\pi\sqrt{-1}+2\sqrt{\frac{\pi^2\sqrt{-1}}{(p-ab)n}}\right)
  e^{-\frac{\pi\sqrt{-1}}{(p-ab)n}}
  \exp\left[n\left(-ab-\frac{m^2}{4ab}\right)\pi\sqrt{-1}\right]
  +O(n^{-3/2}).
\end{split}
\end{equation*}

\section{Asymptotic expansion of $V_2$}
In this section we will study the asymptotic behavior of $V_2$ as $n\to\infty$.
Recall that
\begin{equation*}
  V_2
  =
  \iint_{C_{\theta}\times C_{\theta}}
  \psi_{2}(z_1)\varphi(z_2)e^{nF(z_1,z_2)}
  \,dz_1\,dz_2,
\end{equation*}
where $\psi_{2}(z_1)=\coth(z_1-\tilde{h})$ with $\tilde{h}=(\frac{1}{n}-\frac{p}{2})\pi\sqrt{-1}$.
\par
As in the case of $V_1$, we have
\begin{equation*}
  V_2
  =
  I_{2,0}+2\pi\sqrt{-1}I_{2,1}+2\pi\sqrt{-1}I_{2,2}+(2\pi\sqrt{-1})^2I_{2,3},
\end{equation*}
where
\begin{align*}
  I_{2,0}
  &:=
  \iint_{(C_{\theta}+w_1)\times(C_{\theta}+w_2)}
  \psi_{2}(z_1)\varphi(z_2)e^{nF(z_1,z_2)}\,dz_1\,dz_2
  \\
  I_{2,1}
  &:=
  \sum_{-p/2-1/n<l<p/2-1/n}
  \Res(\psi_{2};\xi_{2,l})\int_{C_{\theta}+w_2}\varphi(z_2)e^{nF(\xi_{2,l},z_2)}\,dz_2
  \\
  I_{2,2}
  &:=
  \sum_{m=1}^{2ab-1}
  \Res(\varphi;\eta_m)\int_{C_{\theta}+w_1}\psi_{1}(z_1)e^{nF(z_1,\eta_m)}\,dz_1
  \\
  I_{2,3}
  &:=
  \sum_{\substack{-p/2-1/n<l<p/2-1/n\\1\le m\le2ab-1}}
  \Res(\psi_{2};\xi_{2,l})\Res(\varphi;\eta_m)e^{nF(\xi_{2,l},\eta_m)}.
\end{align*}
\subsection{The double integral in $V_2$}
In this subsection we calculate the double integral $I_{2,0}$.
\par
By changing variables, we have
\begin{equation*}
  I_{2,0}
  =
  \iint_{C_{\theta}\times C_{\theta}}
  \psi_{2}(z_1+w_1)\varphi(z_2+w_2)e^{nF(z_1+w_1,z_2+w_2)}\,dz_1\,dz_2.
\end{equation*}
Note that $\psi_{2}(z_1+w_1)=\psi_{2}(z_1)$ and $\varphi(z_2+w_2)=\varphi(z_2)$.
Moreover, we have
\begin{equation*}
  F(z_1+w_1,z_2+w_2)
  =
  -\frac{(z_1-z_2)^2}{(p-ab)\pi\sqrt{-1}}-\frac{z_2^2}{ab\pi\sqrt{-1}}+p\pi\sqrt{-1}.
\end{equation*}
Hence we have
\begin{equation*}
  I_{2,0}
  =
  \iint_{C_{\theta}\times C_{\theta}}
  \psi_{2}(z_1)\varphi(z_2)
  \exp
  \left[
     n
     \left(
       -\frac{(z_1-z_2)^2}{(p-ab)\pi\sqrt{-1}}-\frac{z_2^2}{ab\pi\sqrt{-1}}+p\pi\sqrt{-1}
     \right)
   \right]\,dz_1\,dz_2.
\end{equation*}
Now we change the variable $(z_1,z_2)\mapsto(-z_1,-z_2)$.
Since $\varphi(-z_2)$ and $\psi_{2}(-z_1)=-\psi_{2}(z_1)$, we conclude that $I_{1,0}=I_{2,0}$.
\par
So we see that $V_1-V_2=2\pi\sqrt{-1}(I_{1,1}-I_{2,2}+I_{1,2}-I_{2,2})+(2\pi\sqrt{-1})^2(I_{1,3}-I_{2,3})$.
\subsection{The double sum in $V_2$}
In this subsection we calculate the double summation in $V_2$.
\par
Since $\xi_{2,l} =(l+p/2+1/n) \pi \sqrt{-1}$ and $\eta_m=m\pi \sqrt{-1} /2$, we have 
\begin{equation*}
  F(\xi_{2,l},\eta_m) 
  =
  \left(
    -\frac{(l-m/2+p/2+1/n)^2}{p-ab}-\frac{m^2}{4ab}+2(l+p/2+1/n)
  \right)
  \pi\sqrt{-1}.
\end{equation*}
Hence we have
\begin{multline*}
  nF(\xi_{2,l},\eta_m)
  \\
  =
  \left[
    \left(
      -\frac{(l-m/2+p/2)^2}{p-ab}-\frac{m^2}{4ab}+2l+p
    \right)
    n
    -\frac{2l-m+p}{p-ab}+2-\frac{1}{(p-ab)n}
  \right]
  \pi\sqrt{-1}.
\end{multline*}
Since $\Res(\psi_{2};\xi_{2,l})\Res(\varphi;\eta_m)=-\frac{(-1)^{m}}{2}\sin(m\pi/a)\sin(m\pi/b)$, we obtain 
\begin{equation*}
\begin{split}
  I_{2,3}
  &=
  \sum_{\substack{-p/2-1/n<l<p/2-1/n\\1\le m\le2ab-1}}
  \Res(\psi_{2};\xi_{2,l})\Res(\varphi;\eta_m)e^{nF(\xi_{2,l},\eta_m)}
  \\
  &=
  \sum_{\substack{-p/2-1/n<l<p/2-1/n\\1\le m\le2ab-1}}
  \frac{(-1)^{m+1}}{2}
  \sin\left(\frac{m\pi}{a}\right)\sin\left(\frac{m\pi}{b}\right)
  e^{-\frac{2l-m+p}{p-ab}\pi\sqrt{-1}}
  e^{-\frac{1}{p-ab}\frac{\pi \sqrt{-1}}{n}}
  \\
  &
  (-1)^{pn}
  \exp
  \left[
    n\left(-\frac{(l-m/2+p/2)^2}{p-ab}-\frac{m^2}{4ab}\right)\pi\sqrt{-1}
  \right].
\end{split}
\end{equation*}
\subsection{The sum over $l$ in $V_2$}
In this subsection we calculate $I_{2,1}$.
\par
Consider the integral
\begin{equation*}
  \int_{C_{\theta}+w_2}
  \varphi(z_2)e^{nF(\xi_{2,l},z_2)}\,dz_2.
\end{equation*}
The polynomial $F(\xi_{2,l},z_2)=-\frac{(\xi_{2,l}- z_2)^2}{(p-ab)\pi\sqrt{-1}}-\frac{z_2^2}{ab\pi\sqrt{-1}}+2\xi_{2,l}$ has a unique critical point
\begin{equation*}
  \frac{ab}{p}\xi_{2,l}
  =
  \alpha+\beta/n
\end{equation*}
and
\begin{equation}\label{eq:F_2}
  F(\xi_{2,l}, z_2)
  =
  F(\xi_{2,l},\alpha+\beta/n)-\frac{p}{ab(p-ab)\pi\sqrt{-1}}(z_2-\alpha-\beta/n)^2.
\end{equation}
Recall that $\alpha=ab\left(\frac{l}{p}+\frac{1}{2}\right)\pi\sqrt{-1}$ and $\beta=\frac{ab}{p}\pi\sqrt{-1}$.
Note that $\alpha+\beta/n$ is not a pole of $\varphi(z_2)$.
\par
We will shift the line of integration from $C_{\theta}+w_2$ to $C_{\theta}+\alpha+\beta/n$.
Since $0<l+p/2+1/n<p$, we see that $\alpha+\beta/n$ is below $w_2=ab\pi\sqrt{-1}$.
So by the residue theorem we have
\begin{equation}\label{eq:J_1n}
\begin{split}
  &\int_{C_{\theta}+w_2}
  \varphi(z_2)e^{nF(\xi_{2,l},z_2)}\,dz_2
  \\
  =&
  \int_{C_{\theta}+\alpha+\beta/n}
  \varphi(z_2)e^{nF(\xi_{2,l},z_2)}\,dz_2
  -
  2\pi\sqrt{-1}
  \sum_{\frac{2ab}{p}(1+p/2+1/n)<m'<2ab}
  \Res(\varphi;\eta_{m'})e^{nF(\xi_{2,l},\eta_{m'})}\,dz_2.
\end{split}
\end{equation}
Note that from \eqref{eq:F_2} we have
\begin{equation*}
\begin{split}
  &
  \int_{C_{\theta}+\alpha+\beta/n}
  \varphi(z_2)e^{nF(\xi_{2,l},z_2)}\,dz_2
  \\
  =&
  e^{nF(\xi_{2,l},\alpha+\beta/n)}
  \int_{C_{\theta}+\alpha+\beta/n}
  \varphi(z_2) e^{\frac{-pn(z_2-\alpha-\beta/n)^2}{ab(p-ab)\pi\sqrt{-1}}}\,dz_2
  \\
  =&
  e^{nF(\xi_{2,l},\alpha+\beta/n)}
  \int_{C_{\theta}}\varphi(z_2+\alpha+\beta/n)
  e^{\frac{-pnz_2^2}{ab(p-ab)\pi\sqrt{-1}}}\,dz_2
  \\
  =&
  e^{nF(\xi_{2,l},\alpha+\beta/n)}
  \int_{C_{\theta}}
  \varphi(z_2+\alpha+\beta/n)e^{\frac{-pnz_2^2}{ab(p-ab)\pi\sqrt{-1}}}\,dz_2
  \\
  =&
  e^{nF(\xi_{2,l},\alpha+\beta/n)}
  \int_{C_{\theta}+\beta/n}
  \varphi(z_2+\alpha)e^{-\frac{pn(z_2-\beta/n)^2}{ab(p-ab)\pi\sqrt{-1}}}\,dz_2
  \\
  =&
  e^{nF(\xi_{2,l},\alpha+\beta/n)}
  e^{-\frac{ab\pi\sqrt{-1}}{p(p-ab)n}}
  \int_{C_{\theta}+\beta/n}
  e^{\frac{2z_2}{p-ab}}\varphi(z_2+\alpha)e^{-\frac{pnz_2^2}{ab(p-ab)\pi\sqrt{-1}}}\,dz_2
  \\
  =&
  e^{nF(\xi_{2,l},\alpha+\beta/n)}
  e^{-\frac{ab\pi\sqrt{-1}}{p(p-ab)n}}
  \int_{C_{\theta}}
  e^{\frac{2z_2}{p-ab}}\varphi(z_2+\alpha)e^{-\frac{pnz_2^2}{ab(p-ab)\pi\sqrt{-1}}}\,dz_2,
\end{split}
\end{equation*}
since there are no poles of $\varphi(z_2+\alpha)$ between $C_{\theta}+\beta/n$ and $C_{\theta}$.
\par
By the saddle point method, we have
\begin{equation*}
  \int_{C_{\theta}}
  e^{\frac{2z_2}{p-ab}}\varphi(z_2+\alpha) e^{-\frac{pnz_2^2}{ab(p-ab)\pi\sqrt{-1}}}\,dz_2
  =
  \sqrt{\frac{ab(p-ab)\pi^2\sqrt{-1}}{pn}}\varphi(\alpha)+O(n^{-3/2}).
\end{equation*}
Hence we have
\begin{multline*}
  \int_{C_{\theta}+\alpha+\beta/n}\varphi(z_2)e^{nF(\xi_{2,l},z_2)}\,dz_2
  \\
  =
  e^{nF(\xi_{2,l},\alpha+\beta/n)}
  e^{-\frac{ab\pi\sqrt{-1}}{p(p-ab)n}}
  \sqrt{\frac{ab(p-ab)\pi^2\sqrt{-1}}{pn}}\varphi(\alpha)+O(n^{-3/2}),
\end{multline*}
where
\begin{equation*}
  \varphi(\alpha) 
  =
  (-1)^{ab+a+b}\sqrt{-1}
  \frac{\sin\left(\frac{2al\pi}{p}\right)\sin\left(\frac{2bl\pi}{p}\right)}
       {\sin\left(\frac{2abl\pi}{p}\right)}.
\end{equation*}
\par
Since $\xi_{2,l}=(l+p/2+1/n)\pi\sqrt{-1}$ and $\alpha+\beta/n=\frac{ab}{p}\xi_{2,l}$, we have
\begin{equation*}
\begin{split}
  F(\xi_{2,l},\alpha+\beta/n)
  &=
  -\frac{(\xi_{2,l}-\alpha-\beta/n)^2}{(p-ab)\pi\sqrt{-1}}
  -\frac{(\alpha+\beta/n)^2}{ab\pi\sqrt{-1}}+2\xi_{2,l}
  \\
  &=
  -\frac{\xi_{2,l}^2}{p\pi\sqrt{-1}}+2\xi_{2,l}
  \\
  &=
  \left(-\frac{(l+p/2+1/n)^2}{p}+2(l+p/2+1/n)\right)\pi\sqrt{-1}
  \\
  &=
  \left(-\frac{(l+p/2)^2}{p}+2(l+p/2)-(\frac{l+p/2}{p}-1)\frac{2}{n}-\frac{1}{pn^2}\right)
  \pi\sqrt{-1}.
\end{split}
\end{equation*}
Hence we have
\begin{equation*}
\begin{split}
  &
  \int_{C_{\theta}+\alpha+\beta/n}\varphi(z_2)e^{nF(\xi_{2,l},z_2)}\,dz_2
  \\
  =&
  \sqrt{\frac{ab(p-ab)\pi^2\sqrt{-1}}{pn}}
  (-1)^{ab+a+b}\sqrt{-1}
  \frac{\sin\left(\frac{2al\pi}{p}\right)\sin\left(\frac{2bl\pi}{p}\right)}
       {\sin\left(\frac{2abl\pi}{p}\right)}
   e^{-(\frac{ab}{p(p-ab)}+\frac{1}{p})\frac{\pi\sqrt{-1}}{n}}
  \\
  &
  \times(-1)^{np-1}
  e^{-\frac{2l}{p}\pi\sqrt{-1}}
  \exp\left[n\left(-\frac{(l+p/2)^2}{p}\right)\pi\sqrt{-1}\right]
  +O(n^{-3/2}).
\end{split}
\end{equation*}
\par
We finally have
\begin{equation*}
\begin{split}
  &I_{2,1}
  \\
  =&
  \sum_{-p/2-1/n<l<p/2-1/n}
  \Res(\psi_{2};\xi_{2,l})
  \int_{C_{\theta}+w_2}\varphi(z_2)e^{nF(\xi_{2,l},z_2)}\,dz_2
  \\
  =&
  -2\pi\sqrt{-1}
  \sum_{\substack{-p/2-1/n<l<p/2-1/n\\\frac{2ab}{p}(1+p/2+1/n)<m'<2ab}}
  \Res(\psi_{2};\xi_{2,l})\Res(\varphi;\eta_{m'})e^{nF(\xi_{2,l},\eta_{m'})}
  \\
  &+
  \sum_{-p/2-1/n<l<p/2-1/n}
  \Res(\psi_{2};\xi_{2,l})
  \int_{C_{\theta}+\alpha+\beta/n}\varphi(z_2)e^{nF(\xi_{2,l},z_2)}\,dz_2
  \\
  =&
  -2\pi\sqrt{-1}
  \sum_{\substack{-p/2-1/n<l<p/2-1/n\\\frac{2ab}{p}(1+p/2+1/n)<m'<2ab}}
  \Res(\psi_{2};\xi_{2,l})\Res(\varphi;\eta_{m'})e^{nF(\xi_{2,l},\eta_{m'})}
  \\
  &+
  \sum_{-p/2-1/n<l<p/2-1/n}
  \sqrt{\frac{ab(p-ab)\pi^2\sqrt{-1}}{pn}}
  (-1)^{ab+a+b}\sqrt{-1}
  \frac{\sin\left(\frac{2al\pi}{p}\right)\sin\left(\frac{2bl\pi}{p}\right)}
       {\sin\left(\frac{2abl\pi}{p}\right)}
  e^{-\frac{1}{p-ab}\frac{\pi\sqrt{-1}}{n}}
  \\
  &\quad\times
  (-1)^{np-1}e^{-\frac{2l}{p}\pi\sqrt{-1}}
  \exp\left[n\left(-\frac{(l+p/2)^2}{p}\right)\pi\sqrt{-1}\right]
  +O(n^{-3/2})
\end{split}
\end{equation*}
where $\frac{2ab}{p}(l+p/2+1/n)<m'<2ab$.
\subsection{The sum over $m$ in $V_2$}
In this subsection, we study the asymptotic behavior of $I_{2,2}$.
\par
Consider
\begin{equation*}
  \int_{C_{\theta}+w_1}\psi_{2}(z_1)e^{nF(z_1,\eta_m)}\,dz_1.
\end{equation*}
The polynomial $F(z_1,\eta_m)=-\frac{(z_1-\eta_m)^2}{(p-ab)\pi\sqrt{-1}}-\frac{\eta_m^2}{ab\pi\sqrt{-1}}+2z_1$ has a unique critical point
\begin{equation*}
  \gamma:=(m/2+p-ab)\pi\sqrt{-1}
\end{equation*}
and
\begin{equation*}
  F(z_1,\eta_m)
  =
  F(\gamma,\eta_m)-\frac{1}{(p-ab)\pi \sqrt{-1}}(z_1-\gamma)^2.
\end{equation*}
Note that $\gamma$ is not a pole of $\psi_{2}(z_1)$. 
We have
\begin{equation}\label{eq:J_2n}
\begin{split}
  &
  \int_{C_{\theta}+w_1}\psi_{2}(z_1)e^{nF(z_1,\eta_m)}\,dz_1
  \\
  =&
  \int_{C_{\theta}+\gamma}
  \psi_{2}(z_1)e^{nF(z_1,\eta_m)}
  -
  2\pi\sqrt{-1}
  \sum_{m/2+p/2-ab-1/n<l'<p/2-1/n}
  \Res(\psi_{2};\xi_{2,l'})e^{nF(\xi_{2,l'},\eta_m)}\,dz_1.
\end{split}
\end{equation}
Note that 
\begin{equation*}
\begin{split}
  \int_{C_{\theta}+\gamma}\psi_{2}(z_1)e^{nF(z_1,\eta_m)}\,dz_1
  &=
  e^{nF(\gamma,\eta_m)}
  \int_{C_{\theta}+\gamma}
  \psi_{2}(z_1)e^{-\frac{n(z_1-\gamma)^2}{(p-ab)\pi\sqrt{-1}}}\,dz_1
  \\
  &=
  e^{nF(\gamma,\eta_m)}
  \int_{C_{\theta}}\psi_{2}(z_1+\gamma)e^{-\frac{nz_1^2}{(p-ab)\pi\sqrt{-1}}}\,dz_1.
\end{split}
\end{equation*}
Since $\gamma=(m/2+p-ab)\pi\sqrt{-1}$ and $\eta_m=m\pi\sqrt{-1}/2$, we have
\begin{equation*}
  F(\gamma,\eta_m)
  =
  \left(p-ab-\frac{m^2}{4ab}+m\right)\pi\sqrt{-1}.
\end{equation*}
\par
Note that $\psi_{2}(z_1+\zeta)=\coth(z_1+\gamma-\tilde{h}')$.
Put $\delta:=(m/2+p/2-ab)\pi\sqrt{-1}$.
Then $\psi_{2}(z_1+\gamma)=\coth(z_1+\delta-h).$ 
\par
The integral above becomes
\begin{equation*}
\begin{split}
  \int_{C_{\theta}}
  \cot(z_1+\delta-h)e^{\frac{-nz_1^2}{(p-ab)\pi\sqrt{-1}}}\,dz_1
  &=
  \int_{C_{\theta}-h}
  \coth(z_1+\delta)e^{\frac{-n(z_1+h)^2}{(p-ab)\pi\sqrt{-1}}}\,dz_1
  \\
  &=
  e^{-\frac{\pi\sqrt{-1}}{(p-ab)n}}
  \int_{C_{\theta}-h}e^{\frac{-2z_1}{p-ab}}\coth(z_1+\delta) e^{\frac{-nz_1^2}{(p-ab)\pi\sqrt{-1}}}\,dz_1.
\end{split}
\end{equation*}
\par
There are two cases to consider.
\begin{itemize}
\item[Case 1:]
If $m \neq p \pmod{2}$, then $\delta\in(\Z + 1/2)\pi\sqrt{-1}$ and so $\coth(z_1+\delta)=\coth(z_1+\pi\sqrt{-1}/2)$.
Hence we have
\begin{equation*}
\begin{split}
  \int_{C_{\theta}-h}
  e^{\frac{-2z_1}{p-ab}}\coth(z_1+\delta)e^{\frac{-nz_1^2}{(p-ab)\pi\sqrt{-1}}}\,dz_1
  &=
  \int_{C_{\theta}}
  e^{\frac{-2z_1}{p-ab}}\coth(z_1+\pi\sqrt{-1}/2)e^{\frac{-nz_1^2}{(p-ab)\pi\sqrt{-1}}}\,dz_1
  \\
  &=
  \sqrt{\frac{(p-ab)\pi^2\sqrt{-1}}{n}}\coth(\pi\sqrt{-1}/2)+O(n^{-3/2})
  \\
  &=O(n^{-3/2})
\end{split}
\end{equation*}
since there are no poles of $\coth(z_1+\pi\sqrt{-1}/2)$ between $C_{\theta}$ and $C_{\theta_1}-h$.
\item[Case 2:]
If $m = p \pmod{2}$, then $\delta\in\Z\pi\sqrt{-1}$ and $\coth(z_1+\delta)=\coth(z_1)$.
Hence we have
\begin{equation*}
\begin{split}
  &
  \int_{C_{\theta}-h}e^{\frac{-2z_1}{p-ab}}\coth(z_1+\delta)e^{\frac{-nz_1^2}{(p-ab)\pi\sqrt{-1}}}\,dz_1
  \\
  =&
  \int_{C_{\theta}-h}
  \left(e^{\frac{-2z_1}{p-ab}}\coth(z_1)-\frac{1}{z_1}\right)e^{\frac{-nz_1^2}{(p-ab)\pi\sqrt{-1}}}\,dz_1
  +
  \int_{C_{\theta}-h}
  \frac{1}{z_1}e^{\frac{-nz_1^2}{(p-ab)\pi\sqrt{-1}}}\,dz_1
  \\
  =&
  \int_{C_{\theta}}
  \left(e^{\frac{-2z_1}{p-ab}}\coth(z_1)-\frac{1}{z_1}\right)e^{\frac{-nz_1^2}{(p-ab)\pi\sqrt{-1}}}\,dz_1
  +
  \int_{C_{\theta}-h}\frac{1}{z_1}e^{\frac{-nz_1^2}{(p-ab)\pi\sqrt{-1}}}\,dz_1
  \\
  =&
  \sqrt{\frac{(p-ab)\pi^2\sqrt{-1}}{n}}
  \left(-\frac{2}{p-ab}\right)
  +
  O(n^{-3/2})
  +
  \int_{C_{\theta}-h}\frac{1}{z_1}e^{\frac{-nz_1^2}{(p-ab)\pi\sqrt{-1}}}\,dz_1
\end{split}
\end{equation*}
since there are no poles of $e^{\frac{-2z_1}{p-ab}}\coth(z_1)-\frac{1}{z_1}$ between $C_{\theta}$ and $C_{\theta}-h$.
Note that $\left(e^{\frac{-2z_1}{p-ab}}\coth(z_1)-\frac{1}{z_1}\right)\Bigm|_{z_1=0}=\frac{-2}{p-ab}$.
Moreover we have
\begin{equation*}
  \int_{C_{\theta}-h}\frac{1}{z}e^{\frac{-nz^2}{(p-ab)\pi\sqrt{-1}}}\,dz
  =
  -\int_{C_{\theta}+h}\frac{1}{z}e^{\frac{-nz^2}{(p-ab)\pi\sqrt{-1}}}\,dz
  =\pi\sqrt{-1}.
\end{equation*}
Therefore in this case we have
\begin{equation*}
  \int_{C_{\theta}-h}
  e^{\frac{-2z_1}{p-ab}}\cosh(z_1+\delta) e^{\frac{-nz_1^2}{(p-ab)\pi\sqrt{-1}}}\,dz_1
  =
  \pi\sqrt{-1}- \frac{2}{p-ab}\sqrt{\frac{(p-ab)\pi^2\sqrt{-1}}{n}}
  +O(n^{-3/2})
\end{equation*}
and
\begin{equation*}
\begin{split}
  &
  \int_{C_{\theta_1+\zeta}}\psi_{2}(z_1)e^{nF(z_1,\eta_m)}\,dz_1
  \\
  =&
  e^{nF(\gamma,\eta_m)}
  e^{\frac{-\pi\sqrt{-1}}{(p-ab)n}}
  \int_{C_{\theta}-h}e^{\frac{-2z_1}{p-ab}}
  \coth(z_1+\delta) e^{\frac{-nz_1^2}{(p-ab)\pi\sqrt{-1}}}\,dz_1
  \\
  =&
  \left(\pi\sqrt{-1}-2\sqrt{\frac{\pi^2\sqrt{-1}}{(p-ab)n}}\right)
  e^{-\frac{\pi\sqrt{-1}}{(p-ab)n}}
  \exp\left[n\left(p-ab-\frac{m^2}{4ab}+m\right)\pi\sqrt{-1}\right]
  +O(n^{-3/2}).
\end{split}
\end{equation*}
\end{itemize}
Hence we have
\begin{equation*}
\begin{split}
  I_{2,2}
  &=
  \sum_{m=1}^{2ab-1}
  \Res(\varphi;\eta_m)
  \int_{C_{\theta}+w_1}\psi_{2}(z_1)e^{nF(z_1,\eta_m)}\,dz_1
  \\
  &=
  -2\pi\sqrt{-1}
  \sum_{\substack{m/2+p/2-ab-1/n<l'<p/2-1/n\\1\le m\le2ab-1}}
  \Res(\varphi;\eta_m)\Res(\psi_2;\zeta_{2,l'})e^{nF(\zeta_{2,l'},\eta_m)}
  \\
  &\quad+
  \sum_{m=1}^{2ab-1}
  \Res(\varphi;\eta_m)
  \int_{C_{\theta}+\gamma}
  \psi_{2}(z_1)e^{nF(z_1,\eta_m)}\,dz_1
  \\
  &=
  -2\pi\sqrt{-1}
  \sum_{\substack{m/2+p/2-ab-1/n<l'<p/2-1/n\\1\le m\le2ab-1}}
  \Res(\varphi;\eta_m)\Res(\psi_2;\zeta_{2,l'})e^{nF(\zeta_{2,l'},\eta_m)}
  \\
  &\quad+
  \sum_{\substack{1\le m\le2ab-1\\\text{$m-p$: even}}}
  (-1)(2\pi\sqrt{-1})\frac{(-1)^m}{2}
  \sin\left(\frac{m\pi}{a}\right)\sin\left(\frac{m\pi}{b}\right)
  \\ 
  &
  \quad
  \times
  \left(\pi\sqrt{-1}-2\sqrt{\frac{\pi^2\sqrt{-1}}{(p-ab)n}}\right)
  e^{-\frac{\pi\sqrt{-1}}{(p-ab)n}}
  \exp\left[n\left(-ab-\frac{m^2}{4ab}\right)\pi\sqrt{-1}\right]
  +O(n^{-3/2}).
\end{split}
\end{equation*}
\subsection{Summary}
Now we combine the results so far.
\par
Note that $I_{2,3}$ is the sum of $\Res(\psi_{2};\xi_{2,l})\Res(\varphi;\eta_m)e^{nF(\xi_{2,l},\eta_m)}$ over $(l,m)$.
The same summand appears both in \eqref{eq:J_1n} and in \eqref{eq:J_2n}.
Taking cancellation into account, the range of summation turns out to be
\begin{equation*}
\begin{split}
  &
  \mathcal{S}_2
  \\
  :=&
  \left\{(l,m)\in\Z^2\bigm|
  -\frac{p}{2}<l+\frac{1}{n}<\frac{p}{2},
  0<m<2ab,
  a\nmid m, b\nmid m\right\}
  \\
  &\setminus
  \left\{
    (l,m)\in\Z^2\bigm|
    -\frac{p}{2}<l+\frac{1}{n}<\frac{p}{2},
    \frac{2ab}{p}(l+\frac{p}{2}+\frac{1}{n})<m<2ab
  \right\}
  \\
  &\setminus
  \left\{
    (l,m)\in\Z^2\Bigm|\frac{m}{2}+p-ab<l+\frac{p}{2}+\frac{1}{n}<p,0<m<2ab
  \right\}
  \\
  =&
  \left\{
    (l,m)\in\Z^2\bigm|
    -\frac{p}{2}<l+\frac{1}{n}<\frac{p}{2},m\le\frac{2ab}{p}(l+\frac{1}{n})+ab,
    l+\frac{1}{n}\le\frac{m}{2}+\frac{p}{2}-ab,
    0<m<2ab,
    a\nmid m, b\nmid m
  \right\}
  \\
  =&
  \left\{
    (l,m)\in\Z^2\bigm|
    -\frac{p}{2}\le l<\frac{p}{2},
    m\le\frac{2abl}{p}+ab,
    l<\frac{m}{2}+\frac{p}{2}-ab,
    0<m<2ab,
    a\nmid m, b\nmid m
  \right\}.
\end{split}
\end{equation*}
Observe that the line $m=2abl/p+ab$ contains no integer points.
This can be shown as follows:
Suppose that $(l,m)\in\Z^2$ satisfies $m=2abl/p+ab$.
Then we have $pm=ab(2l+p)$ and so $p$ divides $2l$ since $(p,ab)=1$.
So $m$ is a multiple of $ab$, contradicting the assumption that $a\nmid m$ and $b\nmid m$.
Therefore $\mathcal{S}_2$ is indeed of the form
\begin{equation*}
  \mathcal{S}_2
  =
  \left\{
    (l,m)\in\Z^2\bigm|
    -\frac{p}{2}<l<\frac{p}{2},
    m\le\frac{2abl}{p}+ab,
    l<\frac{m}{2}+\frac{p}{2}-ab,
    0<m<2ab,
    a\nmid m, b\nmid m
  \right\}.
\end{equation*}
See Figure~\ref{fig:S_2}.
\begin{figure}[h]
  \raisebox{-0.5\height}{\includegraphics[scale=0.4]{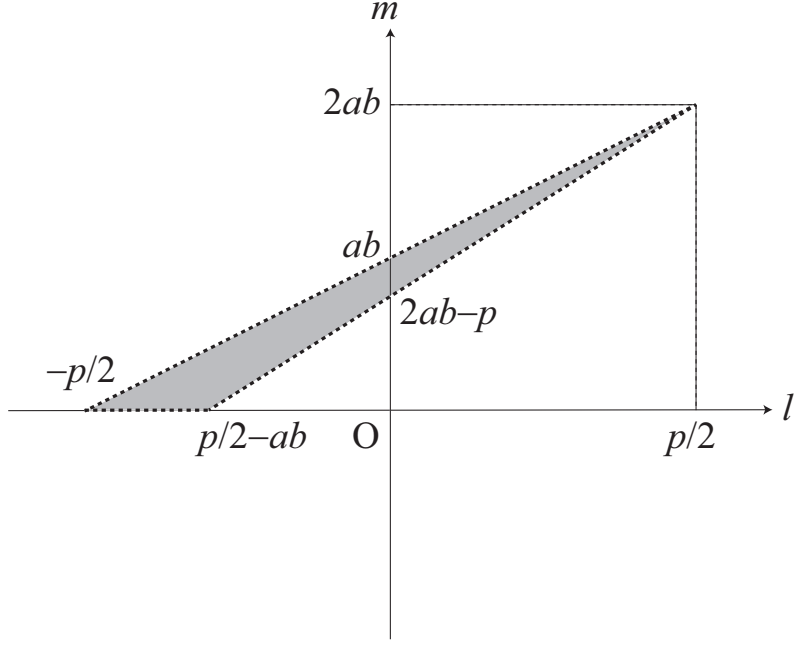}}
  \caption{The gray area indicates $\mathcal{S}_2$. The edges are not included.}
  \label{fig:S_2}
\end{figure}
Note that $\mathcal{S}_1=\mathcal{S}_2\cup\partial{S}_1$, where $\partial{S}_1$ denotes the edge connecting $(p/2,2ab)$ and $(-p/2,0)$, i.e, $\partial{S}_1=\{(l,m)\in\Z^2\mid0<m<2a,a\nmid m,b\nmid m,l=m/2+p/2-ab\}$.
\par
We have
\begin{equation*}
\begin{split}
  V_2
  &=
  I_{2,0}
  +
  \sum_{(l,m)\in\mathcal{S}_2}
  2\pi^2 (-1)^{m}
  \sin\left(\frac{m\pi}{a}\right)\sin\left(\frac{m\pi}{b}\right)
  e^{-\frac{2l-m+p}{p-ab}\pi\sqrt{-1}}
  e^{-\frac{1}{p-ab}\frac{\pi\sqrt{-1}}{n}}
  \\
  &\times
  (-1)^{pn}\exp\left[n\left(-\frac{(l- m/2+p/2)^2}{p-ab}-\frac{m^2}{4ab}\right)\pi\sqrt{-1}\right]
  \\
  &\quad
  +
  \sum_{-p/2-1/n<l<p/n-1/n}
  (2\pi\sqrt{-1})
  \sqrt{\frac{ab(p-ab)\pi^2\sqrt{-1}}{pn}}
  \sqrt{-1}(-1)^{a+b-ab}
  \frac{\sin\left(\frac{2al\pi}{p}\right)\sin\left(\frac{2bl\pi}{p}\right)}
       {\sin\left(\frac{2abl\pi}{p}\right)}
  e^{-\frac{1}{p-ab}\frac{\pi\sqrt{-1}}{n} }
  \\
  &
  (-1)^{np-1}e^{-\frac{2l}{p}\pi\sqrt{-1}}
  \exp\left[n\left(-\frac{(l+p/2)^2}{p}\right)\pi\sqrt{-1}\right]
  \\
  &+
  \quad
  \sum_{\substack{1\le m\le2ab-1\\\text{$m-p$: even}}}
  (-1)(2\pi\sqrt{-1})\frac{(-1)^m}{2}
  \sin\left(\frac{m\pi}{a}\right)\sin\left(\frac{m\pi}{b}\right)
  \\
  &
  \left(\pi\sqrt{-1}-2\sqrt{\frac{\pi^2\sqrt{-1}}{(p-ab)n}}\right)
  e^{-\frac{\pi\sqrt{-1}}{(p-ab)n}}
  \exp\left[n\left(-ab-\frac{m^2}{4ab}\right)\pi\sqrt{-1}\right]
  +O(n^{-3/2}).
\end{split}
\end{equation*}

\section{$V_1$ and $V_2$}\label{sec:V1V2}
First of all note that $I_{1,0} = I_{2,0}$.
Recall that $\mathcal{S}_1=\mathcal{S}_2\cup\partial\mathcal{S}_1$.
Hence we have
\begin{equation*}
\begin{split}
  &
  V_1-V_2
  \\
  =&
  \sum_{(l,m)\in\mathcal{S}_1}
  2\pi^2(-1)^{m}
  \sin\left(\frac{m\pi}{a}\right)\sin\left(\frac{m\pi}{b}\right)
  e^{\frac{2l-m+p}{p-ab}\pi\i}
  e^{-\frac{1}{p-ab}\frac{\pi\i}{n}}
  \\
  &\quad
  \times
  (-1)^{pn}\exp\left[n\left(-\frac{(l-m/2+p/2)^2}{p-ab}-\frac{m^2}{4ab} \right)\pi\i\right]
  \\
  &-
  \sum_{(l,m)\in\mathcal{S}_2}
  2\pi^2 (-1)^{m}
  \sin\left(\frac{m\pi}{a}\right)\sin\left(\frac{m\pi}{b}\right)
  e^{-\frac{2l-m+p}{p-ab}\pi\i}
  e^{-\frac{1}{p-ab}\frac{\pi\i}{n}}
  \\
  &\times
  (-1)^{pn}\exp\left[n\left(-\frac{(l- m/2+p/2)^2}{p-ab}-\frac{m^2}{4ab}\right)\pi\i\right]
  \\
  &+
  2\pi\i
  \sum_{l}
  \sqrt{\frac{ab(p-ab)\pi^2\i}{pn}}
  \i(-1)^{ab+a+b}
  \frac{\sin\left(\frac{2al\pi}{p}\right)\sin\left(\frac{2bl\pi}{p}\right)
        \sin\left(\frac{2l\pi}{p}\right)}
      {\sin\left(\frac{2abl\pi}{p}\right)}
  e^{-\frac{1}{p-ab}\frac{\pi\i}{n}}
  \\
  &\quad\times
  (-1)^{np-1}
  \exp\left[n\left(-\frac{(l+p/2)^2}{p}\right)\pi\i\right]
  \\
  &+
  2\sum_{\text{$m-p$: even}}
  (-1)(2\pi\i)\frac{(-1)^m}{2}\sin\left(\frac{m\pi}{a}\right)\sin\left(\frac{m\pi}{b}\right)
  \\
  &\quad
  \times
  \left(-\pi\i+2\sqrt{\frac{\pi^2\i}{(p-ab)n}}\right)
  e^{-\frac{\pi\i}{(p-ab)n}}
  \exp\left[n\left(-ab-\frac{m^2}{4ab}\right)\pi\i\right]
  +O(n^{-3/2}).
\end{split}
\end{equation*}
Since $e^{\frac{2l-m+p}{p-ab}\pi\i}-e^{-\frac{2l-m+p}{p-ab}\pi\i}=2\i\sin\left(\frac{(2l-m+p)\pi}{p-ab}\right)$, we have
\begin{equation*}
\begin{split}
  &
  V_1-V_2
  \\
  =&
  \sum_{(l,m)\in\partial\mathcal{S}_1}
  2\pi^2(-1)^{m}
  \sin\left(\frac{m\pi}{a}\right)\sin\left(\frac{m\pi}{b}\right)
  e^{\frac{2l-m+p}{p-ab}\pi\i}
  e^{-\frac{1}{p-ab}\frac{\pi\i}{n}}
  \\
  &\quad
  \times
  (-1)^{pn}\exp\left[n\left(-\frac{(l-m/2+p/2)^2}{p-ab}-\frac{m^2}{4ab} \right)\pi\i\right]
  \\
  &+
  \sum_{(l,m) \in \mathcal{S}_2}
  2\pi^2(-1)^{m}
  \sin\left(\frac{m\pi}{a}\right)\sin\left(\frac{m\pi}{b}\right)
  2\i\sin\left(\frac{(2l-m+p)\pi}{p-ab}\right)
  e^{-\frac{1}{p-ab}\frac{\pi\i}{n}}
  \\
  &\quad
  \times
  (-1)^{pn}\exp\left[n\left(-\frac{(l-m/2+p/2)^2}{p-ab}-\frac{m^2}{4ab} \right)\pi\i\right]
  \\
  &+
  2\pi\i
  \sum_{l}
  \sqrt{\frac{ab(p-ab)\pi^2\i}{pn}}
  \i(-1)^{ab+a+b}
  \frac{\sin\left(\frac{2al\pi}{p}\right)\sin\left(\frac{2bl\pi}{p}\right)
        \sin\left(\frac{2l\pi}{p}\right)}
      {\sin\left(\frac{2abl\pi}{p}\right)}
  e^{-\frac{1}{p-ab}\frac{\pi\i}{n}}
  \\
  &\quad\times
  (-1)^{np-1}
  \exp\left[n\left(-\frac{(l+p/2)^2}{p}\right)\pi\i\right]
  \\
  &+
  2\sum_{\text{$m-p$: even}}
  (-1)(2\pi\i)\frac{(-1)^m}{2}\sin\left(\frac{m\pi}{a}\right)\sin\left(\frac{m\pi}{b}\right)
  \\
  &\quad
  \times
  \left(-\pi\i+2\sqrt{\frac{\pi^2\i}{(p-ab)n}}\right)
  e^{-\frac{\pi\i}{(p-ab)n}}
  \exp\left[n\left(-ab-\frac{m^2}{4ab}\right)\pi\i\right]
  +O(n^{-3/2}).
\end{split}
\end{equation*}
When $l=m/2+p/2-ab$, the summand in the first term becomes
\begin{equation*}
\begin{split}
  &2\pi^2(-1)^{m}
  \sin\left(\frac{m\pi}{a}\right)\sin\left(\frac{m\pi}{b}\right)
  e^{-\frac{1}{p-ab}\frac{\pi\i}{n}}
  \\
  &\quad
  \times
  (-1)^{pn}\exp\left[n\left(-(p-ab)-\frac{m^2}{4ab} \right)\pi\i\right]
\end{split}
\end{equation*}
and the summation is over all $m$ such that $0<m<2ab$, $a\nmid m$, $b\nmid m$, and that $p-m$ is even.
So this term cancels with a part of the last term involving $-\pi\i$.
Therefore we have
\begin{equation*}
\begin{split}
  &V_1-V_2
  \\
  =&
  \sum_{(l,m) \in \mathcal{S}_2}
  2\pi^2(-1)^{m}
  \sin\left(\frac{m\pi}{a}\right)\sin\left(\frac{m\pi}{b}\right)
  2\i\sin\left(\frac{(2l-m+p)\pi}{p-ab}\right)
  e^{-\frac{1}{p-ab}\frac{\pi\i}{n}}
  \\
  &\quad
  \times
  (-1)^{pn}\exp\left[n\left(-\frac{(l-m/2+p/2)^2}{p-ab}-\frac{m^2}{4ab} \right)\pi\i\right]
  \\
  &+
  2\pi\i
  \sum_{l}
  \sqrt{\frac{ab(p-ab)\pi^2\i}{pn}}
  \i(-1)^{ab+a+b}
  \frac{\sin\left(\frac{2al\pi}{p}\right)\sin\left(\frac{2bl\pi}{p}\right)
        \sin\left(\frac{2l\pi}{p}\right)}
      {\sin\left(\frac{2abl\pi}{p}\right)}
  e^{-\frac{1}{p-ab}\frac{\pi\i}{n}}
  \\
  &\quad\times
  (-1)^{np-1}
  \exp\left[n\left(-\frac{(l+p/2)^2}{p}\right)\pi\i\right]
  \\
  &+
  2\sum_{\text{$m-p$: even}}
  (-1)(2\pi\i)\frac{(-1)^m}{2}\sin\left(\frac{m\pi}{a}\right)\sin\left(\frac{m\pi}{b}\right)
  \\
  &\quad
  \times
  \left(2\sqrt{\frac{\pi^2\i}{(p-ab)n}}\right)
  e^{-\frac{\pi\i}{(p-ab)n}}
  \exp\left[n\left(-ab-\frac{m^2}{4ab}\right)\pi\i\right]
  +O(n^{-3/2})
  \\
  =
  &
  (-1)^{pn}4\pi^2\i
  e^{-\frac{\pi\i}{(p-ab)n}}
  \sum_{(l,m) \in \mathcal{S}_2}
  (-1)^{m}
  \sin\left(\frac{m\pi}{a}\right)\sin\left(\frac{m\pi}{b}\right)
 \sin\left(\frac{(2l-m+p)\pi}{p-ab}\right)
  \\
  &\quad
  \times
  \exp\left[n\left(-\frac{(l-m/2+p/2)^2}{p-ab}-\frac{m^2}{4ab} \right)\pi\i\right]
  \\
  &
  +2\pi
  (-1)^{ab+a+b+pn}
  \sqrt{\frac{ab(p-ab)\pi^2\i}{pn}}
  e^{-\frac{\pi\i}{(p-ab)n}}
  \\
  &\quad\times
  \sum_{l}
  \frac{\sin\left(\frac{2al\pi}{p}\right)\sin\left(\frac{2bl\pi}{p}\right)
        \sin\left(\frac{2l\pi}{p}\right)}
      {\sin\left(\frac{2abl\pi}{p}\right)}
  \exp\left[n\left(-\frac{(l+p/2)^2}{p}\right)\pi\i\right]
  \\
  &
  -4\pi^2e^{3\pi\i/4}e^{-\frac{\pi\i}{(p-ab)n}}
  \frac{(-1)^{abn}}{\sqrt{(p-ab)n}}
  \\
  &\quad\times
  \sum_{\text{$m-p$: even}}
  (-1)^m\sin\left(\frac{m\pi}{a}\right)\sin\left(\frac{m\pi}{b}\right)
  \exp\left(-\frac{m^2n}{4ab}\pi\i\right)
  +O(n^{-3/2}).
\end{split}
\end{equation*}
\par
By Lemma~\ref{lem:sin_sin_exp} we have
\begin{equation*}
  \sum_{\substack{0\le m\le 2ab\\\text{$m-p$: even}}}
  \sin\left(\frac{m\pi}{a}\right)\sin\left(\frac{m\pi}{b}\right)
  \exp\left(-n\frac{m^2}{4ab}\pi \i\right)
  =
  0.
\end{equation*}
This implies that
\begin{equation*}
\begin{split}
  &V_1-V_2
  \\
  =&
  (-1)^{pn}4\pi^2\i
  e^{-\frac{\pi\i}{(p-ab)n}}
  \sum_{(l,m)\in\mathcal{S}_2}
  (-1)^{m}\sin\left(\frac{m\pi}{a}\right)\sin\left(\frac{m\pi}{b}\right)
  \sin\left(\frac{(2l-m+p)\pi}{p-ab}\right)
  \\
  &\quad\times
  \exp\left[n\left(-\frac{(2l-m+p)^2}{4(p-ab)}-\frac{m^2}{4ab} \right) \pi \i \right]
  \\
  &
  +
  (-1)^{a+b+ab+pn}4\pi\i\sqrt{\frac{ab(p-ab)\pi^2\i}{pn}}
  e^{-\frac{\pi\i}{(p-ab)n}}
  \\
  &\quad\times
  \sum_{|l|\le(p-1)/2}
  \frac{\sin\left(\frac{2al\pi}{p}\right)\sin\left(\frac{2bl\pi}{p}\right)
        \sin\left(\frac{2l\pi}{p}\right)}
       {\sin\left(\frac{2abl\pi}{p}\right)}
  \exp\left[n\left(-\frac{(l+p/2)^2}{p}\right)\pi\i\right]
  +O(n^{-3/2}).
\end{split}
\end{equation*}
From \eqref{eq:tau_V1_V2}, we have
\begin{equation*}
\begin{split}
  &\htau_n(X_p;\exp(4\pi\i/n))
  \\
  =&
  \frac{e^{3\pi\i/4}\sqrt{n}}{4\pi^2\sin(2\pi/n)}
  (-1)^{np+p}
  e^{-(p-ab+a/b+b/a - 3)
  \frac{\pi\i}{n}}
  e^{\frac{n}{4}\pi\i}
  \frac{1}{\sqrt{ab(p-ab)}}(V_1-V_2)
  \\
  =&
  (-1)^{p+1}
  e^{-(p-ab+a/b+b/a - 3)
  \frac{\pi\i}{n}}
  e^{\frac{n}{4}\pi\i}
  \frac{\sqrt{n}}{\sqrt{ab(p-ab)}\sin(2\pi/n)}
  \\
  &
  \Biggl\{
  e^{\pi\i/4}
  e^{-\frac{\pi\i}{(p-ab)n}}
  \sum_{(l,m)\in\mathcal{S}_2}
  (-1)^{m}\sin\left(\frac{m\pi}{a}\right)\sin\left(\frac{m\pi}{b}\right)
  \sin\left(\frac{(2l-m+p)\pi}{p-ab}\right)
  \\
  &
  \exp\left[n\left(-\frac{(2l-m+p)^2}{4(p-ab)}-\frac{m^2}{4ab}\right)\pi\i\right]
  \\
  &
  +
  \i
  \sqrt{\frac{ab(p-ab)}{pn}}
  (-1)^{a+b+ab}
  e^{-\frac{\pi\i}{(p-ab)n}}
  \sum_{|l|\le(p-1)/2}
  \frac{\sin\left(\frac{2al\pi}{p}\right)\sin\left(\frac{2bl\pi}{p}\right)
        \sin\left(\frac{2l\pi}{p}\right)}
       {\sin\left(\frac{2abl\pi}{p}\right)}
  \\
  &
  \exp\left[n\left(-\frac{(l+p/2)^2}{p}\right)\pi\i\right]
  +
  O(n^{-3/2})
  \Biggr\}.
\end{split}
\end{equation*}
Hence we have 
\begin{equation}\label{eq:tau_A_B}
  \htau_n(X_p;\exp(4\pi\i/n))
  =
  \frac{(-1)^{p+1}n^{3/2}}{2\pi}\left(A(n)+B(n)n^{-1/2}+O(n^{-1})\right),
\end{equation}
where
\begin{align*}
  A(n)
  &:=
  e^{\frac{n+1}{4}\pi\i}
  \sum_{(l,m)\in\mathcal{S}_2}
  (-1)^{m}\frac{\sin\left(\frac{m\pi}{a}\right)\sin\left(\frac{m\pi}{b}\right)
  \sin\left(\frac{(2l-m+p)\pi}{p-ab}\right)}{\sqrt{ab(p-ab)}}
  \\
  &\quad\times
  \exp
  \left[
    n\left(-\frac{(2l-m+p)^2}{4(p-ab)}-\frac{m^2}{4ab}\right)\pi\i
  \right],
  \\
  B(n)
  &:=
  2\i(-1)^{a+b+ab}
  e^{\frac{(1-p)n}{4}\pi\i}
  \sum_{1\le l\le(p-1)/2}
  (-1)^{l}
  \frac{\sin\left(\frac{2al\pi}{p}\right)\sin\left(\frac{2bl\pi}{p}\right)
  \sin\left(\frac{2l\pi}{p}\right)}
  {\sqrt{p}\sin\left(\frac{2abl\pi}{p}\right)}
  \\
  &\quad\times
  \exp
  \left[
    n\left(-\frac{l^2}{p}\right)\pi\i
  \right].
\end{align*}
If we put $g:=2l-m+p$, then the summation range $\mathcal{S}_2$ becomes
\begin{equation*}
\begin{split}
  \tilde{\mathcal{S}}
  :=
  \Bigl\{(g,m)\in\Z^2\mid&\,0<m<2ab,\frac{p-ab}{ab}m<g<2(p-ab),
  \\
  &\,
  h\equiv p-m\pmod{2},a\nmid m,b\nmid m\Bigr\},
\end{split}
\end{equation*}
which is the interior of the right-angled triangle with vertices $(0,0)$, $(2(p-ab),0)$, and $(2(p-ab),2ab)$.
Using parameters $(g,m)$, we have
\begin{equation}\label{eq:A_hm}
  A(n)
  =
  \frac{e^{\frac{n+1}{4}\pi\i}}{\sqrt{ab(p-ab)}}
  \sum_{(g,m)\in\tilde{\mathcal{S}}}G(g,m)
\end{equation}
with
\begin{multline}\label{eq:G_def}
  G(g,m)
  \\
  :=
  (-1)^{m}
  \sin\left(\frac{m\pi}{a}\right)
  \sin\left(\frac{m\pi}{b}\right)
  \sin\left(\frac{g\pi}{p-ab}\right)
  \exp
  \left[
    n\left(-\frac{g^2}{4(p-ab)}-\frac{m^2}{4ab}\right)\pi\i
  \right].
\end{multline}
We have the following symmetries of $G(g,m)$:
\begin{align*}
  G(2(p-ab)-g,m)
  &=
  (-1)^{ab+m+1}G(g,m),
  \\
  G(g,2ab-m)
  &=
  (-1)^{ab+m}G(g,m).
\end{align*}
We divide $\tilde{\mathcal{S}}$ into three parts $\RD$, $\mathcal{S}'$ and $\tilde{\mathcal{R}}$:
\begin{align*}
  \RD
  &:=
  \tilde{\mathcal{S}}\cap\{(g,m)\in\Z^2\mid g<p-ab\}
  \\
  \mathcal{S}'
  &:=
  \tilde{\mathcal{S}}\cap\{(g,m)\in\Z^2\mid m>ab\}
  \\
  \tilde{\mathcal{R}}
  &:=
  \tilde{\mathcal{S}}\cap\{(g,m)\in\Z^2\mid p-ab<g, m<ab\}
\end{align*}
Note that we can exclude the line $g=p-ab$, since $A(n)$ vanishes when $g=p-ab$.
\begin{figure}[H]
  \raisebox{-0.5\height}{\includegraphics[scale=0.4]{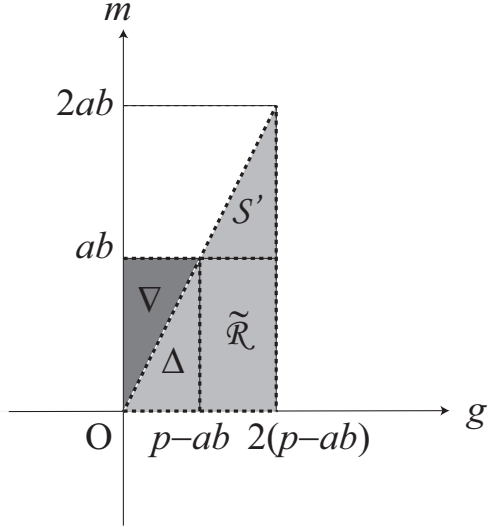}}
  \caption{The light gray area indicates $\tilde{\mathcal{S}}=\RD\cup\mathcal{S}'\cup\tilde{\mathcal{R}}$.
  The dark gray area is $\RN$.
  We put $\mathcal{R}:=\RD\cup\RN$.}
  \label{fig:S_tilde}
\end{figure}
Then we have
\begin{equation}\label{eq:G_sum_R}
\begin{split}
  \sum_{(g,m)\in\tilde{\mathcal{S}}}G(g,m)
  &=
  \sum_{(g,m)\in\RD\cup\mathcal{S}'\cup\tilde{\mathcal{R}}}G(g,m)
  \\
  &=
  \sum_{(g,m)\in\RD}G(g,m)
  -
  \sum_{(g,m)\in\RN}G(g,m)
  +
  \sum_{(g,m)\in\mathcal{R}}(-1)^{ab+m+1}G(g,m),
\end{split}
\end{equation}
where
\begin{align*}
  \RN
  &:=
  \left\{
    (g,m)
    \mid
    0<m<ab,0<g<p-ab,h<\frac{p-ab}{ab}m,
    h\modtwo{p-m},a\nmid m,b\nmid m
  \right\},
  \\
  \mathcal{R}
  &:=
  \left\{
    (g,m)
    \mid
    0<m<ab,0<g<p-ab,
    h\modtwo{p-m},a\nmid m,b\nmid m
  \right\}.
\end{align*}
Since
\begin{equation*}
  \RD
  =
  \left\{
    (g,m)
    \mid
    0<m<2ab,\frac{p-ab}{ab}m<g<p-ab,h\modtwo{p-m},a\nmid m,b\nmid m
  \right\},
\end{equation*}
$\mathcal{R}$ splits into $\RD$ and $\RN$.
So we can regard $A(n)$ as a summation over $\mathcal{R}$.

\section{The Reidemeister torsion and the Chern--Simons invariant}
\subsection{Fundamental group}
We assume that $b$ is odd.
For the torus knot $T(a,b)$ we have
\begin{equation*}
  \pi_1(S^3\setminus{T(a,b)})
  =
  \langle x,y\mid x^a=y^b\rangle.
\end{equation*}
Note that
\begin{itemize}
\item $z:=x^a=y^b$ is central,
\item $\mu=x^{-c}y^d$ is a meridian, where $c$ and $d$ are integers such that $ad-bc=1$, and 
\item $\lambda = z\mu^{-ab}$ is the preferred longitude. 
\end{itemize}
Put $X:=S^3\setminus\Int{N(T(a,b))}$, and $D:=D^2\times S^1$.
Let $X_p$ be the closed three-manifold obtained from $X$ by $p$-surgery.
Then $X_p$ is $X\cup_{i_p}D$, where $i_p\colon\partial{D}\to\partial{X}$ sending $\partial{D}^2\times\{\text{point}\}$ to $\mu_D:=\lambda\mu^p$, and $\{\text{point in $\partial{D^2}$}\}\times S^1$ to $\mu^{-1}$, where we identify a simple closed curve on the torus $\partial{X}$ with an element in $\pi_1(\partial{X})$.
\par
Since in $\pi_1(X_p)$, $\mu_D=1$, which is equivalent to $\mu^{p-ab}=z^{-1}$, we have the following presentation of $\pi_1(X_p)$:
\begin{equation}\label{eq:pi1_X_p}
  \pi_1(X_p)
  =
  \langle x,y,\mu\mid x^a=y^b,\mu=x^{-c}y^d,\mu^{p-ab}=x^{-a}\rangle.
  \tag{G}
\end{equation}
Note that we do not need the generator $\mu$.
\par
Regarding $X_p$ as the Seifert fibered space $S(-a/c,b/d,p-ab)$, we have another presentation:
\begin{equation}\label{eq:pi1_Seifert}
  \pi_1\bigr(S(-a/c,b/d,p-ab)\bigr)
  =
  \langle
    \alpha,\beta,\gamma,f
    \mid
    [\alpha,f]=[\beta,f]=[\gamma,f]=\alpha^af^{-c}=\beta^bf^{d}=\gamma^{p-ab}f
    =\alpha\beta\gamma=1
  \rangle,
\end{equation}
where $\alpha$, $\beta$, and $\gamma$ go around the singular fibers with indices $-a/b$, $b/d$, and $p-ab$ respectively, and $f$ is a regular fiber.
We use $\pi_1(X_p)$ for the presentation \eqref{eq:pi1_X_p} and $\pi_1\bigr(S(-a/c,b/d,p-ab)\bigr)$ for that described in \eqref{eq:pi1_Seifert}.
\par
We will construct a concrete isomorphism between $\pi_1(X_p)$ and $\pi_1\bigr(S(-a/c,b/d,p-ab)\bigr)$.
\par
Define a homomorphism $\Phi\colon \pi_1\bigr(S(-a/c,b/d,p-ab)\bigr)\to\pi_1(X_p)$ by
\begin{align*}
  \Phi(\alpha)
  &:=
  x^{-c}
  \\
  \Phi(\beta)
  &:=
  y^{d},
  \\
  \Phi(\gamma)
  &:=
  \mu^{-1},
  \\
  \Phi(f)
  &:=
  \mu^{p-ab}.
\end{align*}
Since we have
\begin{align*}
  \Phi(\alpha^{a}f^{-c})
  &=
  x^{-ac}\mu^{-c(p-ab)}
  =
  x^{-ac}x^{ac}
  =1,
  \\
  \Phi(\beta^{b}f^{d})
  &=
  y^{bd}\mu^{d(p-ab)}
  =
  y^{bd}x^{-ad}
  =1,
  \\
  \Phi(\gamma^{p-ab}f)
  &=
  \mu^{ab-p}\mu^{p-ab}
  =1,
  \\
  \Phi([\alpha,f])
  &=
  [x^{-c},\mu^{p-ab}]
  =
  [x^{-c},x^{-a}]
  =1,
  \\
  \Phi([\beta,f])
  &=
  [y^{d},\mu^{p-ab}]
  =
  [y^{d},y^{-b}]
  =1,
  \\
  \Phi([\gamma,f])
  &=
  [\mu^{-1},\mu^{p-ab}]
  =1,
  \\
  \Phi(\alpha\beta\gamma)
  &=
  x^{-c}y^{d}\mu^{-1}
  =1,
\end{align*}
$\Phi$ is well-defined.
\par
We also define $\Psi\colon\pi_1(X_p)\to\pi_1\bigr(S(-a/c,b/d,p-ab)\bigr)$ by
\begin{align*}
  \Psi(x)
  &:=
  \alpha^{b}f^{-d},
  \\
  \Psi(y)
  &:=
  \beta^{a}f^{c},
  \\
  \Phi(\mu)
  &:=
  \gamma^{-1}.
\end{align*}
Since we have
\begin{align*}
  \Psi(x^ay^{-b})
  &=
  \alpha^{ab}f^{-ad}\beta^{-ab}f^{-bc}
  =
  f^{bc}f^{-ad}f^{ad}f^{-bc}
  =1,
  \\
  \Psi(\mu y^{-d}x^{c})
  &=
  \gamma^{-1}f^{-cd}\beta^{-ad}\alpha^{bc}f^{-cd}
  =
  \gamma^{-1}\beta^{bc}\beta^{-ad}\alpha^{bc}\alpha^{-ad}
  =
  \gamma^{-1}\beta^{-1}\alpha^{-1}
  =1,
  \\
  \Psi(\mu^{p-ab}x^{a})
  &=
  \gamma^{ab-p}\left(\alpha^{b}f^{-d}\right)^{a}
  =
  f\alpha^{ab}f^{-ad}
  =
  f\cdot f^{bc}f^{-ad}
  =1,
\end{align*}
$\Psi$ is also well-defined.
\par
Since $\Phi\circ\Psi=\Id_{\pi_1(X_p)}$ and $\Psi\circ\Phi=\Id_{\pi_1\bigr(S(-a/c,b/d,p-ab)\bigr)}$, both $\Phi$ and $\Psi$ are isomorphisms, as desired.
\subsection{Representations}\label{subsec:representation}
It is known that the $\SL(2;\C)$ character variety of $T(a,b)$ has $(a-1)(b-1)/2$ irreducible components and an Abelian component.
Put
\begin{equation*}
  \mathcal{P}
  :=
  \{(k,l)\mid0<k<a,0<l<b,k\equiv l\pmod{2}\}.
\end{equation*}
Then the irreducible components are indexed by $\mathcal{P}$ as follows.
Let $\rho_{k,l}^{\rm{Irr}}\colon\pi_1(X)\to\SL(2;\C)$ be an irreducible representation that belongs to the component indexed by $(k,l)\in\mathcal{P}$.
Then it satisfies
\begin{align}\label{eq:irrep_kl}
  \tr\rho_{k,l}^{\rm{Irr}}(x)
  &=
  2\cos\left(\frac{k\pi}{a}\right),
  \\
  \tr\rho_{k,l}^{\rm{Irr}}(y)
  &=
  2\cos\left(\frac{l\pi}{b}\right).
\end{align}
See \cite{Klassen:TRAAM1991,Munoz:REVMC2009}.
\par
The irreducible component indexed by $(k,l)\in\mathcal{P}$ intersects with the Abelian component in two representations whose traces of the meridian $\mu$ are $2\cos\left(\frac{(adl-bck)\pi}{ab}\right)$ and $2\cos\left(\frac{(adl+bck)\pi}{ab}\right)$, where $c$ and $d$ are integers satisfying $ad-bc=1$.
\par
We want to extend the representation $\rho_{k,l}^{\rm{Irr}}$ to a representation of $\pi_1(X_p)$.
Since $\tr\rho_{k,l}^{\rm{Irr}}(x)\ne\pm2$, we can find $Q\in\SL(2;\C)$ such that $Q^{-1}\rho_{k,l}^{\rm{Irr}}(x)Q=\begin{pmatrix}e^{k\pi\i/a}&0\\0&e^{-k\pi\i/a}\end{pmatrix}$.
Then we have $\rho_{k,l}^{\rm{Irr}}(\mu_D)=(-1)^k\left(\rho_{k,l}^{\rm{Irr}}(\mu)\right)^{p-ab}$.
Therefore if $\tr\rho_{k,l}^{\rm{Irr}}(\mu)=2\cos\left(\frac{h\pi}{p-ab}\right)$ with $0<h<p-ab$ and $h\equiv k\equiv l\pmod{2}$, then there exists a representation $\trho_{h,k,l}^{\rm{Irr}}\colon\pi_1(X_p)\to\SL(2;\C)$.
\par
Therefore we have the following proposition.
\begin{proposition}\label{prop:irr_rep}
Put
\begin{equation}\label{eq:H}
  \mathcal{H}
  :=
  \{(h,k,l)\in\Z^3\mid 0<h<p-ab,0<k<a,0<l<b,h\equiv k\equiv l\pmod{2}\}.
\end{equation}
Then for any triple $(h,k,l)$ there exists an irreducible representation $\trho_{h,k,l}^{\rm{Irr}}\colon\pi_1(X_p)\to\SL(2;\C)$ such that
\begin{align*}
  \tr\trho_{h,k,l}^{\rm{Irr}}(x)
  &=
  2\cos\left(\frac{k\pi}{a}\right),
  \\
  \tr\trho_{h,k,l}^{\rm{Irr}}(y)
  &=
  2\cos\left(\frac{l\pi}{b}\right),
  \\
  \tr\trho_{h,k,l}^{\rm{Irr}}(\mu)
  &=
  2\cos\left(\frac{h\pi}{p-ab}\right).
\end{align*}
Moreover the representation is unique up to conjugation.
\end{proposition}
We can describe these representations in terms of $\pi_1\bigl(S(-a/c,b/d,p-ab)\bigr)$ by using the isomorphism $\Phi\colon\pi_1\bigl(S(-a/c,b/d,p-ab)\bigr)\to\pi_1(X_p)$.
The generators of the presentation \eqref{eq:pi1_Seifert} is mapped as follows:
\begin{align}\label{eq:irr_pi1_Seifert}
  \tr\trho_{h,k,l}^{\rm{Irr}}(\alpha)
  &=
  2\cos\left(\frac{ck\pi}{a}\right),
  \\
  \tr\trho_{h,k,l}^{\rm{Irr}}(\beta)
  &=
  2\cos\left(\frac{dl\pi}{b}\right),
  \\
  \tr\trho_{h,k,l}^{\rm{Irr}}(\gamma)
  &=
  2\cos\left(\frac{h\pi}{p-ab}\right),
  \\
  \tr\trho_{h,k,l}^{\rm{Irr}}(f)
  &=
  2(-1)^h.
\end{align}
\begin{remark}
Since $X_{p}$ is the Seifert fibered space $S(-a/c,b/d,p-ab)$, the result above is well known.
\end{remark}
We also introduce reducible Abelian representations indexed by $\{l\in\Z\mid0<l<p\}$.
\begin{definition}\label{def:abel_rep}
For an integer with $0<l<p$, let $\trho_{l}^{\rm{Abel}}\colon\pi_1(X_p)\to\SL(2;\C)$ be the Abelian representation sending $\mu\to\begin{pmatrix}e^{2l\pi\i/p}&0\\0&e^{-2l\pi\i/p}\end{pmatrix}$.
\end{definition}
Since $H_1(X_{p})=\Z/p\Z$, this is well-defined.
\subsection{Reidemeister torsion}\label{subsec:torsion}
In this subsection, we describe the \emph{homological} Reidemeister torsion of $X_p$ twisted by the adjoint action of the irreducible representation $\trho_{h,k,l}^{\rm{Irr}}$ with $(h,k,l)\in\mathcal{H}$.
\par
For a representation $\rho\colon\pi_1(X)\to\SL(2;\C)$, let $\Tor_{\gamma}(X;\rho)$ be the twisted Reidemeister torsion associated with a simple closed curve $\gamma\subset\partial{X}$.
Here we assume that $\rho$ is $\gamma$-regular \cite[D{\'e}finition~3.21]{Porti:MAMCAU1997}.
\par
If $M$ is an oriented, closed three-manifold and $\trho\colon\pi_1(M)\to\SL(2;\C)$ is a representation.
We denote by $\Tor(M;\trho)$ the Reidemeister torsion twisted by the adjoint action of $\rho$.
\par
Now we can calculate $\Tor(X_{p};\trho_{h,k,l}^{\rm{Irr}})$ as follows.
\begin{lemma}\label{lem:torsion_hkl}
The Reidemeister torsion twisted by the irreducible representation $\trho_{h,k,l}^{\rm{Irr}}$ is given by
\begin{equation*}
  \Tor(X_{p};\trho_{h,k,l}^{\rm{Irr}})
  =
  \pm
  \frac{ab(p-ab)}
       {64\sin^2\left(\frac{k\pi}{a}\right)
          \sin^2\left(\frac{l\pi}{b}\right)
          \sin^2\left(\frac{h\pi}{p-ab}\right)}.
\end{equation*}
\end{lemma}
\begin{proof}
Since $\lambda=z\mu^{-ab}$ and $z$ is central, by a conjugation we may assume that $\rho_{k,l}^{\rm{Irr}}(\mu)=\begin{pmatrix}e^{\frac{h\pi\i}{p-ab}}&\ast\\0&e^{-\frac{h\pi\i}{p-ab}}\end{pmatrix}$ and that $\rho_{k,l}^{\rm{Irr}}(\lambda)=\pm\begin{pmatrix}e^{\frac{-abh\pi\i}{p-ab}}&\ast\\0&e^{\frac{abh\pi\i}{p-ab}}\end{pmatrix}$.
Therefore we have
\begin{equation*}
  \Tor_{\mu}(X;\rho_{k,l}^{\rm{Irr}})
  =
  \pm
  \frac{1}{ab}
  \frac{a^2b^2}{16\sin^2\left(\frac{k\pi}{a}\right)\sin^2\left(\frac{l\pi}{b}\right)}
  \\
  =
  \pm\frac{ab}{16\sin^2\left(\frac{k\pi}{a}\right)\sin^2\left(\frac{l\pi}{b}\right)}
\end{equation*}
from \eqref{eq:torsion_gamma}.
Note that this holds for any irreducible representation in the component indexed by $(k,l)$.
\par
Since $\rho_{k,l}^{\rm{Irr}}(\lambda\mu^{p})=\pm\begin{pmatrix}e^{h\pi\i}&\ast\\0&e^{-h\pi\i}\end{pmatrix}=\pm\trho_{h,k,l}^{\rm{Irr}}(\mu)^{p-ab}$, we have
\begin{equation*}
  \Tor_{i_p(\mu_D)}(X;\rho_{k,l}^{\rm{Irr}})
  =
  \pm
  (p-ab)
  \Tor_{\mu}(X;\rho_{k,l}^{\rm{Irr}})
  =
  \pm
  \frac{ab(p-ab)}{16\sin^2\left(\frac{k\pi}{a}\right)\sin^2\left(\frac{l\pi}{b}\right)}
\end{equation*}
from \eqref{eq:torsion_gamma}.
\par
Since $\rho_{k,l}^{\rm{Irr}}(i_p(\lambda_D))=\rho_{k,l}^{\rm{Irr}}(\mu)^{-1}$, we have $\tr\rho_{k,l}^{\rm{Irr}}(i_p(\lambda_D))=2\cos\left(\frac{h\pi}{p-ab}\right)$.
So we finally have
\begin{equation*}
\begin{split}
  \Tor(X_{p};\trho_{h,k,l}^{\rm{Irr}})
  &=
  \pm
  \frac{ab(p-ab)}{16\sin^2\left(\frac{k\pi}{a}\right)\sin^2\left(\frac{l\pi}{b}\right)}
  \times
  \frac{1}{4\cos^2\left(\frac{h\pi}{p-ab}\right)-4}
  \\
  &=
  \pm
  \frac{ab(p-ab)}
       {64\sin^2\left(\frac{k\pi}{a}\right)
          \sin^2\left(\frac{l\pi}{b}\right)
          \sin^2\left(\frac{h\pi}{p-ab}\right)}
\end{split}
\end{equation*}
from \eqref{eq:torsion_surgery}.
\end{proof}
\par
Let $\rho_{l}^{\rm{Abel}}$ be the reducible Abelian representation $\pi_1(X)\to\SL(2;\C)$ sending $\mu$ to $\begin{pmatrix}e^{2l\pi\i/p}&0\\0&e^{-2l\pi\i/p}\end{pmatrix}$.
Note that $\rho_{l}^{\rm{Abel}}$ can be extended to $\trho_{l}^{\rm{Abel}}\colon\pi_1(X_{p})\to\SL(2;\C)$; $\rho_{l}^{\rm{Abel}}=\trho_{l}^{\rm{Abel}}\Bigm|_{\pi_1(X)}$.
The Reidemeister torsion $\Tor_{\mu}(X;\rho_{l}^{\rm{Abel}})$ of $X$ twisted by $\rho_{l}^{\rm{Abel}}$ associated with $\mu$ is given by $\pm\left(\frac{\Delta(T(a,b);e^{4l\pi\i/p})}{2\sinh(2l\pi\i/p)}\right)^2$ (\cite[Theorem~4]{Milnor:ANNMA21962}, \cite[Theorem~1.1.2]{Turaev:USPMN1986}; see also \cite[Proposition~5.1]{Murakami/Yokota:2018}), where $\Delta(K;t)$ is the normalized Alexander polynomial of a knot $K$.
\par
Then we show
\begin{lemma}\label{lem:Reidemeister_X_p_Abel}
The Reidemeister torsion twisted by the Abelian representation $\rho_{l}^{\rm{Abel}}$ is given by
\begin{equation*}
  \Tor(X_p;\trho_{l}^{\rm{Abel}})
  =
  \pm
  \frac{p\sin^2\left(\frac{2lab\pi}{p}\right)}
       {16\sin^2\left(\frac{2la\pi}{p}\right)\sin^2\left(\frac{2lb\pi}{p}\right)
          \sin^2\left(\frac{2l\pi}{p}\right)}.
\end{equation*}
\end{lemma}
\begin{proof}
Since it is well-known (see for example \cite[Chapter~11]{Lickorish:1997}) that
\begin{equation*}
  \Delta(T(a,b);t)
  =
  \frac{\left(t^{ab/2}-t^{-ab/2}\right)\left(t^{1/2}-t^{-1/2}\right)}
       {\left(t^{a/2}-t^{-a/2}\right)\left(t^{b/2}-t^{-b/2}\right)},
\end{equation*}
we have
\begin{equation*}
  \Tor_{\mu}(X;\rho_{l}^{\rm{Abel}})
  =
  \pm
  \frac{\sin^2\left(\frac{2abl\pi}{p}\right)}
       {4\sin^2\left(\frac{2al\pi}{p}\right)\sin^2\left(\frac{2bl\pi}{p}\right)}.
\end{equation*}
Since $\rho_{l}^{\rm{Abel}}(i_p(\mu_D))=\rho_{l}^{\rm{Abel}}(\mu)^p$, we have
\begin{equation*}
  \Tor_{i_p(\mu_D)}(X;\rho_{l}^{\rm{Abel}})
  =
  \pm
  \frac{p\sin^2\left(\frac{2ab\pi}{p}\right)}
       {4\sin^2\left(\frac{2al\pi}{p}\right)\sin^2\left(\frac{2bl\pi}{p}\right)}
\end{equation*}
from \eqref{eq:torsion_gamma}.
Since we also know that $\rho_{l}^{\rm{Abel}}(i_p(\lambda_D))=\rho_{l}^{\rm{Abel}}(\mu)^{-1}$, from \eqref{eq:torsion_surgery} we have
\begin{equation*}
\begin{split}
  \Tor(X_p;\trho_{l}^{\rm{Abel}})
  &=
  \pm
  \frac{\Tor_{i_p(\mu_D)}(X;\rho_{l}^{\rm{Abel}})}{(\tr\rho_{l}^{\rm{Abel}}(i_p(\lambda_D)))^2-4}
  \\
  &=
  \pm
  \frac{p\sin^2\left(\frac{2abl\pi}{p}\right)}
       {4\sin^2\left(\frac{2al\pi}{p}\right)\sin^2\left(\frac{2bl\pi}{p}\right)}
  \times\frac{1}{4\cos^2\left(\frac{2l\pi}{p}\right)-4}
  \\
  &=
  \pm
  \frac{p\sin^2\left(\frac{2abl\pi}{p}\right)}
       {16\sin^2\left(\frac{2al\pi}{p}\right)\sin^2\left(\frac{2bl\pi}{p}\right)
          \sin^2\left(\frac{2l\pi}{p}\right)},
\end{split}
\end{equation*}
completing the proof
\end{proof}
\subsection{Chern-Simons invariant}
In this subsection we calculate the Chern--Simons invariants of $X_p$ associated with representations described in Subsection~\ref{subsec:representation}.
We denote by $\CS(M;\rho)\in\C/\Z$ the Chern--Simons invariant of a closed three-manifold associated with a representation $\rho$.
We use a formula by Kirk and Klassen \cite[Theorem~4.2]{Kirk/Klassen:MATHA1990} (see also the first paragraph in Page 354 in that paper).
Note that in our case, $K=T(a,b)$, $E=X$, $M=X_p$, $\mu_D=\mu^p\lambda\in\pi_1(X)$ and $\lambda_D=\mu^{-1}\in\pi_1(X)$.
\par
First we consider the Abelian representations.
See \cite[Theorem~4.5]{Kirk/Klassen:COMMP1993}, noting our sign convention (Remark~\ref{rem:CS_sign}).
\begin{lemma}
Let $\trho_{l}^{\rm{Abel}}$ be the Abelian representation defined in Definition~\ref{def:abel_rep}.
Then the Chern--Simons invariant of $\trho_{l}^{\rm{Abel}}$ is given as
\begin{equation*}
  \CS(X_p;\trho_{l}^{\rm{Abel}})
  =
  -\frac{l^2}{p}
  \in\C/\Z.
\end{equation*}
\end{lemma}
\begin{proof}
Now let us consider a path of representations $\rho_t\colon\pi_1(X)\to\SL(2;\C)$ sending $\mu$ to $\begin{pmatrix}e^{2tl\pi\i/p}&0\\0&e^{-2tl\pi\i/p}\end{pmatrix}$ and $\lambda$ to the identity matrix.
Then $\rho_0$ is trivial and $\rho_1(\mu)=\trho_{l}^{\rm{Abel}}\Bigm|_{X}(\mu)$.
Therefore both $\rho_0$ and $\rho_1$ can be extended to representations $\tilde{\rho}_0$, which is trivial, and $\tilde{\rho}_1=\trho_{l}^{\rm{Abel}}$ of $\pi_1(X_p)$.
We also see that if we put $\alpha(t):=lt$ and $\beta(t):=-lt/p$, then we have
\begin{align*}
  \rho_t(\mu_D)
  &=
  \begin{pmatrix}e^{2\pi\i\alpha(t)}&0\\0&e^{-2\pi\i\alpha(t)}\end{pmatrix},
  \\
  \rho_t(\lambda_D)
  &=
  \begin{pmatrix}e^{2\pi\i\beta(t)}&0\\0&e^{-2\pi\i\beta(t)}\end{pmatrix}.
\end{align*}
Therefore from Theorem~\ref{thm:Kirk_Klassen} we have
\begin{equation*}
  \CS(X_p;\trho_{l}^{\rm{Abel}})
  =
  \CS(X_p;\tilde{\rho}_1)-\CS(X_p;\tilde{\rho}_0)
  =
  2\int_{0}^{1}\frac{-l^2t}{p}\,dt
  =
  -\frac{l^2}{p},
\end{equation*}
since $\CS(X_p;\tilde{\rho}_0)=0$.
See \cite[Theorem~5.1]{Kirk/Klassen:MATHA1990}.
\end{proof}
Next we consider the irreducible representations.
The following lemma is also well known.
See \cite[5.2.~Theorem]{Kirk/Klassen:MATHA1990} and \cite[Proposition~2.3]{Andersen/Petersen:2018}.
Notice again our sign convention; ours agrees with that of \cite{Andersen/Petersen:2018}.
\begin{lemma}\label{lem:CS_X_p_irr}
Let $\trho_{h,k,l}^{\rm{Irr}}$ be the irreducible representation described in Proposition~\ref{prop:irr_rep}.
Then we have
\begin{equation}\label{eq:CS_hkl}
  \CS(X_p;\trho_{h,k,l}^{\rm{Irr}})
  =
  -\frac{h^2}{4(p-ab)}-\frac{(adl-bck)^2}{4ab}
  \in\C/\Z.
\end{equation}
\end{lemma}
\begin{proof}
Fix integers $c$ and $d$ such that $ad-bc=1$ as usual.
\par
We will define a path of representations $\varphi_t\colon\pi_1(X)\to\SL(2;\C)$ ($0\le t\le1$) as follows.
\begin{itemize}
\item
For $0\le t\le1/2$, define
\begin{equation}\label{eq:path_1}
\begin{cases}
  \varphi_t(x)
  &:=
  \begin{pmatrix}e^{2(adl-bck)t\pi\i/a}&0\\0&e^{-2(adl-bck)t\pi\i/a}\end{pmatrix},
  \\
  \varphi_t(y)
  &:=
  \begin{pmatrix}e^{2(adl-bck)t\pi\i/b}&0\\0&e^{-2(adl-bck)t\pi\i/b}\end{pmatrix}.
\end{cases}
\end{equation}
Note that $\varphi_0$ is trivial and so it can be extended to the trivial representation $\tilde{\varphi}_0\colon\pi_1(X_p)\to\SL(2;\C)$.
Note also that
\begin{equation*}
\begin{cases}
  \varphi_{1/2}(x)
  &=
  \begin{pmatrix}e^{k\pi\i/a}&0\\0&e^{-k\pi\i/a}\end{pmatrix},
  \\
  \varphi_{1/2}(y)
  &=
  \begin{pmatrix}e^{l\pi\i/b}&0\\0&e^{-l\pi\i/b}\end{pmatrix}
\end{cases}
\end{equation*}
since $k\equiv l\pmod{2}$.
\item
For $1/2\le t\le1$, we will construct a path from $\varphi_{1/2}$ to $\rho_{k,l}^{\rm{Irr}}\colon\pi_1(X)\to\SL(2;\C)$.
Recall that the representation $\rho_{k,l}^{\rm{Irr}}$ satisfies $\tr\rho_{k,l}^{\rm{Irr}}(x)=2\cos(k\pi/a)$ and $\tr\rho_{k,l}^{\rm{Irr}}(y)=2\cos(l\pi/b)$ (see Proposition~\ref{prop:irr_rep}).
Since the Chern--Simons invariant does not depend on conjugacy classes, we may assume that $\rho_{k,l}^{\rm{Irr}}(x)=\begin{pmatrix}e^{k\pi\i/a}&0\\0&e^{-k\pi\i/a}\end{pmatrix}$.
Let $Q\in\SL(2;\C)$ be a matrix such that $Q^{-1}\rho_{k,l}^{\rm{Irr}}(y)Q=\begin{pmatrix}e^{l\pi\i/b}&0\\0&e^{-l\pi\i/b}\end{pmatrix}$.
Since $\SL(2;\C)$ is connected, there exists a path $P(t)\colon[1/2,1]\to\SL(2;\C)$ such that $P(1/2)=I_2$ and $P(1)=Q$, where $I_2$ is the $2\times2$ identity matrix.
Now we put
\begin{equation}\label{eq:path_2}
\begin{cases}
  \varphi_t(x)
  &:=
  \begin{pmatrix}e^{k\pi\i/a}&0\\0&e^{-k\pi\i/a}\end{pmatrix},
  \\
  \varphi_t(y)
  &:=
  P(t)\begin{pmatrix}e^{l\pi\i/b}&0\\0&e^{-l\pi\i/b}\end{pmatrix}P(t)^{-1}
\end{cases}
\end{equation}
for $1/2\le t\le1$.
Then since $\varphi_t(y)^b=(-1)^{l}I_2$ and $k\equiv l\pmod{2}$, we see that $\varphi_t$ is well defined.
Note that $\varphi_{1/2}$ in \eqref{eq:path_2} coincides with that in \eqref{eq:path_1}, and that $\varphi_{1}$ equals $\rho_{k,l}^{\rm{Irr}}$, which can be extended to $\trho_{h,k,l}^{\rm{Irr}}\colon\pi_1(X_p)\to\SL(2;\C)$ for $(h,k,l)\in\mathcal{H}$.
\end{itemize}
\par
Next we calculate $\varphi_t(\mu_D)$ and $\varphi_t(\lambda_D)$.
Recall that $\mu=x^{-c}y^d$, $\lambda=x^a\mu^{-ab}$, $\mu_D=\lambda\mu^p$, and $\lambda_D=\mu^{-1}$.
\begin{itemize}
\item
For $0\le t\le1/2$ we see that $\varphi_t(\mu)$ is a diagonal matrix with $(1,1)$-entry $e^{2(adl-bck)t\pi\i/(ab)}$ and that $\varphi_t(\lambda)=I_2$.
Therefore we have
\begin{equation*}
\begin{cases}
  \varphi_t(\mu_D)
  &=
  \begin{pmatrix}
    e^{2p(adl-bck)t\pi\i/(ab)}&0\\0&e^{-2p(adl-bck)t\pi\i/(ab)}
  \end{pmatrix},
  \\
  \varphi_t(\lambda_D)
  &=
  \begin{pmatrix}
    e^{-2(adl-bck)t\pi\i/(ab)}&0\\0&e^{2(adl-bck)t\pi\i/(ab)}
  \end{pmatrix},
\end{cases}
\end{equation*}
and so we can put $\alpha(t):=\dfrac{p(adl-bck)}{ab}t$ and $\beta(t):=\dfrac{-(adl-bck)}{ab}t$ to use Theorem~\ref{thm:Kirk_Klassen}.
Then we have
\begin{equation}\label{eq:cs_rho_hkl_1}
  2\int_{0}^{1/2}\beta(t)\frac{d\,\alpha(t)}{d\,t}\,dt
  =
  -\frac{p(adl-bck)^2}{4a^2b^2}
\end{equation}
\item
For $1/2\le t\le1$ we see that $\varphi_t(x^a)=(-1)^kI_2$ for any $1/2\le t\le1$.
So we have $\varphi_t(\lambda)=(-1)^k\varphi_t(\mu)^{-ab}$, $\varphi_t(\mu_D)=(-1)^k\varphi_t(\mu)^{p-ab}$ and $\varphi_t(\lambda_D)=\varphi_t(\mu)^{-1}$.
Therefore if we assume that $\varphi_t(\mu)=\begin{pmatrix}e^{2\pi\i f(t)}&0\\0&e^{-2\pi\i f(t)}\end{pmatrix}$ after conjugation, then we have
\begin{equation*}
\begin{cases}
  \varphi_t(\mu_D)
  &=
  \begin{pmatrix}
    e^{2\pi\i((p-ab)f(t)+k/2)}&0 \\
    0&e^{-2\pi\i((p-ab)f(t)+k/2)}
  \end{pmatrix},
  \\
  \varphi_t(\lambda_D)
  &=
  \begin{pmatrix}
    e^{-2\pi\i f(t)}&0 \\
    0&e^{2\pi\i f(t)}
  \end{pmatrix}.
\end{cases}
\end{equation*}
Therefore we can put $\alpha(t):=(p-ab)f(t)+k/2$ and $\beta(t):=-f(t)$ to use Theorem~\ref{thm:Kirk_Klassen}.
Then we have
\begin{equation}\label{eq:cs_rho_hkl_2}
\begin{split}
  2
  \int_{1/2}^{1}\beta(t)\frac{d\,\alpha(t)}{d\,t}\,dt
  &=
  -2
  \int_{1/2}^{1}
  (p-ab)f'(t)f(t)
  \,dt
  \\
  &=
  -(p-ab)\Bigl[f(t)^2\Bigr]_{1/2}^{1}
  \\
  &=
  -\frac{h^2}{4(p-ab)}
  +
  \frac{(p-ab)(adl-bck)^2}{4a^2b^2}
\end{split}
\end{equation}
since $f(1)=\frac{h}{2(p-ab)}$ and $f(\frac{1}{2})=\frac{adl-bck}{2ab}$.
\end{itemize}
Therefore from \eqref{eq:cs_rho_hkl_1}, \eqref{eq:cs_rho_hkl_2} and Theorem~\ref{thm:Kirk_Klassen}, we conclude that
\begin{equation*}
  \CS(X_p;\trho_{h,k,l}^{\rm{Irr}})
  =
  2
  \int_{0}^{1}\beta(t)\frac{d\,\alpha(t)}{d\,t}dt
  =
  -\frac{h^2}{4(p-ab)}-\frac{(adl-bck)^2}{4ab},
\end{equation*}
completing the proof.
\end{proof}
\begin{remark}
Let us confirm that the right hand side of \eqref{eq:CS_hkl} does not depend on the choice of $(c,d)$ as an element in $\C/\Z$.
\par
Suppose that $c'$ and $d'$ are integers such that $ad'-bc'=1$.
Then we have
\begin{equation*}
\begin{split}
  &(adl-bck)^2-(ad'l-bc'k)^2
  \\
  =&
  a^2l^2(d^2-d'^2)+b^2k^2(c^2-c'^2)-2abkl(cd-c'd')
  \\
  =&
  al^2(d+d')b(c-c')+bk^2(c+c')a(d-d')-2abkl(cd-c'd')
  \\
  =&
  ab\bigl(l^2(d+d')(c-c')+k^2(c+c')(d-d')-2kl(cd-c'd')\bigr)
\end{split}
\end{equation*}
since $a(d-d')=b(c-c')$.
So $\frac{(adl-bck)^2}{4ab}-\frac{(ad'l-bc'k)}{4ab}$ is an integer because $k\equiv l\pmod{2}$ and the right hand side of \eqref{eq:CS_hkl} does not depend on the choice of $c$ and $d$.
\end{remark}

\section{$\mathcal{H}$ and $\mathcal{R}$}
In the previous subsection we show that the irreducible representations of $\pi_1(X_p)$ to $\SL(2;\C)$ are indexed by $\mathcal{H}$.
In this section we construct two injections $\tGamma_{+}$ and $\tGamma_{-}$ from $\mathcal{H}$ to $\mathcal{R}$ so that $\mathcal{R}=\tGamma_{+}(\mathcal{H})\sqcup\tGamma_{-}(\mathcal{H})$, where $\sqcup$ means a disjoint union.
\par
Let $a$ and $b$ positive coprime integers.
We assume that $b$ is odd.
\par
We first define two finite sets $\mathcal{P}$ and $\mathcal{Q}$ as follows:
\begin{align*}
  \mathcal{P}
  &:=
  \{(k,l)\mid1\le k\le a-1,1\le l\le b-1,k\modtwo l\},
  \\
  \mathcal{Q}
  &:=
  \{m\in\Z\mid1\le m\le ab-1,a\nmid m,b\nmid m\},
\end{align*}
where $k\modtwo l$ means $k\equiv l\pmod{2}$.
\par
Define maps $\Gamma_+$ and $\Gamma_-$ from $\mathcal{P}$ to $\mathcal{Q}$ by
\begin{equation*}
\begin{split}
  \Gamma_{+}(k,l)
  &:=
  [adl-bck]_{ab},
  \\
  \Gamma_{-}(k,l)
  &:=
  [-adl-bck]_{ab}.
\end{split}
\end{equation*}
Here we choose integers $c$ and $d$ such that $ad-bc=1$, and $[x]_p$ is the integer satisfying $0\le[x]_p<p$ and $x\equiv[x]_p\pmod{p}$.
\begin{remark}
By Sunzi's theorem \cite{Lam/Ang:SunZi}, $\Gamma_{\pm}(k,l)$ is characterized as follows:
\begin{align*}
  ([\Gamma_+(k,l)]_a,[\Gamma_+(k,l)]_b)&=(k,l),
  \\
  ([\Gamma_-(k,l)]_a,[\Gamma_-(k,l)]_b)&=(k,b-l)
\end{align*}
since $\Gamma_{+}(k,l)\equiv k\pmod{a}$ and $\Gamma_{+}(k,l)\equiv l\pmod{b}$, and $\Gamma_{-}(k,l)\equiv k\pmod{a}$ and $\Gamma_{-}(k,l)\equiv b-l\pmod{b}$.
\par
We also have
\begin{align}
  \Gamma_{+}([m]_{a},[m]_{b})
  &=
  m,
  \label{eq:characterization_Gamma_+}
  \\
  \Gamma_{-}([m]_{a},[-m]_{b})
  &=
  m\label{eq:characterization_Gamma_-}
\end{align}
for any integer $m$ with $0<m<ab-1$.
\end{remark}
Note that both $\Gamma_{+}$ and $\Gamma_{-}$ are injective, and that $\Gamma_{+}(\mathcal{P})\cap\Gamma_{-}(\mathcal{P})=\emptyset$.
The former follows from Sunzi's theorem.
The latter is because $[\Gamma_{+}(k,l)]_a\modtwo[\Gamma_{+}(k,l)]_b$ but $[\Gamma_{-}(k,l)]_a\not\modtwo[\Gamma_{-}(k,l)]_b$ since $b$ is odd.
\par
Since $\#\mathcal{P}=(a-1)(b-1)/2$ and $\#\mathcal{Q}=(a-1)(b-1)$, we conclude that
\begin{equation*}
  \mathcal{Q}=\mathcal{Q}_{+}\sqcup\mathcal{Q}_{-},
\end{equation*}
where $\sqcup$ denotes the disjoint union and $\mathcal{Q}_{\pm}:=\Gamma_{\pm}(\mathcal{P})$.
Note that
\begin{align*}
  \mathcal{Q}_{+}
  &=
  \left\{m\in\mathcal{Q}\Bigm|[m]_a\modtwo[m]_b\right\}.
  \\
  \mathcal{Q}_{-}
  &=
  \left\{m\in\mathcal{Q}\Bigm|[m]_a\nmodtwo[m]_b\right\}.
\end{align*}
\par
We also define maps $\Theta_{+}\colon\mathcal{Q}_{+}\to\mathcal{P}$ and $\Theta_{-}\colon\mathcal{Q}_{-}\to\mathcal{P}$by
\begin{equation*}
  \Theta_{\pm}(m)
  :=
  ([m]_a,[\pm m]_b)
\end{equation*}
\par
Then we have
\begin{equation*}
\begin{split}
  (\Theta_{\pm}\circ\Gamma_{\pm})(k,l)
  &=
  \Theta_{\pm}([\pm adl-bck]_{ab})
  =
  ([\pm adl-bck]_{a},[adl\mp bck)]_{b}
  =
  (k,l).
\end{split}
\end{equation*}
We also have
\begin{equation*}
\begin{split}
  (\Gamma_{+}\circ\Theta_{+})(m)
  &=
  \Gamma_{+}([m]_a,[m]_b)=m\quad\text{if $m\in\mathcal{Q}_{+}$},
  \\
  (\Gamma_{-}\circ\Theta_{-})(m)
  &=
  \Gamma_{-}([m]_a,[-m]_b)=m\quad\text{if $m\in\mathcal{Q}_{-}$}
\end{split}
\end{equation*}
from \eqref{eq:characterization_Gamma_+} and \eqref{eq:characterization_Gamma_-}.
So $\Theta_{\pm}$ is the inverse of $\Gamma_{+}$.
\begin{example}\label{ex:a_1_P_Q}
When $a=2$, we have
\begin{equation*}
  \mathcal{P}
  =
  \{(1,l)\mid1\le l\le b-1,\text{$l$: odd}\}.
\end{equation*}
and
\begin{align*}
  \mathcal{Q}_{+}
  &=
  \{1,3,\dots,b-2\},
  \\
  \mathcal{Q}_{-}
  &=
  \{b+2,b+4,\dots,2b-1\}.
\end{align*}
We also see that $\Gamma_{+}(1,l)=l$ and $\Gamma_{-}(1,l)=2b-l$, and that $\Theta_{+}(m)=(1,m)$ if $m\in\mathcal{Q}_{+}$ and $\Theta_{-}(m)=(1,2b-m)$ if $m\in\mathcal{Q}_{-}$.
\end{example}
\begin{example}
When $(a,b)=(4,3)$, we have
\begin{align*}
  \mathcal{P}
  &=
  \{(1,1),(2,2),(3,1)\},
  \\
  \mathcal{Q}_{+}
  &=
  \{1,2,7\}
  \\
  \mathcal{Q}_{-}
  &=
  \{5,10,11\},
\end{align*}
\begin{align*}
  \Gamma_{+}&\colon
  (1,1)\to1, (2,2)\to 2, (3,1)\to7,
  \\
  \Gamma_{-}&\colon
  (1,1)\to5, (2,2)\to10,(3,1)\to11,
\end{align*}
and
\begin{align*}
  \Theta_{+}\colon
  &1\to(1,1), 2\to(2,2), 7\to(3,1),
  \\
  \Theta_{-}\colon
  &5\to(1,1), 10\to(2,2),11\to(3,1).
\end{align*}
\end{example}
\begin{example}\label{ex:P_Q_3_5}
When $(a,b)=(3,5)$, we have
\begin{align*}
  \mathcal{P}
  &=
  \{(1,1),(1,3),(2,2),(2,4)\},
  \\
  \mathcal{Q}_{+}
  &=
  \{1,2,13,14\},
  \\
  \mathcal{Q}_{-}
  &=
  \{4,7,8,11\},
\end{align*}
\begin{align*}
  \Gamma_{+}
  &\colon
  (1,1)\to1,(1,3)\to13,(2,2)\to2,(2,4)\to14,
  \\
  \Gamma_{-}
  &\colon
  (1,1)\to4,(1,3)\to7,(2,2)\to8,(2,4)\to11,
\end{align*}
and
\begin{align*}
  \Theta_{+}
  \colon
  &1\to(1,1),2\to(2,2),13\to(1,3),14\to(2,4),
  \\
  \Theta_{-}
  \colon
  &4\to(1,1),7\to(1,3),8\to(2,2),11\colon(2,4).
\end{align*}
\end{example}
\par
Next we construct maps $\tGamma_{\pm}$ from $\mathcal{H}$ to $\mathcal{R}$.
\par
Recall the following:
\begin{align*}
  \mathcal{H}
  :=&
  \Bigl\{(h,k,l)\in\Z^3\mid0<h<p-ab,0<k<a,0<l<b, h\modtwo k\modtwo l\Bigr\},
  \\
  \mathcal{R}
  :=&
  \Bigl\{(g,m)\in\Z^2\mid\,0<m<ab,0<g<p-ab,g\modtwo p-m,a\nmid m,b\nmid m\Bigr\},
  \\
  \Gamma_{\pm}(k,l)
  :=&
  [\pm adl-bck]_{ab}.
\end{align*}
\par
We put
\begin{equation*}
  \tGamma_{\pm}(h,k,l)
  :=
  \begin{cases}
    (p-ab-h,\Gamma_{\pm}(k,l))
    &\text{if $\Gamma_{\pm}(k,l)+ab+h\modtwo0$},
    \\
    (p-ab-h,ab-\Gamma_{\pm}(k,l))
    &\text{if $\Gamma_{\pm}(k,l)+ab+h\nmodtwo0$}.
  \end{cases}
\end{equation*}
Note that if $a$ is even, we have
\begin{equation*}
\begin{split}
  \Gamma_{\pm}(k,l)+ab+h
  &=
  [\pm adl-bck]_{ab}+ab+h
  \\
  &=
  [bc(-k\pm l)\pm l]_{ab}+ab+h
  \\
  &\modtwo
  0
\end{split}
\end{equation*}
since $k\modtwo{l}\modtwo{h}$.
So in this case $\tGamma_{\pm}(h,k,l)=(p-ab-h,\Gamma_{\pm}(k,l))$.
\par
Let us check whether the image of $\tGamma_{\pm}$ is contained in $\mathcal{R}$ or not.
\par
First assume that $\Gamma_{\pm}(k,l)+ab+h\modtwo0$.
Then we have
\begin{equation*}
  (p-ab-h)-\tGamma_{\pm}(h,k,l)
  \modtwo
  p
\end{equation*}
and so the image is in $\mathcal{R}$.
Next assume that $\Gamma_{\pm}(k,l)+ab+h\nmodtwo0$.
Noting that $a$ should be odd, we have
\begin{equation*}
  (p-ab-h)-\bigl(ab-\tGamma_{\pm}(h,k,l)\bigr)
  \modtwo
  (p-ab-h)+(h+1)
  \modtwo
  p.
\end{equation*}
Therefore the image of $\tGamma_{\pm}$ is also contained in $\mathcal{R}$.
\par
Since $\Gamma_{\pm}(a-k,b-l)=[\pm ad(b-l)-bc(a-k)]_{ab}=[-(\pm adl-bck)]_{ab}=ab-\Gamma_{\pm}(k,l)$ and $\Gamma_{+}(\mathcal{P})\cap\Gamma_{-}(\mathcal{P})=\emptyset$, we see that $\tGamma_{+}(\mathcal{H})\cap\tGamma_{-}(\mathcal{H})=\emptyset$.
\par
We show that $\tGamma_{\pm}$ is injective.
Suppose that $\tGamma_{\pm}(h,k,l)=\tGamma_{\pm}(k',k',h')$ for $(h,k,l),(k',l',h')\in\mathcal{H}$.
Then from the definition we have $h=h'$.
\par
We also have either
\begin{enumerate}
\item[(i).]
$\Gamma_{\pm}(k,l)\underset{(2)}{\equiv}ab+h\underset{(2)}{\equiv}\Gamma_{\pm}(k',l')$ and $\Gamma_{\pm}(k,l)=\Gamma_{\pm}(k',l')$,
\item[(ii).]
$\Gamma_{\pm}(k,l)\underset{(2)}{\equiv}ab+h\underset{(2)}{\not\equiv}\Gamma_{\pm}(k',l')$ and $\Gamma_{\pm}(k,l)=ab-\Gamma_{\pm}(k',l')$,
\item[(iii).]
$\Gamma_{\pm}(k,l)\underset{(2)}{\not\equiv}ab+h\underset{(2)}{\equiv}\Gamma_{\pm}(k',l')$ and $ab-\Gamma_{\pm}(k,l)=\Gamma_{\pm}(k',l')$, or
\item[(iv).]
$\Gamma_{\pm}(k,l)\underset{(2)}{\not\equiv}ab+h\underset{(2)}{\not\equiv}\Gamma_{\pm}(k',l')$ and $\Gamma_{\pm}(k,l)=\Gamma_{\pm}(k',l')$.
\end{enumerate}
For the cases (i) and (iv), we have $(k,l)=(k',l')$ from the injectivity of $\Gamma_{\pm}$.
For the case (ii), since $ab-\Gamma_{\pm}(k',l')=\Gamma_{\pm}(a-k',b-l')$, we have $b=l+l'$, which contradicts the condition $l\modtwo h\modtwo l'$.
So the case (ii) does not happen.
Similarly, the case (iii) does not happen, either.
\par
Therefore we have $(h,k,l)=(k',l',h')$ and $\tGamma_{\pm}$ are injective.
\par
Noting that
\begin{equation*}
  \#\mathcal{H}
  =
  \begin{cases}
  \frac{1}{4}(a-1)(b-1)(p-ab-1)&
  \text{when $p$ is odd,}
  \\
  \frac{(a-1)(b-1)}{2}\lfloor\frac{p-ab-1}{2}\rfloor&
  \text{when $p$ is even.}
  \end{cases}
\end{equation*}
\begin{equation*}
  \#\mathcal{R}
  =
  \begin{cases}
  \frac{1}{2}(a-1)(b-1)(p-ab-1)&
  \text{when $p$ is odd,}
  \\
  (a-1)(b-1)\lfloor\frac{p-ab-1}{2}\rfloor&
  \text{when $p$ is even.}
  \end{cases}
\end{equation*}
we see that $\tGamma_{+}(\mathcal{H})\sqcup\tGamma_{-}(\mathcal{H})=\mathcal{R}$.
\par
We denote by $\Rpm$ the image of $\tGamma_{\pm}$.
\begin{lemma}
We can characterize $\Rpm$ as follows.
\begin{align*}
  \Rp
  =&
  \left\{(g,m)\in\mathcal{R}\Bigm|[m]_a\modtwo[m]_b\right\},
  \\
  \Rm
  =&
  \left\{g,m)\in\mathcal{R}\Bigm|[m]_a\nmodtwo[m]_b\right\}.
\end{align*}
\end{lemma}
\begin{proof}
It is sufficient to show that $[\Gamma_{+}(k,l)]_{a}\modtwo[\Gamma_{+}(k,l)]_{b}$, and that $[\Gamma_{-}(k,l)]_{a}\nmodtwo[\Gamma_{-}(k,l)]_{b}$.
\par
Since $\Gamma_{\pm}(k,l)=[\pm adl-bck]_{ab}$ and $ad-bc=1$, we have
\begin{align*}
  [\Gamma_{\pm}(k,l)]_{a}
  &=
  [k]_{a}
  =
  k,
  \\
  [\Gamma_{+}(k,l)]_{b}
  &=
  [l]_{b}
  =
  l,
  \\
  [\Gamma_{-}(k,l)]_{b}
  &=
  [-l]_{b}
  =
  b-l.
\end{align*}
The conclusion follows since $b$ is odd and $k\modtwo l$.
\end{proof}
We define maps $\tTheta_{+}\colon\Rp\to\mathcal{H}$ and $\tTheta_{-}\colon\Rm\to\mathcal{H}$ as follows:
\begin{align*}
  \tTheta_{\pm}(g,m)
  &:=
  (p-ab-g,[m]_{a},[\pm m]_{b})
  &\quad\text{if $ab+m+[m]_{a}\modtwo0$},
  \\
  \tTheta_{\pm}(g,m)
  &:=
  (p-ab-g,[-m]_{a},[\mp m]_{b})
  &\quad\text{if $ab+m+[m]_{a}\nmodtwo0$}.
\end{align*}
Note that if $a$ is even, then $[m]_{a}\modtwo m$ and so $ab+m+[m]_{a}\modtwo0$.
As a result, if $ab+m+[m]_{a}\nmodtwo0$, then $a$ is odd.
\par
We need to check that $\tTheta_{\pm}$ is a map to $\mathcal{H}$.
\par
Assume that $ab+m+[m]_{a}\modtwo0$.
Then, since $(g,m)\in\mathcal{R}$, we have
\begin{equation*}
  (p-ab-g)-[m]_{a}
  \modtwo
  p+g+m
  \modtwo
  0.
\end{equation*}
If $(g,m)\in\Rp$, then $[m]_{a}\modtwo[m]_{b}$ and so $\tTheta_{+}(g,m)\in\mathcal{H}$.
If $(g,m)\in\Rm$, then $[m]_{a}\nmodtwo[m]_{b}$.
Since $[-m]_{b}=b-[m]_{b}\nmodtwo[m]_{b}$, $[m]_{a}\modtwo[-m]_{b}$ and so $\tTheta_{+}(g,m)\in\mathcal{H}$.
\par
Assume that $ab+m+[m]_{a}\nmodtwo0$.
Since $a$ is odd as mentioned before, we have
\begin{equation*}
  (p-ab-g)-[-m]_{a}
  =
  (p-ab-g)-(a-[m]_{a})
  \modtwo
  p+m+g
  \modtwo0.
\end{equation*}
If $(g,m)\in\Rp$, then $[-m]_{a}=a-[m]_{a}\modtwo b-[m]_{b}\modtwo [-m]_{b}$ and so $\tTheta_{+}(g,m)\in\mathcal{H}$.
If $(g,m)\in\Rm$, then $[-m]_{a}=a-[m]_{a}\nmodtwo b-[m]_{b}\nmodtwo [m]_{b}$ and so $\tTheta_{-}(g,m)\in\mathcal{H}$.
Therefore the image of $\tTheta_{\pm}$ is in $\mathcal{H}$.
\par
Next we show that $\tTheta_{\pm}$ is the inverse of $\tGamma_{\pm}$.
\par
When $\Gamma_{\pm}(k,l)+ab+h\modtwo0$, we have
\begin{equation*}
\begin{split}
  (\tTheta_{\pm}\circ\tGamma_{\pm})(h,k,l)
  &=
  \tTheta_{\pm}(p-ab-h,\Gamma_{\pm}(k,l))
  \\
  &=
  (h,[\Gamma_{\pm}(k,l)]_{a},[\pm\Gamma_{\pm}(k,l)]_{b})
  \\
  &=
  (h,k,l)
\end{split}
\end{equation*}
since $ab+\Gamma_{\pm}(k,l)+[\Gamma_{\pm}(k,l)]_{a}=ab+\Gamma_{\pm}(k,l)+k\modtwo0$.
When $\Gamma_{\pm}(k,l)+ab+h\nmodtwo0$, we have
\begin{equation*}
\begin{split}
  (\tTheta_{\pm}\circ\tGamma_{\pm})(h,k,l)
  &=
  \tTheta_{\pm}(p-ab-h,\Gamma_{\pm}(a-k,b-l))
  \\
  &=
  (h,\left[-\Gamma_{\pm}(a-k,b-l)\right]_{a},\left[\mp\Gamma_{\pm}(a-k,b-l)\right]_{b})
  \\
  &=
  (h,[\mp ad(b-l)+bc(a-k)]_{a},[-ad(b-l)\pm bc(a-k)]_{b})
  \\
  &=
  (h,k,l).
\end{split}
\end{equation*}
Here the second equality follows since $ab+\Gamma_{\pm}(a-k,b-l)+[\Gamma_{\pm}(a-k,b-l)]_a\modtwo\Gamma_{\pm}(k,l)+b-l\nmodtwo ab+h+b+l\modtwo0$.
Therefore $\tTheta_{\pm}\circ\tGamma_{\pm}$ is the identity on $\mathcal{H}$.
\par
When $(g,m)\in\Rp$ and $ab+m+[m]_{a}\modtwo0$, we have
\begin{equation*}
  (\tGamma_{+}\circ\tTheta_{+})(g,m)
  =
  \tGamma_{+}(p-ab-g,[m]_{a},[m]_{b}).
\end{equation*}
Since $\Gamma_{+}([m]_{a},[m]_{b})=m$ from \eqref{eq:characterization_Gamma_+}, and $m+ab+(p-ab-g)\modtwo0$, we have
\begin{equation*}
  \tGamma_{+}(p-ab-g,[m]_{a},[m]_{b})
  =
  (g,\Gamma_{+}([m]_{a},[m]_{b}))
  =
  (g,m).
\end{equation*}
When $(g,m)\in\Rp$ and $ab+m+[m]_{a}\nmodtwo0$, we have
\begin{equation*}
  (\tGamma_{+}\circ\tTheta_{+})(g,m)
  =
  \tGamma_{+}(p-ab-g,[-m]_{a},[-m]_{b}).
\end{equation*}
Since $\Gamma_{+}([-m]_{a},[-m]_{b})=\Gamma_{+}([ab-m]_{a},[ab-m]_{b})=ab-m$ and $(ab-m)+ab+(p-ab-g)\modtwo ab\nmodtwo0$, we have
\begin{equation*}
\begin{split}
  \tGamma_{+}(p-ab-g,[-m]_{a},[-m]_{b})
  &=
  (g,ab-\Gamma_{+}([-m]_{a},[-m]_{b}))
  \\
  &=
  (g,ab-\Gamma_{+}([ab-m]_{a},[ab-m]_{b}))
  \\
  &=
  (g,m).
\end{split}
\end{equation*}
Therefore $\tGamma_{+}\circ\tTheta_{+}$ is the identity on $\Rp$.
\par
When $(g,m)\in\Rm$ and $ab+m+[m]_{a}\modtwo0$, we have
\begin{equation*}
  (\tGamma_{-}\circ\tTheta_{-})(g,m)
  =
  \tGamma_{-}(p-ab-g,[m]_{a},[-m]_{b}).
\end{equation*}
Since $\Gamma_{-}([m]_{a},[-m]_{b})=m$ from \eqref{eq:characterization_Gamma_-}, and $m+ab+(p-ab-g)\modtwo0$, we have
\begin{equation*}
  \tGamma_{-}(p-ab-g,[m]_{a},[-m]_{b})
  =
  (g,\Gamma_{-}([m]_{a},[-m]_{b}))
  =
  (g,m).
\end{equation*}
When $(g,m)\in\Rm$ and $ab+m+[m]_{a}\nmodtwo0$, we have
\begin{equation*}
  (\tGamma_{-}\circ\tTheta_{-})(g,m)
  =
  \tGamma_{-}(p-ab-g,[-m]_{a},[m]_{b}).
\end{equation*}
Since $\Gamma_{-}([-m]_{a},[m]_{b})=\Gamma_{-}([ab-m]_{a},[m-ab]_{b})=ab-m$, and $(ab-m)+ab+(p-ab-g)\modtwo ab\nmodtwo0$, we have
\begin{equation*}
  \tGamma_{-}(p-ab-g,[-m]_{a},[m]_{b})
  =
  (g,ab-\tGamma_{-}([-m]_{a},[m]_{b}))
  =
  (g,m).
\end{equation*}
Therefore $\tGamma_{\pm}\circ\tTheta_{-}$ is the identity on $\Rm$.
\begin{example}\label{ex:H_R_Gamma}.
When $a=2$, $b$ and $p$ are odd from the assumption.
So we have
\begin{align*}
  \mathcal{H}
  &=
  \{(h,1,l)\mid h=1,3,\dots,p-2b-2,l=1,3,\dots,b-2\},
  \\
  \mathcal{R}
  &=
  \{(g,m)\mid g=2,4,\dots,p-2b-1,m=1,3,\dots,b-2,b+2,\dots,2b-1\},
  \\
  \Rp
  &=
  \{(g,m)\mid g=2,4,\dots,p-2b-1,m=1,3,\dots,b-2\},
  \\
  \Rm
  &=
  \{(g,m)\mid g=2,4,\dots,p-2b-1,m=b+2,b+4,\dots,2b-1\}.
\end{align*}
We can put $d:=(1-b)/2$ and $c:=-1$.
So we have
\begin{align*}
  \tGamma_{+}(h,1,l)
  &=
  (p-2b-h,[(1-b)l+b]_{2b})
  =
  (p-2b-h,l),
  \\
  \tGamma_{-}(h,1,l)
  &=
  (p-2b-h,[-(1-b)l+b]_{2b})
  =
  (p-2b-h,2b-l),
\end{align*}
since $l$ is odd.
\par
We also have
\begin{equation*}
  \tTheta_{+}(g,m)
  =
  (p-2b-g,1,m)
\end{equation*}
when $(g,m)\in\Rp$ and
\begin{equation*}
  \tTheta_{-}(g,m)
  =
  (p-2b-g,1,2b-m)
\end{equation*}
when $(g,m)\in\Rm$.
\end{example}
\begin{example}\label{H_R_Gamma_Theta_3_5}
When $(a,b)=(3,5)$ and $p=19$, we have
\begin{align*}
  \mathcal{H}
  &=
  \{(1,1,1),(1,1,3),(2,2,2),(2,2,4),(3,1,1),(3,1,3)\},
  \\
  \mathcal{R}
  &=
  \{(1,2),(1,4),(1,8),(1,14),(2,1),(2,7),(2,11),(2,13),(3,2),(3,4),(3,8),(3,14)\},
  \\
  \Rp
  &=
  \{(1,2),(1,14),(2,1),(2,13),(3,2),(3,14)\},
  \\
  \Rm
  &=
  \{(1,4),(1,8),(2,7),(2,11),(3,4),(3,8)\}.
\end{align*}
From Example~\ref{ex:P_Q_3_5}, we have
\begin{align*}
  \tGamma_{+}
  &\colon
  (1,1,1)\to(3,14),
  (1,1,3)\to(3,2),
  (2,2,2)\to(2,13),
  (2,2,4)\to(2,1),
  \\
  &\phantom{\colon}
  (3,1,1)\to(1,14),
  (3,1,3)\to(1,2),
  \\
  \tGamma_{-}
  &\colon
  (1,1,1)\to(3,4),
  (1,1,3)\to(3,8),
  (2,2,2)\to(2,7),
  (2,2,4)\to(2,11),
  \\
  &\phantom{\colon}
  (3,1,1)\to(1,4),
  (3,1,3)\to(1,8),
  \\
  \tTheta_{+}
  &\colon
  (1,2)\to(3,1,3),
  (1,14)\to(3,1,1),
  (2,1)\to(2,2,4),
  (2,13)\to(2,2,2),
  \\
  &\phantom{\colon}
  (3,2)\to(1,1,3),
  (3,14)\to(1,1,1),
  \\
  \tTheta_{-}
  &\colon
  (1,4)\to(3,1,1),
  (1,8)\to(3,1,3),
  (2,7)\to(2,2,2),
  (2,11)\to(2,2,4),
  \\
  &\phantom{\colon}
  (3,4)\to(1,1,1),
  (3,8)\to(1,1,3).
\end{align*}
\end{example}

\section{Topological Interpretations of $A(n)$ and $B(n)$}\label{sec:A_n}
As shown in \eqref{eq:H}, the irreducible representations of $\pi_1(X_p)$ to $\SL(2;\C)$ are indexed by $\mathcal{H}$.
\subsection{Topological interpretation of $A(n)$}
In this subsection, we give a topological interpretation of $A(n)$.
\par
Let $\tGamma_{\pm}\colon\mathcal{H}\to\Rpm$ be the bijections described in the previous section.
Recall the following subsets of $\mathcal{R}$ (see Figure~\ref{fig:R}):
\begin{align*}
  \mathcal{R}
  =&
  \left\{
    (g,m)\in\Z^2
    \mid
    0<m<ab,0<g<p-ab,
    g\modtwo{p-m},a\nmid m,b\nmid m
  \right\},
  \\
  \RD
  =&
  \left\{(g,m)\in\mathcal{R}\Bigm|\frac{m}{ab}<\frac{g}{p-ab}\right\},
  \\
  \RN
  =&
  \left\{(g,m)\in\mathcal{R}\Bigm|\frac{m}{ab}>\frac{g}{p-ab}\right\},
  \\
  \Rp
  =&
  \left\{(g,m)\in\mathcal{R}\Bigm|[m]_a\modtwo[m]_b\right\},
  \\
  \Rm
  =&
  \left\{(g,m)\in\mathcal{R}\Bigm|[m]_a\nmodtwo[m]_b\right\}.
\end{align*}
\begin{figure}[H]
  \raisebox{-0.5\height}{\includegraphics[scale=0.4]{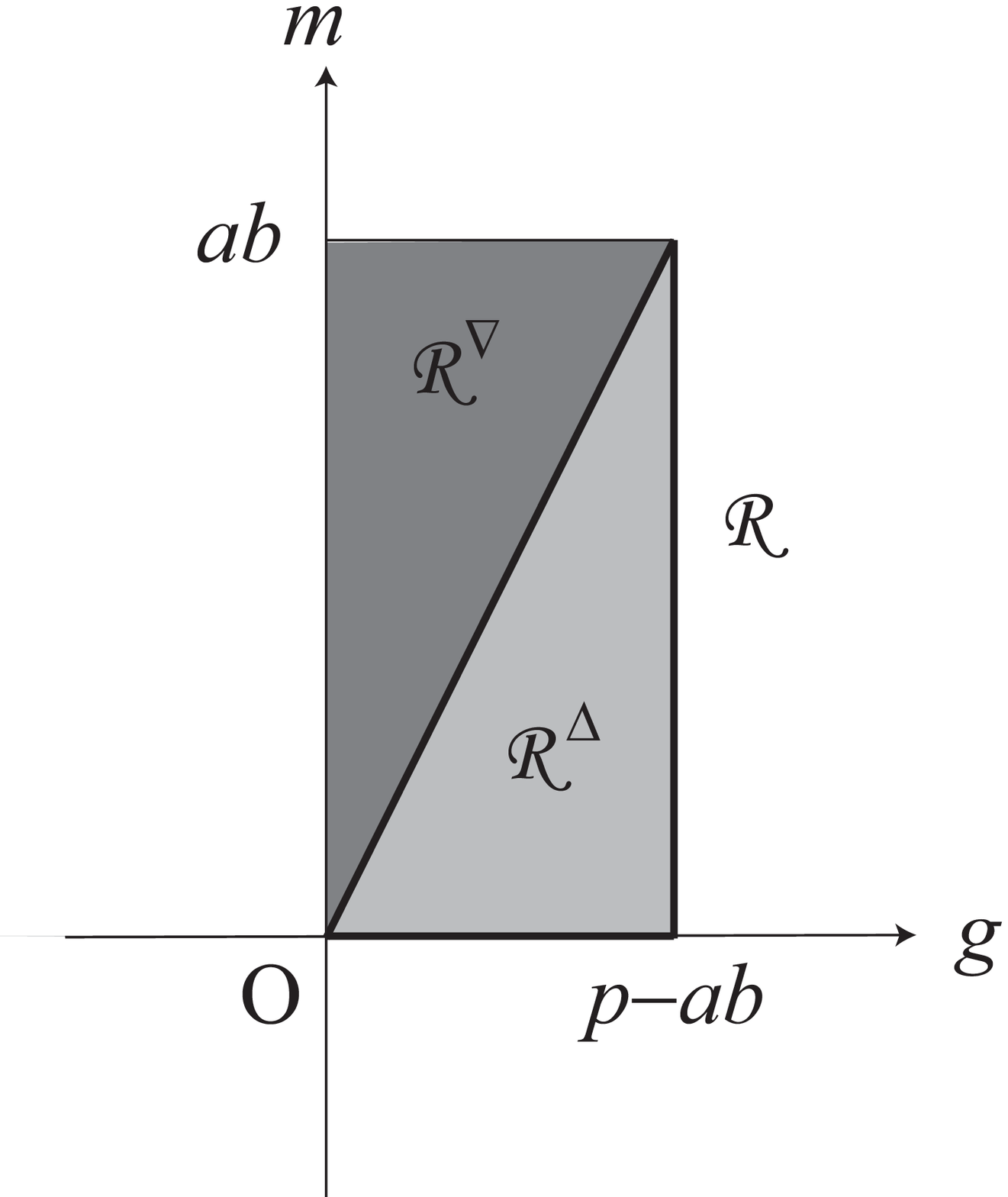}}
  \caption{The dark gray area is $\RN$, the light gray area is $\RD$,
  and $\mathcal{R}=\RD\cup\RN$.}
  \label{fig:R}
\end{figure}
\par
Put $\RpmD:=\Rpm\cap\RD$, $\RpmN:=\Rpm\cap\RN$, $\HpmD:=\tTheta_{\pm}(\RpmD)$ and $\HpmN:=\tTheta_{\pm}(\RpmN)$.
Then we have the two decompositions of $\mathcal{H}$ and $\mathcal{R}$:
\begin{align*}
  \mathcal{H}
  &=
  \HpD\sqcup\HpN,
  \\
  \mathcal{H}
  &=
  \HmD\sqcup\HmN,
  \\
  \mathcal{R}
  &=
  \RpD\sqcup\RpN\sqcup
  \RmD\sqcup\RmN.
\end{align*}
\begin{remark}
Note that $\HpmD\cap\HpmN=\emptyset$ but that $\HpD\cap\HmD\ne\emptyset$ and $\HmD\cap\HpD\ne\emptyset$ in general.
\end{remark}
\begin{remark}
Since an element $(h,k,l)$ in $\mathcal{H}$ belongs to $\HpmD$ ($\HpmN$, respectively) if and only if $\tGamma_{\pm}(h,k,l)\in\RD$ ($\tGamma_{\pm}(h,k,l)\in\RN$, respectively).
Therefore, concretely speaking, the sets $\HpmD$ and $\HpmN$ are given as follows.
\begin{align*}
  \HpmD
  &=
  \Bigl\{
    (h,k,l)\in\mathcal{H}\Bigm|
    \text{$\frac{\Gamma_{\pm}(k,l)}{ab}+\frac{h}{p-ab}<1$
     (if $\Gamma_{\pm}(k,l)+ab+h\modtwo0)$},
  \\
  &\phantom{=\Bigl\{(h,k,l)\in\mathcal{H}\Bigm|}\quad
    \text{$\frac{\Gamma_{\pm}(k,l)}{ab}>\frac{h}{p-ab}$
     (if $\Gamma_{\pm}(k,l)+ab+h\modtwo1)$}
  \Bigr\},
  \\
  \HpmN
  &=
  \Bigl\{
    (h,k,l)\in\mathcal{H}\Bigm|
    \text{$\frac{\Gamma_{\pm}(k,l)}{ab}+\frac{h}{p-ab}>1$
     (if $\Gamma_{\pm}(k,l)+ab+h\modtwo0)$},
  \\
  &\phantom{=\Bigl\{(h,k,l)\in\mathcal{H}\Bigm|}\quad
    \text{$\frac{\Gamma_{\pm}(k,l)}{ab}<\frac{h}{p-ab}$
     (if $\Gamma_{\pm}(k,l)+ab+h\modtwo1)$}
  \Bigr\}.
\end{align*}
\end{remark}
\begin{example}\label{ex:H}
Suppose that $a=2$.
Then we have
\begin{align*}
  \Rp
  &=
  \{(g,m)\mid h=2,4,\dots,p-2b-1,m=1,3,\dots,b-2\},
  \\
  \Rm
  &=
  \{(g,m)\mid h=2,4,\dots,p-2b-1,m=b+2,b+4,\dots,2b-1\}
\end{align*}
and so
\begin{align*}
  \RpD
  &=
  \left\{
    (g,m)\mid
    2\le h\le p-2b-1,1\le m\le b-2,\frac{h}{p-2b}>\frac{m}{2b},
    h\modtwo0,m\modtwo1
  \right\},
  \\
  \RmD
  &=
  \left\{
    (g,m)\mid
    2\le h\le p-2b-1,b+2\le m\le 2b-1,\frac{h}{p-2b}>\frac{m}{2b},
    h\modtwo0,m\modtwo1
  \right\},
  \\
  \RpN
  &=
  \left\{
    (g,m)\mid
    2\le h\le p-2b-1,m\le l\le b-2,\frac{h}{p-2b}<\frac{m}{2b},
    h\modtwo0,m\modtwo1
  \right\},
  \\
  \RmN
  &=
  \left\{
    (g,m)\mid
    2\le h\le p-2b-1,b+2\le m\le 2b-1,\frac{h}{p-2b}>\frac{m}{2b},
    h\modtwo0,m\modtwo1
  \right\}.
\end{align*}
See Figure~\ref{fig:R_a_2}.
\begin{figure}[H]
  \raisebox{-0.5\height}{\includegraphics[scale=0.4]{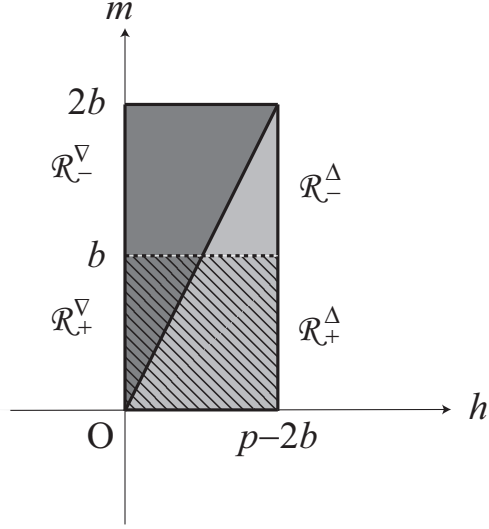}}
  \caption{The decomposition of $\mathcal{R}$ when $a=2$ ($k=1$).}
  \label{fig:R_a_2}
\end{figure}
Therefore from Example~\ref{ex:H_R_Gamma}, we have
\begin{align*}
  \HpD
  &=
  \left\{
    (h,1,l)\mid
    1\le h\le p-2b-2,1\le l\le b-2,\frac{h}{p-2b}+\frac{l}{2b}<1,
    h\modtwo1,l\modtwo1
  \right\},
  \\
  \HpN
  &=
  \left\{
    (h,1,l)\mid
    1\le h\le p-2b-2,1\le l\le b-2,\frac{h}{p-2b}+\frac{l}{2b}>1,
    h\modtwo1,l\modtwo1
  \right\},
  \\
  \HmD
  &=
  \left\{
    (h,1,l)\mid
    1\le h\le p-2b-2,1\le l\le b-2,\frac{h}{p-2b}<\frac{l}{2b},
    h\modtwo1,l\modtwo1
  \right\},
  \\
  \HmN
  &=
  \left\{
    (h,1,l)\mid
    1\le h\le p-2b-2,1\le l\le b-2,\frac{h}{p-2b}>\frac{l}{2b},
    h\modtwo1,l\modtwo1
  \right\}
\end{align*}
as indicated in Figures~\ref{fig:H}.
\begin{figure}[H]
  \raisebox{-0.5\height}{\includegraphics[scale=0.4]{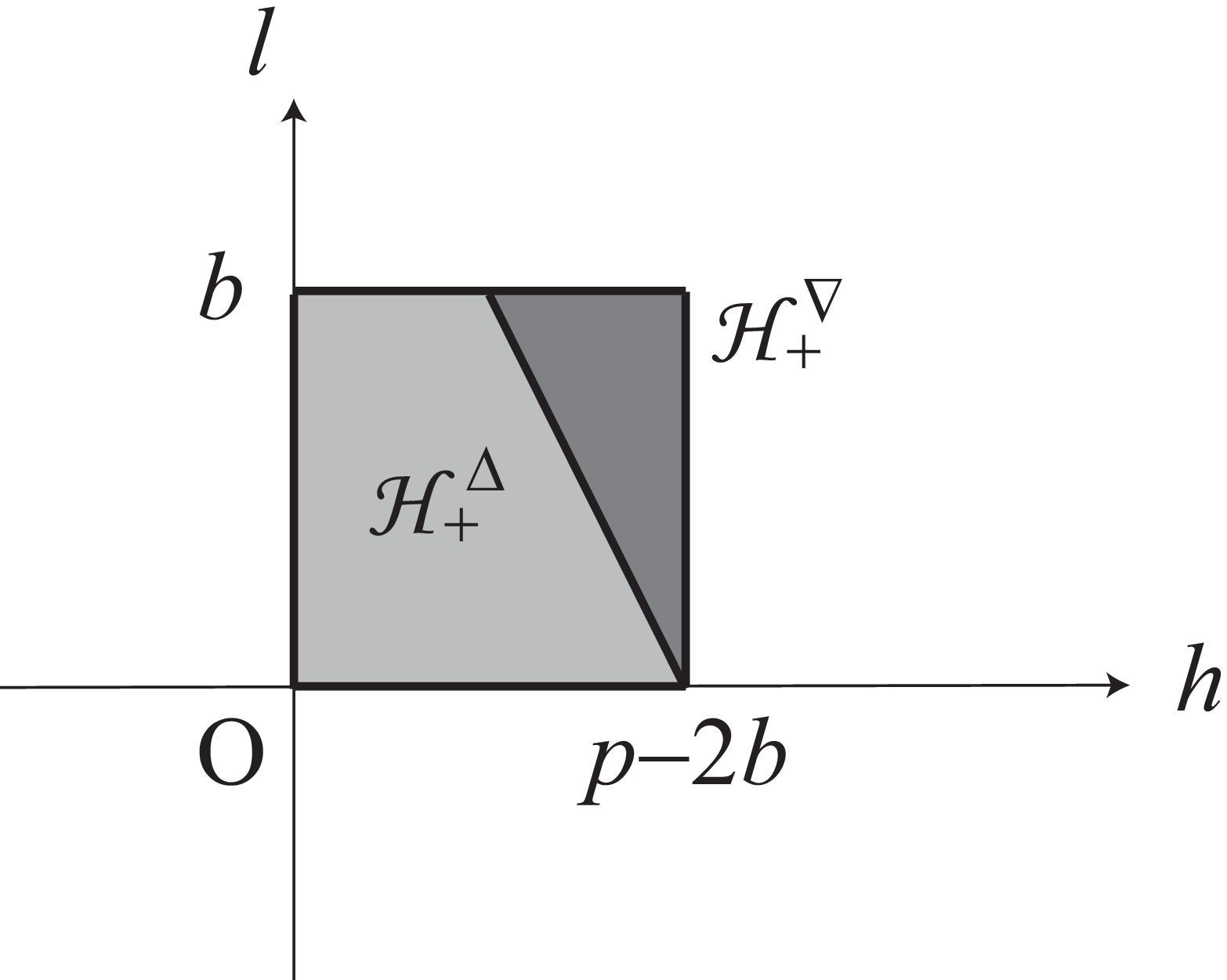}}
  \quad
  \raisebox{-0.5\height}{\includegraphics[scale=0.4]{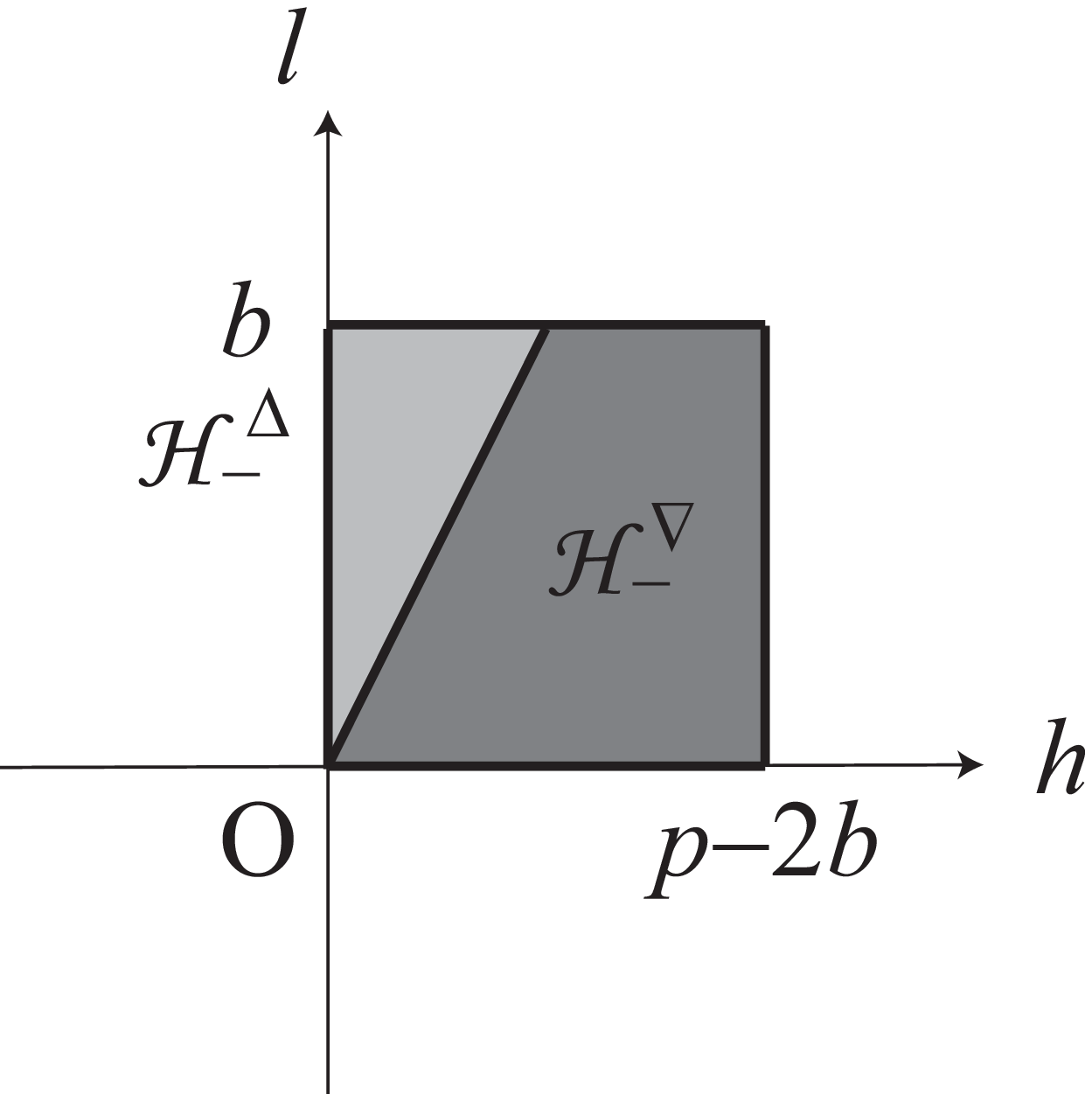}}
  \caption{Two decompositions of $\mathcal{H}$.}
  \label{fig:H}
\end{figure}
Note that $\HmD\subset\HpD$ and that $\HpN\subset\HmN$.
\end{example}
\begin{example}
If $(a,b)=(3,5)$ and $p=19$, we have
\begin{align*}
  \RD
  &=
  \{(1,2),(2,1),(3,2),(3,4),(3,8)\}
  \\
  \RN
  &=
  \{(1,4),(1,8),(1,14),(2,7),(2,11),(2,13),(3,14)\},
\end{align*}
\begin{align*}
  \RpD
  &=
  \{(1,2),(2,1),(3,2)\},
  \\
  \RpN
  &=
  \{(1,14),(2,13),(3,14)\},
  \\
  \RmD
  &=
  \{(3,4),(3,8)\},
  \\
  \RmN
  &=
  \{(1,4),(1,8),(2,7),(2,11)\},
\end{align*}
and
\begin{align*}
  \HpD
  &=
  \{(3,1,3),(2,2,4),(1,1,3)\},
  \\
  \HpN
  &=
  \{(3,1,1),(2,2,2),(1,1,1)\},
  \\
  \HmD
  &=
  \{(1,1,1),(1,1,3)\},
  \\
  \HmN
  &=
  \{(3,1,1),(3,1,3),(2,2,2),(2,2,4)\},
\end{align*}
from \ref{H_R_Gamma_Theta_3_5}.
\par
Note that $\HpD\cap\HmD=\{(1,1,3)\}$ and that $\HpN\cap\HmN=\{(3,1,1),(2,2,2)\}$.
\end{example}
\par
Since $\mathcal{R}=\RD\sqcup\RN$, $\RD=\RpD\sqcup\RmD$, and $\RN=\RpN\sqcup\RmN$, from \eqref{eq:G_sum_R} and the condition that $g\modtwo p-m$ when $(g,m)\in\mathcal{R}$, we have
\begin{equation*}
\begin{split}
  &e^{-\frac{(n+1)\pi\i}{4}}\sqrt{ab(p-ab)}A(n)
  \\
  =&
  \sum_{(g,m)\in\RD}G(g,m)
  -
  \sum_{(g,m)\in\RN}G(g,m)
  -
  \sum_{(g,m)\in\mathcal{R}}(-1)^{ab+p+g}G(g,m)
  \\
  =&
  \sum_{(g,m)\in\RpD}G(g,m)
  +
  \sum_{(g,m)\in\RmD}G(g,m)
  -
  \sum_{(g,m)\in\RpN}G(g,m)
  -
  \sum_{(g,m)\in\RmN}G(g,m)
  \\
  &-
  \sum_{(g,m)\in\RpD}(-1)^{ab+p+g}G(g,m)
  -
  \sum_{(g,m)\in\RmD}(-1)^{ab+p+g}G(g,m)
  \\
  &-
  \sum_{(g,m)\in\RpN}(-1)^{ab+p+g}G(g,m)
  -
  \sum_{(g,m)\in\RmN}(-1)^{ab+p+g}G(g,m)
  \\
  =&
  \sum_{(g,m)\in\RpD}
  \left(1-(-1)^{ab+p+g}\right)G(g,m)
  +
  \sum_{(g,m)\in\RmD}
  \left(1-(-1)^{ab+p+g}\right)G(g,m)
  \\
  &-
  \sum_{(g,m)\in\RpN}
  \left(1+(-1)^{ab+p+g}\right)G(g,m)
  -
  \sum_{(g,m)\in\RmN}
  \left(1+(-1)^{ab+p+g}\right)G(g,m)
  \\
  =&
  \sum_{(h,k,l)\in\HpD}
  \left(1-(-1)^{h}\right)G(\tGamma_{+}(h,k,l))
  +
  \sum_{(h,k,l)\in\HmD}
  \left(1-(-1)^{h}\right)G(\tGamma_{-}(h,k,l))
  \\
  &-
  \sum_{(h,k,l)\in\HpN}
  \left(1+(-1)^{h}\right)G(\tGamma_{+}(h,k,l))
  -
  \sum_{(h,k,l)\in\HmN}
  \left(1+(-1)^{h}\right)G(\tGamma_{-}(h,k,l)),
\end{split}
\end{equation*}
since $ab+p+g\modtwo h$ if $(g,m)=\tGamma_{\pm}(h,k,l)$.
\par
Therefore we have
\begin{equation*}
\begin{split}
  &e^{-\frac{(n+1)\pi\i}{4}}\sqrt{ab(p-ab)}A(n)
  \\
  =&
  2\sum_{\substack{(h,k,l)\in\HpD\\ h\modtwo k\modtwo l\modtwo1}}
  G(\tGamma_{+}(h,k,l))
  +
  2\sum_{\substack{(h,k,l)\in\HmD\\ h\modtwo k\modtwo l\modtwo1}}
  G(\tGamma_{-}(h,k,l))
  \\
  &-
  2\sum_{\substack{(h,k,l)\in\HpN\\ h\modtwo k\modtwo l\modtwo0}}
  G(\tGamma_{+}(h,k,l))
  -
  2\sum_{\substack{(h,k,l)\in\HmN\\ h\modtwo k\modtwo l\modtwo0}}
  G(\tGamma_{-}(h,k,l)).
\end{split}
\end{equation*}
\par
Now we define $\ChernSimons_{\pm}(h,k,l)$ and $\Reidemeister_{\pm}(h,k,l)$ as follows:
\begin{align*}
  \ChernSimons_{\pm}^{\rm{Irr}}(h,k,l)
  &:=
  \begin{cases}
    -\frac{(p-ab-h)^2}{4(p-ab)}-\frac{\Gamma_{\pm}(k,l)^2}{4ab}
    &\quad\text{if $\Gamma_{\pm}(k,l)+ab+h\modtwo0$,}
    \\
    -\frac{(p-ab-h)^2}{4(p-ab)}-\frac{(ab-\Gamma_{\pm}(k,l))^2}{4ab}
    &\quad\text{if $\Gamma_{\pm}(k,l)+ab+h\modtwo1$,}
  \end{cases}
  \\
  \Reidemeister_{\pm}^{\rm{Irr}}(h,k,l)
  &:=
  \begin{cases}
    (-1)^{\Gamma_{\pm}(k,l)}
    \frac{8\sin\left(\frac{\Gamma_{\pm}(k,l)\pi}{a}\right)
           \sin\left(\frac{\Gamma_{\pm}(k,l)\pi}{b}\right)
           \sin\left(\frac{(p-ab-h)\pi}{p-ab}\right)}
         {\sqrt{ab(p-ab)}}
    &\quad\text{if $\Gamma_{\pm}(k,l)+ab+h\modtwo0$,}
    \\
    (-1)^{ab-\Gamma_{\pm}(k,l)}
    \frac{8\sin\left(\frac{(ab-\Gamma_{\pm}(k,l))\pi}{a}\right)
           \sin\left(\frac{(ab-\Gamma_{\pm}(k,l))\pi}{b}\right)
           \sin\left(\frac{(p-ab-h)\pi}{p-ab}\right)}
         {\sqrt{ab(p-ab)}}
    &\quad\text{if $\Gamma_{\pm}(k,l)+ab+h\modtwo1$.}
  \end{cases}
\end{align*}
Then from \eqref{eq:G_def} we have
\begin{equation}\label{eq:A_H}
\begin{split}
  &e^{-\frac{(n+1)\pi\i}{4}}A(n)
  \\
  =&
  \frac{1}{4}\sum_{\substack{(h,k,l)\in\HpD\\ h\modtwo k\modtwo l\modtwo1}}
  \Reidemeister_{+}^{\rm{Irr}}(h,k,l)e^{n\ChernSimons_{+}^{\rm{Irr}}(h,k,l)\pi\i}
  +
  \frac{1}{4}\sum_{\substack{(h,k,l)\in\HmD\\ h\modtwo k\modtwo l\modtwo1}}
  \Reidemeister_{-}^{\rm{Irr}}(h,k,l)e^{n\ChernSimons_{-}^{\rm{Irr}}(h,k,l)\pi\i}
  \\
  &-
  \frac{1}{4}\sum_{\substack{(h,k,l)\in\HpN\\ h\modtwo k\modtwo l\modtwo0}}
  \Reidemeister_{+}^{\rm{Irr}}(h,k,l)e^{n\ChernSimons_{+}^{\rm{Irr}}(h,k,l)\pi\i}
  -
  \frac{1}{4}\sum_{\substack{(h,k,l)\in\HmN\\ h\modtwo k\modtwo l\modtwo0}}
  \Reidemeister_{-}^{\rm{Irr}}(h,k,l)e^{n\ChernSimons_{-}^{\rm{Irr}}(h,k,l)\pi\i}.
\end{split}
\end{equation}
\par
\begin{lemma}\label{lem:Reidemeister_X_p_irr}
We have
\begin{equation*}
  \Reidemeister_{\pm}^{\rm{Irr}}(h,k,l)^{-2}
  =
  \left|
    \mathbb{T}(X_{p};\trho_{p-ab-h,a-k,b-l}^{\rm{Irr}})
  \right|.
\end{equation*}
\end{lemma}
\begin{proof}
We will show
\begin{equation*}
  \Reidemeister_{\pm}^{\rm{Irr}}(h,k,l)
  =
  \pm\sin\left(\frac{k\pi}{a}\right)\sin\left(\frac{l\pi}{b}\right)
  \sin\left(\frac{h\pi}{p-ab}\right).
\end{equation*}
It is sufficient to prove
\begin{equation*}
  \sin\left(\frac{\Gamma_{\pm}(k,l)\pi}{a}\right)
  \sin\left(\frac{\Gamma_{\pm}(k,l)\pi}{b}\right)
  =
  \pm
  \sin\left(\frac{k\pi}{a}\right)
  \sin\left(\frac{l\pi}{b}\right).
\end{equation*}
Since $\Gamma_{\pm}(k,l)=\pm adl-bck\equiv -bck\equiv-k\pmod{a}$ and $\Gamma_{\pm}(k,l)\equiv adl\equiv l\pmod{b}$, the equality above holds.
\end{proof}
\begin{lemma}\label{lem:ChernSimons_X_p_irr}
We have
\begin{equation*}
  \ChernSimons(h,k,l)
  \equiv
  \CS(X_p;\trho_{p-ab-h,a-k,b-l}^{\rm{Irr}})
  \in\C/(2\Z).
\end{equation*}
\end{lemma}
\begin{proof}
We will show that
\begin{equation*}
  (ad(b-l)-bc(a-k))^2
  \equiv
  \begin{cases}
    \Gamma_{\pm}(k,l)^2
    &\pmod{2ab}
    \quad\text{if $\Gamma_{\pm}(k,l)+ab+h\modtwo0$,}
    \\
    (ab-\Gamma_{\pm}(k,l))^2
    &\pmod{2ab}
    \quad\text{if $\Gamma_{\pm}(k,l)+ab+h\modtwo1$.}
  \end{cases}
\end{equation*}
\begin{itemize}
\item
$\Gamma_{\pm}(k,l)+ab+h\modtwo0$.
\par
Put $u_{\pm}:=\Gamma_{\pm}(k,l)$.
Then there exists $v_{\pm}\in\Z$ such that $\pm adl-bck=abv_{\pm}+u_{\pm}$.
So we have
\begin{equation*}
\begin{split}
  \Gamma_{+}(k,l)^2-(ad(b-l)-bc(a-k))^2
  &=
  (adl-bck-abv_{+})^2-(ad(b-l)-bc(a-k))^2
  \\
  &=
  ab(d-c-v_{+})(2adl-abd+abc-2bck-abv_{+}).
  \\
  &\equiv
  ab(d-c-v_{+})(abd+abc+abv_{+})
  \pmod{2ab}.
\end{split}
\end{equation*}
However since $abv_{+}=adl-bck-u_{+}\modtwo(ad-bc)l+\Gamma_{+}(k,l)\modtwo l+ab+h\modtwo ab$, this is congruent to
\begin{equation*}
  ab(d-c-v_{+})(abd+abc+ab)
  =
  a^2b^2(d-c-v_{+})(d+c+1).
\end{equation*}
If $a$ is even, this is congruent to $0$ modulo $2ab$.
If $a$ is odd, then $d+c+1$ is congruent to $0$ modulo $2$ since $ad-bc=1$ and so this is also congruent to $0$ modulo $2ab$.
\par
We also have
\begin{equation*}
\begin{split}
  \Gamma_{-}(k,l)^2-(ad(b-l)-bc(a-k))^2
  &=
  (-adl-bck-abv_{-})^2-(ad(b-l)-bc(a-k))^2
  \\
  &=
  (abd-2adl-abc-abv_{-})(-abd+abc-2bck-abv_{-})
  \\
  &=
  ab(bd-2dl-bc-bv_{-})(-ad+ac-2ck-av_{-})
  \\
  &\equiv
  ab(bd-bc-bv_{-})(-ad+ac-av_{-})
  \\
  &\equiv
  (abd-abc-abv_{-})(-abd+abc-abv_{-})
  \\
  &\equiv0
  \pmod{2ab},
\end{split}
\end{equation*}
where the last congruence follows by the same reason as above.
\par
Therefore we conclude that $\Gamma_{\pm}(k,l)^2\equiv(ad(b-l)-bc(a-k))^2\pmod{2ab}$.
\item
$\Gamma_{\pm}(k,l)+ab+h\modtwo1$.
\par
Let $u_{\pm}$ and $v_{\pm}$ be as above.
\par
We have
\begin{equation*}
\begin{split}
  &(ab-\Gamma_{+}(k,l))^2-(ad(b-l)-bc(a-k))^2
  \\
  =&
  (ab-adl+bck+abv_{+})^2-(ad(b-l)-bc(a-k))^2
  \\
  =&
  (ab+abd-2adl-abc+2bck+abv_{+})(ab-abd+abc+abv_{+})
  \\
  =&
  ab(ab+abd-2adl-abc+2bck+abv_{+})(1-d+c+v_{+})
  \\
  =&
  ab(ab+abd-abc+abv_{+})(1-d+c+v_{+})
  \pmod{2ab}.
\end{split}
\end{equation*}
Now since $abv_{+}=adl-bck-u_{+}\modtwo ab+1$, this is congruent to
\begin{equation*}
  ab(ab+abd-abc+ab+1)(1-d+c+v_{+})
  \equiv
  a^2b^2(c+d+1)(1-d+c+v_{+})
  \pmod{2ab},
\end{equation*}
which is congruent by the same reason as above.
\par
We also have
\begin{equation*}
\begin{split}
  &\Gamma_{-}(k,l)^2-(ad(b-l)-bc(a-k))^2
  \\
  =&
  (ab+adl+bck+abv_{-})^2-(ad(b-l)-bc(a-k))^2
  \\
  =&
  (ab+abd+2bck-abc+abv_{-})(ab-abd+2adl+abc+abv_{-})
  \\
  =&
  ab(a+ad+2ck-ac+av_{-})(b-bd+2dl+bc+bv_{-})
  \\
  &\equiv
  ab(a+ad-ac+av_{-})(b-bd+bc+bv_{-})
  \\
  &=
  ab(1+d-c+v_{-})(ab-abd+abc+abv_{-})
  \\
  &\equiv0
  \pmod{2ab}
\end{split}
\end{equation*}
by the same reason as above.
\par
Therefore we conclude that $\Gamma_{\pm}(k,l)^2\equiv(ad(b-l)-bc(a-k))^2\pmod{2ab}$.
\end{itemize}
The proof is complete.
\end{proof}
From Lemmas~\ref{lem:Reidemeister_X_p_irr} and \ref{lem:ChernSimons_X_p_irr}, we have topological interpretations of the terms in the right hand side of \eqref{eq:A_H}.
\begin{example}\label{ex:a_2_A_B}
Suppose that $a=2$.
If $(h,k,l)\in\mathcal{H}$, then $k=1$, and $l$ and $h$ are odd.
So the last two terms in \eqref{eq:A_H} vanish.
\par
Moreover, from Example~\ref{ex:H}, we see that $\HmD\subset\HpD$ and that $\HpN\subset\HmN$.
So \eqref{eq:A_H} becomes
\begin{equation*}
\begin{split}
  &e^{-\frac{(n+1)\pi\i}{4}}A(n)
  \\
  =&
  \frac{1}{4}\sum_{(h,1,l)\in\HpD}
  \Reidemeister_{+}^{\rm{Irr}}(h,1,l)e^{n\ChernSimons_{+}^{\rm{Irr}}(h,1,l)\pi\i}
  +
  \frac{1}{4}\sum_{(h,1,l)\in\HmD}
  \Reidemeister_{-}^{\rm{Irr}}(h,1,l)e^{n\ChernSimons_{-}^{\rm{Irr}}(h,1,l)\pi\i}
\end{split}
\end{equation*}
\par
From Example~\ref{ex:a_1_P_Q}, $\Gamma_{+}(1,l)=l$ and $\Gamma_{-}(1,l)=2b-l$.
Note that $\Gamma_{\pm}(k,l)+ab+h\equiv0$ in this case.
So we have
\begin{align*}
  \ChernSimons_{+}^{\rm{Irr}}(h,1,l)
  &=
  -\frac{(p-2b-h)^2}{4(p-2b)}-\frac{l^2}{8b},
  \\
  \ChernSimons_{-}^{\rm{Irr}}(h,1,l)
  &=
  -\frac{(p-2b-h)^2}{4(p-2b)}-\frac{(2b-l)^2}{8b}
  =
  \ChernSimons_{+}^{\rm{Irr}}(h,1,l)+\frac{b-l}{2}
\end{align*}
from Example~\ref{ex:H_R_Gamma}.
Note that $(b-l)/2\in\Z$.
We also have
\begin{align*}
  \Reidemeister_{+}^{\rm{Irr}}(h,1,l)
  &=
  -
  \frac{8\sin\left(\frac{l}{2}\pi\right)\sin\left(\frac{l}{b}\pi\right)
         \sin\left(\frac{p-2b-h}{p-2b}\pi\right)}
       {\sqrt{2b(p-2b)}},
  \\
  \Reidemeister_{-}^{\rm{Irr}}(h,1,l)
  &=
  -
  \frac{8\sin\left(\frac{2b-l}{2}\pi\right)\sin\left(\frac{2b-l}{b}\pi\right)
         \sin\left(\frac{p-2b-h}{p-2b}\pi\right)}
       {\sqrt{2b(p-2b)}}
  \\
  &=
  -\Reidemeister_{+}^{\rm{Irr}}(h,1,l).
\end{align*}
So we have
\begin{equation*}
\begin{split}
  &e^{-\frac{(n+1)\pi\i}{4}}A(n)
  \\
  =&
  \frac{1}{4}\sum_{(h,1,l)\in\HpD}
  \Reidemeister_{+}^{\rm{Irr}}(h,1,l)e^{n\ChernSimons_{+}^{\rm{Irr}}(h,1,l)\pi\i}
  -
  \frac{1}{4}\sum_{(h,1,l)\in\HmD}
  (-1)^{(b-l)/2}
  \Reidemeister_{+}^{\rm{Irr}}(h,1,l)e^{n\ChernSimons_{+}^{\rm{Irr}}(h,1,l)\pi\i}
  \\
  =&
  \frac{1}{2}\sum_{\substack{(h,1,l)\in\HmD\\ b-l\equiv2\pmod4}}
  \Reidemeister_{+}^{\rm{Irr}}(h,1,l)e^{n\ChernSimons_{+}^{\rm{Irr}}(h,1,l)\pi\i}
  +
  \frac{1}{4}\sum_{(h,1,l)\in\HpD\setminus\HmD}
  \Reidemeister_{+}^{\rm{Irr}}(h,1,l)e^{n\ChernSimons_{+}^{\rm{Irr}}(h,1,l)\pi\i}
\end{split}
\end{equation*}
since $\HmD\subset\HpD$.
\end{example}
\subsection{Topological interpretation of $B(n)$}
In this subsection we give a topological interpretation of $B(n)$.
\par
First note that $H_1(M_p;\Z)=\Z/p\Z$ and so a reducible Abelian representation $\sigma_l$ of $\pi_1(M_p)$ is characterized as $\tr\sigma_l=\cos(2l\pi/p)$ with $1<l<(p-1)/2$.
\par
Put
\begin{align*}
  \Reidemeister^{\rm{Abel}}(l)
  &:=
  (-1)^l
  \frac{4\sin\left(\frac{2al\pi}{p}\right)\sin\left(\frac{2bl\pi}{p}\right)
        \sin\left(\frac{2l\pi}{p}\right)}
        {\sqrt{p}\sin\left(\frac{2abl\pi}{p}\right)},
  \\
  \ChernSimons^{\rm{Abel}}(l)
  &:=
  -\frac{l^2}{p}.
\end{align*}
Then we can write $B(n)$ as follows:
\begin{equation}\label{eq:B}
  B(n)
  =
  \frac{1}{2}\i(-1)^{a+b+ab}e^{n(1-p)\pi\i/4}
  \sum_{0<l<(p-1)/2}
  \Reidemeister^{\rm{Abel}}(l)e^{n\ChernSimons^{\rm{Abel}}(l)\pi\i}.
\end{equation}
Note that
\begin{align}\label{RCS_Abel}
  \left(\Reidemeister^{\rm{Abel}}(l)\right)^{-2}
  &=
  \left|\Tor(X_{p};\trho^{\rm{Abel}}_{l})\right|,
  \\
  \ChernSimons^{\rm{Abel}}(l)
  &=
  \CS(X_{p};\trho^{\rm{Abel}}_{l})
  \quad
  \pmod\Z.
\end{align}
\par
From \eqref{eq:tau_A_B}, \eqref{eq:A_H}, \eqref{eq:B}, \eqref{RCS_Abel}, and Lemmas~\ref{lem:Reidemeister_X_p_irr} and \ref{lem:ChernSimons_X_p_irr} we have the following theorem.
\begin{theorem}\label{thm:main}
The Witten--Reshetikhin-Turaev invariant of $X_p$ evaluated at $e^{4\pi\i/n}$ has the following asymptotic expansion.
\begin{equation*}
  \htau_n(X_p;\exp(4\pi\i/n))
  =
  \frac{(-1)^{p+1}n^{3/2}}{2\pi}
  \left(A(n)+B(n)n^{-1/2}+O(n^{-1})\right)
\end{equation*}
with
\begin{equation*}
\begin{split}
  A(n)
  =&
  2e^{\frac{n+1}{4}\pi\sqrt{-1}}
  \left(
    \sum_{\substack{(h,k,l)\in\mathcal{H}_{+}^{\Delta}\\ h\modtwo k\modtwo l\modtwo1}}
    -
    \sum_{\substack{(h,k,l)\in\mathcal{H}_{+}^{\nabla}\\ h\modtwo k\modtwo l\modtwo0}}
  \right)
  \Reidemeister_{+}^{\rm{Irr}}(h,k,l)e^{n\ChernSimons_{+}^{\rm{Irr}}(h,k,l)\pi\i}
  \\
  &+
  2e^{\frac{n+1}{4}\pi\sqrt{-1}}
  \left(
    \sum_{\substack{(h,k,l)\in\mathcal{H}_{-}^{\Delta}\\ h\modtwo k\modtwo l\modtwo1}}
    -
    \sum_{\substack{(h,k,l)\in\mathcal{H}_{-}^{\nabla}\\ h\modtwo k\modtwo l\modtwo0}}
  \right)
  \Reidemeister_{-}^{\rm{Irr}}(h,k,l)e^{n\ChernSimons_{-}^{\rm{Irr}}(h,k,l)\pi\i}
\end{split}
\end{equation*}
and
\begin{equation*}
  B(n)
  =
  \frac{1}{2}\i(-1)^{a+b+ab}e^{n(1-p)\pi\i/4}
  \sum_{0<l<(p-1)/2}
  \Reidemeister^{\rm{Abel}}(l)e^{n\ChernSimons^{\rm{Abel}}(l)\pi\i},
\end{equation*}
where $\Reidemeister_{\pm}^{\rm{Irr}}(h,k,l)$ and $\Reidemeister^{\rm{Abel}}(l)$ are related to the twisted Reidemeister torsions, and $\ChernSimons_{\pm}^{\rm{Irr}}(h,k,l)$ and $\ChernSimons^{\rm{Abel}}(l)$ are related to the Chern--Simons invariant as described above.
\end{theorem}
\subsection{$\SU(2)$ representations and $\SU(1,1)$ representations}
Motivated by \cite[Theorem~1.3]{Ohtsuki/Takata::COMMP2019} (see Theorem~\ref{thm:Ohtsuki_Takata}), we will study when a given irreducible representation $\rho\colon\pi_1(X_p)\to\SL(2;\C)$ is an $\SU(2)$ or an $\SU(1,1)$ representation.
As a result, at least in the case where $a=2$, we show that $A(n)$ in Theorem~\ref{thm:main} can be written in terms of these representations.
\par
A $2\times2$ complex matrix $L$ is in $\SU(2)$ if and only if $L^{\ast}\,L=I_2$ and $\det{L}=1$, where $L^{\ast}$ is the conjugate transpose of $L$.
Note that $\SU(2)=\left\{\begin{pmatrix}u&v\\-\overline{v}&\overline{u}\end{pmatrix}\Bigm|u,v\in\C,|u|^2+|v|^2=1\right\}$, where $\overline{u}$ is the complex conjugate of $u$.
A $2\times2$ complex matrix $M$ is in $\SU(1,1)$ if and only if $M^{\ast}\begin{pmatrix}1&0\\0&-1\end{pmatrix}M=\begin{pmatrix}1&0\\0&-1\end{pmatrix}$ and $\det{M}=1$.
Note that $\SU(1,1)=\left\{\begin{pmatrix}u&v\\\overline{v}&\overline{u}\end{pmatrix}\mid u,v\in\C,|u|^2-|v|^2=1\right\}$.
The Cayley map defined by $M\mapsto\mathcal{C}M\mathcal{C}^{-1}$ gives an isomorphism between $\SU(1,1)$ and $\SL(2;\R)$, where $\mathcal{C}:=\frac{1}{\sqrt{2}}\begin{pmatrix}1&\i\\\i&1\end{pmatrix}\in\SL(2;\C)$.
\begin{proposition}\label{prop:SU2_SL2R}
An irreducible representation $\trho_{h,k,l}^{\rm{Irr}}\to\SL(2;\C)$ is either an $\SU(2)$ representation or an $\SU(1,1)$ representation.
Moreover it is an $\SU(2)$ representation if and only if
\begin{equation}\label{eq:SU(2)_rep}
  \left(
    \cos\left(\frac{h\pi}{p-ab}\right)
    -
    \cos\left(\frac{(adl+bck)\pi}{ab}\right)
  \right)
  \left(
    \cos\left(\frac{h\pi}{p-ab}\right)
    -
    \cos\left(\frac{(adl-bck)\pi}{ab}\right)
  \right)
  >0,
\end{equation}
and it is an $\SU(1,1)$ representation if and only if
\begin{equation}\label{eq:SU(1,1)_rep}
  \left(
    \cos\left(\frac{h\pi}{p-ab}\right)
    -
    \cos\left(\frac{(adl+bck)\pi}{ab}\right)
  \right)
  \left(
    \cos\left(\frac{h\pi}{p-ab}\right)
    -
    \cos\left(\frac{(adl-bck)\pi}{ab}\right)
  \right)
  <0.
\end{equation}
\end{proposition}
To prove the proposition, we prepare a lemma.
\begin{lemma}\label{lem:Chebyshev}
For a matrix $M\in\SL(2;\C)$ and an integer $m$, we have
\begin{equation*}
  M^{n}=S_{n-1}(\tr{M})M-S_{n-2}(\tr{M})I_2,
\end{equation*}
where $S_n(z)$ is the $n$-th Chebyshev polynomial defined by $S_0(z)=1$, $S_1(z)=z$, and $S_n(z)=zS_{n-1}-S_{n-2}(z)$ and $\tr$ is the trace.
\end{lemma}
\begin{proof}
By the Cayley--Hamilton theorem, $M^2=(\tr{M})M-I_2$.
The lemma follows easily by induction.
\end{proof}
Note that $S_n(2\cos\theta)=\frac{\sin\bigl((n+1)\theta\bigr)}{\sin\theta}$ for $\theta\in\R$.
\par
Now we prove Proposition~\ref{prop:SU2_SL2R}.
\begin{proof}[Proof of Proposition~\ref{prop:SU2_SL2R}]
In this proof we use $\rho$ instead of $\trho_{h,k,l}^{\rm{Irr}}$ for short.
Recall the following from Proposition~\ref{prop:irr_rep}:
\begin{align*}
  \tr\rho(x)&=2\cos\left(\frac{k\pi}{a}\right),
  \\
  \tr\rho(y)&=2\cos\left(\frac{l\pi}{b}\right),
  \\
  \tr\rho(\mu)&=2\cos\left(\frac{h\pi}{p-ab}\right).
\end{align*}
\par
First, we show that $\rho$ is a representation to $\SU(2)$ if and only if \eqref{eq:SU(2)_rep} holds.
Suppose that the image of $\rho$ is in $\SU(2)\subset\SL(2;\C)$ .
Since a unitary matrix is diagonalizable, we may assume that $\rho(x)=\begin{pmatrix}e^{k\pi\i/a}&0\\0&e^{-k\pi\i/a}\end{pmatrix}$ up to conjugation.
Note that $\rho(y)$ is conjugate to $\begin{pmatrix}e^{l\pi\i/b}&0\\0&e^{-l\pi\i/b}\end{pmatrix}$.
Write $\rho(y^d)=\begin{pmatrix}u&v\\-\overline{v}&\overline{u}\end{pmatrix}$ with $|u|^2+|v|^2=1$ and $v\ne0$ since $\rho$ is irreducible.
Since $u+\overline{u}=\tr\rho(y^d)=2\cos(dl\pi/b)$, the real part of $u$ equals $\cos(dl\pi/b)$.
So we can put $u=\cos(dl\pi/b)+r\i$ for some $r\in\R$.
Now we have
\begin{equation*}
\begin{split}
  \tr\rho(\mu)
  &=
  \tr\rho(x^{-c}y^d)
  =
  \tr
  \left(
    \begin{pmatrix}
      e^{-ck\pi\i/a}&0 \\
      0            &e^{ck\pi\i/a}
    \end{pmatrix}
    \begin{pmatrix}
      u            &v\\
      -\overline{v}&\overline{u}
    \end{pmatrix}
  \right)
  =
  \tr
  \begin{pmatrix}
    ue^{-ck\pi\i/a}           &ve^{-ck\pi\i/a} \\
    -\overline{v}e^{ck\pi\i/a}&\overline{u}e^{ck\pi\i/a}
  \end{pmatrix}
  \\
  &=
  (\cos(dl\pi/b)+r\i)(\cos(ck\pi/a)-\sin(ck\pi/a)\i)
  \\
  &\quad
  +
  (\cos(dl\pi/b)-r\i)(\cos(ck\pi/a)+\sin(ck\pi/a)\i)
  \\
  &=
  2\cos(dl\pi/b)\cos(ck\pi/a)+2r\sin(ck\pi/a).
\end{split}
\end{equation*}
Since $\tr\rho(u)=2\cos\bigl(h\pi/(p-ab)\bigr)$, we obtain
\begin{equation}\label{eq:r}
  r
  =
  \frac{\cos\bigl(h\pi/(p-ab)\bigr)-\cos(dl\pi/b)\cos(ck\pi/a)}{\sin(ck\pi/a)}.
\end{equation}
Since $|u|^2+|v|^2=1$ and $v\ne0$, we have $|u|^2<1$ and so $\cos^2(dl\pi/b)+r^2<1$.
Hence we have
\begin{equation*}
  \cos^2(dl\pi/b)
  +
  \frac{\left(\cos\bigl(h\pi/(p-ab)\bigr)-\cos(dl\pi/b)\cos(ck\pi/a)\right)^2}{\sin^2(ck\pi/a)}
  <1,
\end{equation*}
which means that
\begin{equation*}
\begin{split}
  0&>
  -\sin^2(ck\pi/a)+\sin^2(ck\pi/a)\cos^2(dl\pi/b)
  +\left(\cos\bigl(h\pi/(p-ab)\bigr)-\cos(dl\pi/b)\cos(ck\pi/a)\right)^2
  \\
  &=
  \left(\cos\bigl(h\pi/(p-ab)\bigr)-\cos(dl\pi/b)\cos(ck\pi/a)\right)^2
  +
  \sin^2(ck\pi/a)\sin^2(dl\pi/b)
  \\
  &=
  \left(\cos\bigl(h\pi/(p-ab)\bigr)-\cos(dl\pi/b+ck\pi/a)\right)
  \left(\cos\bigl(h\pi/(p-ab)\bigr)-\cos(dl\pi/b-ck\pi/a)\right).
\end{split}
\end{equation*}
Therefore if $\rho$ is a representation to $\SU(2)$, then \eqref{eq:SU(2)_rep} holds true.
\par
Conversely, suppose that \eqref{eq:SU(2)_rep} is satisfied.
With $r$ given by \eqref{eq:r} and $u:=\cos(dl\pi/b)+r\i$, we have $|u|^2<1$.
So there exists $v\in\C$ such that $|u|^2+|v|^2=1$.
We now show that there exists $\rho(y)\in\SU(2)$ such that $\rho(y^d)=\begin{pmatrix}u&v\\-\overline{v}&\overline{u}\end{pmatrix}$.
\par
By Lemma~\ref{lem:Chebyshev}, we have
\begin{equation}\label{eq:y^d_SU(2)}
  \rho(y^d)
  =
  S_{d-1}\bigl(2\cos(l\pi/b)\bigr)\rho(y)-S_{d-2}\bigl(2\cos(l\pi/b)\bigr)I_2.
\end{equation}
Since $S_{d-1}\bigl(2\cos(l\pi/b)\bigr)=\frac{\sin(dl\pi/b)}{\sin(l\pi/b)}\ne0$, we put
\begin{equation}\label{eq:y_SU(2)}
  \rho(y)
  :=
  \frac{1}{S_{d-1}\bigl(2\cos(l\pi/b)\bigr)}
  \begin{pmatrix}
    u+S_{d-2}\bigl(2\cos(l\pi/b)\bigr)&v \\
    -\overline{v}                     &\overline{u}+S_{d-2}\bigl(2\cos(l\pi/b)\bigr)
  \end{pmatrix}.
\end{equation}
Then from \eqref{eq:y^d_SU(2)}
\begin{equation*}
  \rho(y^d)
  =
  \begin{pmatrix}
    u+S_{d-2}\bigl(2\cos(l\pi/b)\bigr)&v \\
    -\overline{v}                     &\overline{u}+S_{d-2}\bigl(2\cos(l\pi/b)\bigr)
  \end{pmatrix}
  -
  S_{d-2}\bigl(2\cos(l\pi/b)\bigr)I_2
  =
  \begin{pmatrix}u&v\\-\overline{v}&\overline{u}\end{pmatrix}.
\end{equation*}
We need to show that the right hand side of \eqref{eq:y_SU(2)} is an element in $\SU(2)$.
It will be sufficient if we can show that
\begin{equation*}
  |v|^2+\left|u+S_{d-2}\bigl(2\cos(l\pi/b)\bigr)\right|^2
  =
  \left(S_{d-1}\bigl(2\cos(l\pi/b)\bigr)\right)^2.
\end{equation*}
Put $\theta:=l\pi/b$.
Since $|u|^2+|v|^2=1$ and $u+\overline{u}=2\cos(d\theta)$, we have
\begin{equation*}
\begin{split}
  &|v|^2+\left|u+S_{d-2}\bigl(2\cos(l\pi/b)\bigr)\right|^2
  -
  \left(S_{d-1}\bigl(2\cos(l\pi/b)\bigr)\right)^2
  \\
  =&
  |v|^2
  +
  \left(u+S_{d-2}\bigl(2\cos\theta\bigr)\right)
  \left(\overline{u}+S_{d-2}\bigl(2\cos\theta\bigr)\right)
  -
  \left(S_{d-1}\bigl(2\cos(\theta)\bigr)\right)^2
  \\
  =&
  |v|^2+|u|^2
  +(u+\overline{u})S_{d-2}\bigl(2\cos\theta\bigr)
  +
  \left(S_{d-2}\bigl(2\cos\theta\bigr)\right)^2
  -
  \left(S_{d-1}\bigl(2\cos(\theta)\bigr)\right)^2
  \\
  =&
  1+2\cos(d\theta)\frac{\sin\bigl((d-1)\theta\bigr)}{\sin\theta}
  +
  \left(\frac{\sin\bigl((d-1)\theta\bigr)}{\sin\theta}\right)^2
  -
  \left(\frac{\sin(d\theta)}{\sin\theta}\right)^2
  \\
  =&
  1+2\cos(d\theta)\left(\frac{\sin(d\theta)}{\sin\theta}\cos\theta-\cos(d\theta)\right)
  +
  \left(\frac{\sin(d\theta)}{\sin\theta}\cos\theta-\cos(d\theta)\right)^2
  -
  \left(\frac{\sin(d\theta)}{\sin\theta}\right)^2
  \\
  =&
  1
  -
  2\cos^2(d\theta)
  +
  \left(\frac{\sin(d\theta)}{\sin\theta}\cos(d\theta)\right)^2
  +
  \cos^2(d\theta)
  -
  \left(\frac{\sin(d\theta)}{\sin\theta}\right)^2
  \\
  =&
  1
  -
  \cos^2(d\theta)
  -
  \left(\frac{\sin(d\theta)}{\sin\theta}\right)^2\left(1-\cos^2\theta\right)
  \\
  =&
  0.
\end{split}
\end{equation*}
\par
Next, we show that $\rho$ is a representation to $\SU(1,1)$ if and only if \eqref{eq:SU(1,1)_rep} holds.
\par
Suppose that the image of $\rho$ is in $\SU(1,1)\subset\SL(2;\C)$.
Since the eigevalues of $\rho(x)$ are $e^{\pm k\pi\i/a}$, we may assume that $\rho(x)=\begin{pmatrix}e^{k\pi\i/a}&0\\0&e^{-k\pi\i/a}\end{pmatrix}$ up to conjugation.
By the same reason, $\rho(y)$ is conjugate to $\begin{pmatrix}e^{l\pi\i/b}&0\\0&e^{-l\pi\i/b}\end{pmatrix}$.
Write $\rho(y^d)=\begin{pmatrix}u&v\\\overline{v}&\overline{u}\end{pmatrix}$ with $|u|^2-|v|^2=1$ and $v\ne0$.
Since $u+\overline{u}=\tr\rho(y^d)=2\cos(dl\pi/b)$, the real part of $u$ equals $\cos(dl\pi/b)$.
So we can put $u=\cos(dl\pi/b)+r\i$ for some $r\in\R$.
Now we have
\begin{equation*}
\begin{split}
  \tr\rho(u)
  &=
  \tr\rho(x^{-c}y^d)
  =
  \tr
  \left(
    \begin{pmatrix}
      e^{-ck\pi\i/a}&0 \\
    0             &e^{ck\pi\i/a}
    \end{pmatrix}
    \begin{pmatrix}
      u           &v\\
      \overline{v}&\overline{u}
    \end{pmatrix}
  \right)
  =
  \tr
  \begin{pmatrix}
    ue^{-ck\pi\i/a}          &ve^{-ck\pi\i/a} \\
    \overline{v}e^{ck\pi\i/a}&\overline{u}e^{ck\pi\i/a}
  \end{pmatrix}
  \\
  &=
  (\cos(dl\pi/b)+r\i)(\cos(ck\pi/a)-\sin(ck\pi/a)\i)
  \\
  &\quad+
  (\cos(dl\pi/b)-r\i)(\cos(ck\pi/a)+\sin(ck\pi/a)\i)
  \\
  &=
  2\cos(dl\pi/b)\cos(ck\pi/a)+2r\sin(ck\pi/a).
\end{split}
\end{equation*}
Since $\tr\rho(u)=2\cos\bigl(h\pi/(p-ab)\bigr)$, we obtain
\begin{equation*}
  r
  =
  \frac{\cos\bigl(h\pi/(p-ab)\bigr)-\cos(dl\pi/b)\cos(ck\pi/a)}{\sin(ck\pi/a)}.
\end{equation*}
Since $|u|^2-|v|^2=1$ and $v\ne0$, we have $|u|^2>1$ and so $\cos^2(dl\pi/b)+r^2>1$.
In the same way as above, we can prove \eqref{eq:SU(1,1)_rep}.
\par
Conversely, suppose that \eqref{eq:SU(1,1)_rep} is satisfied.
With $r$ given by \eqref{eq:r} and $u:=\cos(dl\pi/b)+r\i$, we have $|u|^2>1$.
So there exists $v\in\C$ such that $|u|^2-|v|^2=1$.
We now show that there exists $\rho(y)\in\SU(1,1)$ such that $\rho(y^d)=\begin{pmatrix}u&v\\\overline{v}&\overline{u}\end{pmatrix}$.
\par
By Lemma~\ref{lem:Chebyshev}, we have
\begin{equation}\label{eq:y^d_SU(1,1)}
  \rho(y^d)
  =
  S_{d-1}\bigl(2\cos(l\pi/b)\bigr)\rho(y)-S_{d-2}\bigl(2\cos(l\pi/b)\bigr)I_2.
\end{equation}
Since $S_{d-1}\bigl(2\cos(l\pi/b)\bigr)=\frac{\sin(dl\pi/b)}{\sin(l\pi/b)}\ne0$, we put
\begin{equation}\label{eq:y_SU(1,1)}
  \rho(y)
  :=
  \frac{1}{S_{d-1}\bigl(2\cos(l\pi/b)\bigr)}
  \begin{pmatrix}
    u+S_{d-2}\bigl(2\cos(l\pi/b)\bigr)&v \\
    \overline{v}                      &\overline{u}+S_{d-2}\bigl(2\cos(l\pi/b)\bigr)
  \end{pmatrix}.
\end{equation}
Then from \eqref{eq:y^d_SU(1,1)}
\begin{equation*}
  \rho(y^d)
  =
  \begin{pmatrix}
    u+S_{d-2}\bigl(2\cos(l\pi/b)\bigr)&v \\
    \overline{v}                     &\overline{u}+S_{d-2}\bigl(2\cos(l\pi/b)\bigr)
  \end{pmatrix}
  -
  S_{d-2}\bigl(2\cos(l\pi/b)\bigr)I_2
  =
  \begin{pmatrix}u&v\\\overline{v}&\overline{u}\end{pmatrix}.
\end{equation*}
We need to show that the right hand side of \eqref{eq:y_SU(1,1)} is an element in $\SU(1,1)$.
It will be sufficient if we can show that
\begin{equation*}
  -|v|^2+\left|u+S_{d-2}\bigl(2\cos(l\pi/b)\bigr)\right|^2
  =
  \left(S_{d-1}\bigl(2\cos(l\pi/b)\bigr)\right)^2.
\end{equation*}
Put $\theta:=l\pi/b$.
Since $|u|^2-|v|^2=1$ and $u+\overline{u}=2\cos(d\theta)$, we have
\begin{equation*}
\begin{split}
  &-|v|^2+\left|u+S_{d-2}\bigl(2\cos(l\pi/b)\bigr)\right|^2
  -
  \left(S_{d-1}\bigl(2\cos(l\pi/b)\bigr)\right)^2
  \\
  =&
  -|v|^2
  +
  \left(u+S_{d-2}\bigl(2\cos\theta\bigr)\right)
  \left(\overline{u}+S_{d-2}\bigl(2\cos\theta\bigr)\right)
  -
  \left(S_{d-1}\bigl(2\cos(\theta)\bigr)\right)^2
  \\
  =&
  -|v|^2+|u|^2
  +(u+\overline{u})S_{d-2}\bigl(2\cos\theta\bigr)
  +
  \left(S_{d-2}\bigl(2\cos\theta\bigr)\right)^2
  -
  \left(S_{d-1}\bigl(2\cos(\theta)\bigr)\right)^2
  \\
  =&
  1+2\cos(d\theta)\frac{\sin\bigl((d-1)\theta\bigr)}{\sin\theta}
  +
  \left(\frac{\sin\bigl((d-1)\theta\bigr)}{\sin\theta}\right)^2
  -
  \left(\frac{\sin(d\theta)}{\sin\theta}\right)^2
  \\
  =&
  1+2\cos(d\theta)\left(\frac{\sin(d\theta)}{\sin\theta}\cos\theta-\cos(d\theta)\right)
  +
  \left(\frac{\sin(d\theta)}{\sin\theta}\cos\theta-\cos(d\theta)\right)^2
  -
  \left(\frac{\sin(d\theta)}{\sin\theta}\right)^2
  \\
  =&
  1
  -
  2\cos^2(d\theta)
  +
  \left(\frac{\sin(d\theta)}{\sin\theta}\cos(d\theta)\right)^2
  +
  \cos^2(d\theta)
  -
  \left(\frac{\sin(d\theta)}{\sin\theta}\right)^2
  \\
  =&
  1
  -
  \cos^2(d\theta)
  -
  \left(\frac{\sin(d\theta)}{\sin\theta}\right)^2\left(1-\cos^2\theta\right)
  \\
  =&
  0.
\end{split}
\end{equation*}
\end{proof}
\begin{example}\label{ex:Main_Theorem_a_2}
When $a=2$, we can put $d:=(b+1)/2$ and $c:=1$.
Since $k=1$, the representation $\trho_{h,1,l}^{\rm{Irr}}$ is an $\SU(2)$ representation if and only if
\begin{equation*}
\begin{split}
  &\left(
    \cos\left(\frac{h\pi}{p-2b}\right)
    -
    \cos\left(\frac{\bigl((b+1)l+b\bigr)\pi}{2b}\right)
  \right)
  \left(
    \cos\left(\frac{h\pi}{p-2b}\right)
    -
    \cos\left(\frac{\bigl((b+1)l-b\bigr)\pi}{2b}\right)
  \right)
  \\
  =&
  \left(
    \cos\left(\frac{h\pi}{p-2b}\right)
    -
    (-1)^{(l+1)/2}
    \cos\left(\frac{l\pi}{2b}\right)
  \right)
  \left(
    \cos\left(\frac{h\pi}{p-2b}\right)
    -
    (-1)^{(l-1)/2}
    \cos\left(\frac{l\pi}{2b}\right)
  \right)
  \\
  =&
  \left(
    \cos\left(\frac{h\pi}{p-2b}\right)
    -
    \cos\left(\frac{l\pi}{2b}\right)
  \right)
  \left(
    \cos\left(\frac{h\pi}{p-2b}\right)
    +
    \cos\left(\frac{l\pi}{2b}\right)
  \right)
  >0.
\end{split}
\end{equation*}
The second equality follows since $(l+1)/2$ and $(l-1)/2$ have different parities.
Since $0<\frac{h}{p-2b}<1$ and $0<\frac{l}{2b}<\frac{1}{2}$, we have
\begin{equation*}
  \frac{h}{p-2b}<\frac{l}{2b}
  \quad\text{or}\quad
  \frac{l}{2b}+\frac{h}{p-2b}>1.
\end{equation*}
Therefore $\rho$ is an irreducible $\SU(2)$ representation if and only if the pair $(h,l)$ is in the following set:
\begin{equation*}
\begin{split}
  &
  \left\{
    (h,l)\in\Z^2\Bigm|
    1\le h\le p-2b-2,1\le l\le b-2,h\modtwo l\modtwo 1,
    \frac{h}{p-2b}<\frac{l}{2b}
  \right\}
  \\
  \cup&
  \left\{
    (h,l)\in\Z^2\Bigm|
    1\le h\le p-2b-2,1\le l\le b-2,h\modtwo l\modtwo 1,
    \frac{l}{2b}+\frac{h}{p-2b}>1
  \right\}.
\end{split}
\end{equation*}
Since the first set equals $\HmD$ from Example~\ref{ex:H}, we can write $\HmD$ as
\begin{equation*}
  \HmD
  =
  \{(h,1,l)
  \mid
  \text{$\trho_{(h,k,l)}^{\rm{Irr}}$: $\SU(2)$-representation},
  \frac{l}{2b}+\frac{h}{p-2b}<1\}
\end{equation*}
\par
It is an $\SL(2;\R)$ representation if and only if
\begin{equation*}
  \frac{h}{p-2b}>\frac{l}{2b}
  \quad\text{and}\quad
  \frac{l}{2b}+\frac{h}{p-2b}<1.
\end{equation*}
Therefore $\rho$ is an irreducible $\SU(1,1)$ representation if and only if the pair $(h,l)$ is in the following set:
\begin{equation*}
  \left\{
    (h,l)\in\Z^2\mid
    1\le h\le p-2b-2,1\le l\le b-2,h\modtwo l\modtwo 1,
    \frac{h}{p-2b}>\frac{l}{2b},
    \frac{l}{2b}+\frac{h}{p-2b}<1
  \right\},
\end{equation*}
which equals $\HpD\setminus\HmD$.
\par
Now from Example~\ref{ex:a_2_A_B}, we have
\begin{equation*}
\begin{split}
  &\htau_n(X_p;\exp(4\pi\i/n))
  \\
  =&
  \frac{e^{(n+1)\pi\i/4}n^{3/2}}{8\pi}
  \left(
    2\sum_{\substack{(h,1,l)\in\HmD\\ b-l\equiv2\pmod4}}
    +
    \sum_{(h,1,l)\in\HpD\setminus\HmD}
  \right)
  \Reidemeister_{+}^{\rm{Irr}}(h,1,l)e^{n\ChernSimons_{+}^{\rm{Irr}}(h,1,l)\pi\i}
  \\
  &-
  \frac{\i e^{n(1-p)\pi\i/4}n}{4\pi}
  \sum_{0<l<(p-1)/2}
  \Reidemeister^{\rm{Abel}}(l)e^{n\ChernSimons^{\rm{Abel}}(l)\pi\i}
  +
  O(n^{1/2})
  \\
  =&
  \frac{e^{(n+1)\pi\i/4}n^{3/2}}{8\pi}
  \left(
    2\sum_{
      \substack{\text{$\trho_{h,1,l}^{\rm{Irr}}$: $\SU(2)$-representation}
      \\
      \frac{l}{2b}+\frac{h}{p-2b}<1, b-l\equiv2\pmod4}}
    +
    \sum_{\text{$\trho_{h,1,l}^{\rm{Irr}}$: $\SU(1,1)$-representation}}
  \right)
  \Reidemeister_{+}^{\rm{Irr}}(h,1,l)e^{n\ChernSimons_{+}^{\rm{Irr}}(h,1,l)\pi\i}
  \\
  &-
  \frac{\i e^{n(1-p)\pi\i/4}n}{4\pi}
  \sum_{0<l<(p-1)/2}
  \Reidemeister^{\rm{Abel}}(l)e^{n\ChernSimons^{\rm{Abel}}(l)\pi\i}
  +
  O(n^{1/2}).
\end{split}
\end{equation*}
\end{example}
\begin{example}
When $a=4$, $b=3$, and $p=17$, we have
\begin{align*}
  \mathcal{H}&=\{(1,1,1),(1,1,3),(2,2,2),(3,1,1),(3,1,3),(4,2,2)\},\\
  \HpD&=\{(1,1,1),(1,1,3),(2,2,2),(3,1,1),(4,2,2)\},\\
  \HmD&=\{(1,1,1)\},\\
  \HpN&=\{(3,1,3)\},\\
  \HmN&=\{(1,1,3),(2,2,2),(3,1,1),(3,1,3),(4,2,2)\}
\end{align*}
We also see that $\trho_{h,k,l}^{\rm{Irr}}$ is an $\SU(2)$ representation if and only if $(h,k,l)$ is in
\begin{equation*}
  \{(1,1,3), (3,1,1)\},
\end{equation*}
and is an $\SU(1,1)$ representation if and only if $(h,k,l)$ is in
\begin{equation*}
  \{(1,1,1),(2,2,2),(3,1,3),(4,2,2)\}.
\end{equation*}
So there seems to be no good interpretation as in the case $a=2$.
\end{example}
\begin{example}
When $a=3$, $b=5$, and $p=19$, we have
\begin{align*}
  \mathcal{H}&=\{(1,1,1),(1,1,3),(2,2,2),(2,2,4),(3,1,1),(3,1,3)\},\\
  \HpD&=\{(1,1,3),(2,2,4),(3,1,3)\},\\
  \HmD&=\{(1,1,1),(1,1,3),(2,2,2)\},\\
  \HpN&=\{(1,1,1),(2,2,2),(3,1,1)\},\\
  \HmN&=\{(2,2,4),(3,1,1),(3,1,3)\}
\end{align*}
We also see that $\trho_{h,k,l}^{\rm{Irr}}$ is an $\SU(2)$ representation if and only if $(h,k,l)$ is in
\begin{equation*}
  \{(1,1,3), (3,1,1)\},
\end{equation*}
and is an $\SU(1,1)$ representation if and only if $(h,k,l)$ is in
\begin{equation*}
  \{(1,1,1),(2,2,2),(2,2,4),(3,1,3)\}.
\end{equation*}
So there seems to be no good interpretation as in the case $a=2$, either.
\end{example}

\section{Lemma}
In this section we prove the following lemma that we use in this paper.
\begin{lemma}\label{lem:sin_sin_exp}
Suppose that $a$ and $b$ are coprime positive integers.
Then for any odd integer $n\ge3$ we have
\begin{align}
  \sum_{\substack{0\le m\le 2ab \\\text{$m$: \rm{even}}}}
  \sin\left(\frac{m\pi}{a}\right)\sin\left(\frac{m\pi}{b}\right)
  \exp\left(-n\frac{m^2}{4ab}\pi\sqrt{-1}\right)
  &=
  0,
  \label{eq:sin_sin_exp_even}
  \\
  \sum_{\substack{0\le m\le 2ab \\\text{$m$: \rm{odd}}}}
  \sin\left(\frac{m\pi}{a}\right)\sin\left(\frac{m\pi}{b}\right)
  \exp\left(-n\frac{m^2}{4ab}\pi\sqrt{-1}\right)
  &=
  0.
  \label{eq:sin_sin_exp_odd}
\end{align}
\end{lemma}
\begin{proof}
We consider the following four cases:
\begin{enumerate}
\item[(i).] \eqref{eq:sin_sin_exp_even} for $ab$ odd,
\item[(ii).] \eqref{eq:sin_sin_exp_even} for $ab$ even,
\item[(iii).] \eqref{eq:sin_sin_exp_odd} for $ab$ odd,
\item[(iv).] \eqref{eq:sin_sin_exp_odd} for $ab$ even,
\end{enumerate}
\par
First of all, replacing $m$ with $2ab-m$ in the summation, we have
\begin{multline*}
  \sum_{0\le m\le 2ab}
  \sin\left(\frac{m\pi}{a}\right)\sin\left(\frac{m\pi}{b}\right)
  \exp\left(-n\frac{m^2}{4ab}\pi\sqrt{-1}\right)
  \\
  =
  \sum_{0\le m\le 2ab}
  \sin\left(\frac{(2ab-m)\pi}{a}\right)\sin\left(\frac{(2ab-m)\pi}{b}\right)
  \exp\left(-n\frac{(2ab-m)^2}{4ab}\pi\sqrt{-1}\right).
\end{multline*}
However, since
\begin{align*}
  \sin\left(\frac{(2ab-m)\pi}{a}\right)
  &=
  -\sin\left(\frac{m\pi}{a}\right),
  \\
  \sin\left(\frac{(2ab-m)\pi}{b}\right)
  &=
  -\sin\left(\frac{m\pi}{b}\right),
  \\
  \exp\left(-n\frac{(2ab-m)^2}{4ab}\pi\sqrt{-1}\right)
  &=
  (-1)^{n(m-ab)}
  \exp\left(-n\frac{m^2}{4ab}\pi\sqrt{-1}\right),
\end{align*}
and $n$ is odd, we have
\begin{multline}\label{eq:sin_sin_exp_1}
  \sum_{0\le m\le 2ab}
  \sin\left(\frac{m\pi}{a}\right)\sin\left(\frac{m\pi}{b}\right)
  \exp\left(-n\frac{m^2}{4ab}\pi\sqrt{-1}\right)
  \\
  =
  \sum_{0\le m\le 2ab}
  (-1)^{m-ab}
  \sin\left(\frac{m\pi}{a}\right)\sin\left(\frac{m\pi}{b}\right)
  \exp\left(-n\frac{m^2}{4ab}\pi\sqrt{-1}\right).
\end{multline}
Therefore, if $ab$ is even, then we have
\begin{equation*}
\begin{split}
  &
  \sum_{0\le m\le 2ab}
  \sin\left(\frac{m\pi}{a}\right)\sin\left(\frac{m\pi}{b}\right)
  \exp\left(-n\frac{m^2}{4ab}\pi\sqrt{-1}\right)
  \\
  =&
  \sum_{0\le m\le 2ab}
  (-1)^{m}
  \sin\left(\frac{m\pi}{a}\right)\sin\left(\frac{m\pi}{b}\right)
  \exp\left(-n\frac{m^2}{4ab}\pi\sqrt{-1}\right)
  \\
  =&
  \sum_{\substack{0\le m\le 2ab\\\text{$m$: \rm{even}}}}
  \sin\left(\frac{m\pi}{a}\right)\sin\left(\frac{m\pi}{b}\right)
  \exp\left(-n\frac{m^2}{4ab}\pi\sqrt{-1}\right)
  \\
  &-
  \sum_{\substack{0\le m\le 2ab\\\text{$m$: \rm{odd}}}}
  \sin\left(\frac{m\pi}{a}\right)\sin\left(\frac{m\pi}{b}\right)
  \exp\left(-n\frac{m^2}{4ab}\pi\sqrt{-1}\right)
\end{split}
\end{equation*}
So we have
\begin{equation*}
  \sum_{\substack{0\le m\le 2ab\\\text{$m$: \rm{odd}}}}
  \sin\left(\frac{m\pi}{a}\right)\sin\left(\frac{m\pi}{b}\right)
  \exp\left(-n\frac{m^2}{4ab}\pi\sqrt{-1}\right)
  =
  0,
\end{equation*}
proving (iv).
Similarly, if $ab$ is odd, we have
\begin{equation*}
  \sum_{\substack{0\le m \le 2ab\\\text{$m$: \rm{even}}}}
  \sin\left(\frac{m\pi}{a}\right)\sin\left(\frac{m\pi}{b}\right)
  \exp\left(-n\frac{m^2}{4ab}\pi\sqrt{-1}\right)
  =
  0
\end{equation*}
from \eqref{eq:sin_sin_exp_1}, proving (i).
\par
Next we consider the case (ii).
We assume that $ab$ is even.
Putting $m=2k$, we have
\begin{equation*}
\begin{split}
  &
  \sum_{\substack{0\le m\le 2ab\\\text{$m$: \rm{even}}}}
  \sin\left(\frac{m\pi}{a}\right)\sin\left(\frac{m\pi}{b}\right)
  \exp\left(-n\frac{m^2}{4ab}\pi\sqrt{-1}\right)
  \\
  =&
  \sum_{k=0}^{ab}
  \sin\left(\frac{2k\pi}{a}\right)\sin\left(\frac{2k\pi}{b}\right)
  \exp\left(-n\frac{k^2}{ab}\pi\sqrt{-1}\right).
\end{split}
\end{equation*}
We denote the right-hand side by $W_{\rm{ii}}$.
Note that
\begin{equation}\label{eq:sin_sin_exp_ii}
  \sum_{k=1}^{ab}
  \left(e^{2k\pi\sqrt{-1}/a}-e^{-2k\pi\sqrt{-1}/a}\right)
  \left(e^{2k\pi\sqrt{-1}/b}-e^{-2k\pi\sqrt{-1}/b}\right)
  e^{-n\frac{k^2}{ab}\pi\sqrt{-1}}
  =
  -4W_{\rm{ii}}.
\end{equation}
We will use the following Gauss sum reciprocity formula (see for example \cite[Chapter IX]{Chandrasekharan:1985}):
\begin{theorem}[Cauchy--Kronecker]\label{thm:Gauss}
Suppose that $c$ and $d$ are positive integers.
Let $w$ be a rational number such that $cd+2cw\equiv0\pmod{2}$.
Then we have
\begin{equation*}
  \frac{1}{\sqrt{d}}
  \sum_{k=1}^{d}e^{\frac{c}{d}(k+w)^2\pi\sqrt{-1}}
  =
  \frac{e^{\pi\sqrt{-1}/4}}{\sqrt{c}}
  \sum_{l=1}^{c}e^{-\frac{d}{c}l^2\pi\sqrt{-1}+2lw\pi\sqrt{-1}}.
\end{equation*}
\end{theorem}
From \eqref{eq:sin_sin_exp_ii}, we have
\begin{multline*}
  -4W_{\rm{ii}}
  \\
  =
  \sum_{k=1}^{ab}
  \left(
    e^{\left(2k\frac{a+b}{ab}-n\frac{k^2}{ab}\right)\pi\sqrt{-1}}
    -
    e^{\left(2k\frac{-a+b}{ab}-n\frac{k^2}{ab}\right)\pi\sqrt{-1}}
    -
    e^{\left(2k\frac{a-b}{ab}-n\frac{k^2}{ab}\right)\pi\sqrt{-1}}
    +
    e^{\left(2k\frac{-a-b}{ab}-n\frac{k^2}{ab}\right)\pi\sqrt{-1}}
  \right)
\end{multline*}
Putting $c:=ab$ and $d:=n$, and choose $w:=\frac{\pm{a}\pm{b}}{ab}$, we can apply Theorem~\ref{thm:Gauss} because $ab$ is even.
We have
\begin{equation*}
\begin{split}
  &-4W_{\rm{ii}}
  \\
  =&
  \frac{\sqrt{ab}}{e^{\pi\sqrt{-1}/4}\sqrt{n}}
  \sum_{k=1}^{n}
  \left(
    e^{\frac{ab}{n}(k+\frac{a+b}{ab})^2\pi\sqrt{-1}}
    -
    e^{\frac{ab}{n}(k+\frac{-a+b}{ab})^2\pi\sqrt{-1}}
    -
    e^{\frac{ab}{n}(k+\frac{a-b}{ab})^2\pi\sqrt{-1}}
    +
    e^{\frac{ab}{n}(k+\frac{-a-b}{ab})^2\pi\sqrt{-1}}
  \right)
  \\
  =&
  \sqrt{\frac{ab}{n}}e^{-\pi\sqrt{-1}/4}
  \sum_{k=1}^{n}
  \left(
    e^{\frac{\pi\sqrt{-1}}{abn}(abk+a+b)^2}
    -
    e^{\frac{\pi\sqrt{-1}}{abn}(abk-a+b)^2}
    -
    e^{\frac{\pi\sqrt{-1}}{abn}(abk+a-b)^2}
    +
    e^{\frac{\pi\sqrt{-1}}{abn}(abk-a-b)^2}
  \right).
\end{split}
\end{equation*}
Replacing $k$ with $n-k$, we see that
\begin{align*}
  \sum_{k=1}^{n}e^{\frac{\pi\sqrt{-1}}{abn}(abk-a-b)^2}
  &=
  \sum_{k=1}^{n}e^{\frac{\pi\sqrt{-1}}{abn}(abk+a+b)^2}
  \\
  \sum_{k=1}^{n}e^{\frac{\pi\sqrt{-1}}{abn}(abk+a-b)^2}
  &=
  \sum_{k=1}^{n}e^{\frac{\pi\sqrt{-1}}{abn}(abk-a+b)^2}.
\end{align*}
Therefore we have
\begin{equation*}
\begin{split}
  &-4W_{\rm{ii}}
  \\
  =&
  2\sqrt{\frac{ab}{n}}e^{-\pi\sqrt{-1}/4}
  \sum_{k=1}^{n}
  \left(
    e^{\frac{\pi\sqrt{-1}}{abn}(abk+a+b)^2}
    -
    e^{\frac{\pi\sqrt{-1}}{abn}(abk-a+b)^2}
  \right)
  \\
  =&
  2\sqrt{\frac{ab}{n}}e^{-\pi\sqrt{-1}/4}e^{\frac{a\pi\sqrt{-1}}{bn}+\frac{b\pi\sqrt{-1}}{an}}
  \sum_{k=1}^{n}
  e^{\frac{bk(ak+2)\pi\sqrt{-1}}{n}}
  \left(
    e^{\frac{2(ak+1)\pi\sqrt{-1}}{n}}
    -
    e^{\frac{-2(ak+1)\pi\sqrt{-1}}{n}}
  \right)
\end{split}
\end{equation*}
\par
Now let us consider the colored Jones polynomial of the torus knot $T(a,b)$.
We put
\begin{equation*}
\begin{split}
  \tilde{J}_k(q)
  &:=
  q^{ab(k^2-1)/4}(q^{k/2}-q^{-k/2})J_k(T(a,b);q)
  \\
  &=
  \sum_{j=-(k-1)/2}^{(k-1)/2}
  q^{bj(aj+1)}(q^{aj+1/2}-q^{-aj-1/2}).
\end{split}
\end{equation*}
Note that this is nothing but the Kauffman bracket of $T(a,b)$ with $n-1$-th Jones--Wenzl idempotent inserted, replacing $A$ with $q^{-1/4}$.
Then we have
\begin{equation*}
\begin{split}
  &
  \tilde{J}_n\left(e^{4\pi\sqrt{-1}/n}\right)
  \\
  =&
  \sum_{j=-(n-1)/2}^{(n-1)/2}
  e^{4bj(aj+1)\pi\sqrt{-1}/n}
  \left(e^{4(aj+1/2)\pi\sqrt{-1}/n}-e^{-4(aj+1/2)\pi \sqrt{-1}/n}\right)
  \\
  &\text{(Put $l=2j$, noting that $n$ is odd)}
  \\
  =&
  \sum_{\substack{1-n\le l\le n-1\\\text{$l$: even}}}
  e^{bl(al+2)\pi\sqrt{-1}/n}
  \left(e^{2(al+1)\pi\sqrt{-1}/n}-e^{-2(al+1)\pi \sqrt{-1}/n}\right)
  \\
  =&
  \sum_{\substack{0\le l\le n-1\\\text{$l$: even}}}
  e^{bl(al+2)\pi\sqrt{-1}/n}
  \left(e^{2(al+1)\pi\sqrt{-1}/n}-e^{-2(al+1)\pi\sqrt{-1}/n}\right)
  \\
  &+
  \sum_{\substack{1\le l\le n-2\\\text{$l$ odd}}}
  e^{b(l-n)(a(l-n)+2)\pi\sqrt{-1}/n}
  \left(e^{2(a(l-n)+1)\pi\sqrt{-1}/n}-e^{-2(a(l-n)+1)\pi\sqrt{-1}/n}\right)
  \\
  &\text{(since $ab$ is even and $n$ is odd)}
  \\
  =&
  \sum_{l=0}^{n-1}
  e^{ \frac{\pi\sqrt{-1}}{n}bl(al+2)}
  \left(e^{ \frac{2\pi\sqrt{-1}}{n}(al+1)}-e^{-\frac{2\pi\sqrt{-1}}{n}(al+1)}\right). 
\end{split}
\end{equation*}
Hence we have
\begin{equation*}
\begin{split}
  -4W_{\rm{ii}}
  &=
  2\sqrt{\frac{ab}{n}}e^{-\pi\sqrt{-1}/4}e^{\frac{a\pi\sqrt{-1}}{bn}+\frac{b\pi\sqrt{-1}}{an}}
  \tilde{J}_n(K;e^{4\pi \sqrt{-1}/n})
  \\
  &=
  2\sqrt{\frac{ab}{n}}e^{-\pi\sqrt{-1}/4}e^{\frac{a\pi\sqrt{-1}}{bn}+\frac{b\pi\sqrt{-1}}{an}}
  e^{ab(n^2-1)\pi\sqrt{-1}/n}
  \left(e^{2\pi\sqrt{-1}}-e^{-2\pi\sqrt{-1}}\right)
  \tilde{J}_n(e^{4\pi\sqrt{-1}/n})
  \\
  &=
  0
\end{split}
\end{equation*}
proving (ii).
\par
We consider $\rm{(iii)}$.
Note that we are assuming that both $a$ and $b$ are odd.
We have
\begin{equation*}
\begin{split}
  &
  \sum_{\substack{0\le m\le 2ab\\\text{$m$: \rm{odd}}}}
  \sin\left(\frac{m\pi}{a}\right)\sin\left(\frac{m\pi}{b}\right)
  \exp\left(-n\frac{m^2}{4ab}\pi\sqrt{-1}\right)
  \\
  =&
  \sum_{\substack{0\le m\le ab\\\text{$m$: \rm{odd}}}}
  \sin\left(\frac{m\pi}{a}\right)\sin\left(\frac{m\pi}{b}\right)
  \exp\left(-n\frac{m^2}{4ab}\pi\sqrt{-1}\right)
  \\
  &+
  \sum_{\substack{0\le m\le ab\\\text{$m$: \rm{odd}}}}
  \sin\left(\frac{(2ab-m)\pi}{a}\right)\sin\left(\frac{(2ab-m)\pi}{b}\right)
  \exp\left(-n\frac{(2ab-m)^2}{4ab}\pi\sqrt{-1}\right)
  \\
  =&
  2\sum_{\substack{0\le m\le ab\\\text{$m$: \rm{odd}}}}
  \sin\left(\frac{m\pi}{a}\right)\sin\left(\frac{m\pi}{b}\right)
  \exp\left(-n\frac{m^2}{4ab}\pi\sqrt{-1}\right).
  \\
  =&
  2\sum_{\substack{0\le k\le ab\\\text{$k$: \rm{even}}}}
  \sin\left(\frac{(ab-k)\pi}{a}\right)\sin\left(\frac{(ab-k)\pi}{b}\right)
  \exp\left(-n\frac{(ab-k)^2}{4ab}\pi\sqrt{-1}\right)
  \\
  =&
  2\exp\left(\frac{-abn\pi\sqrt{-1}}{4}\right)
  \sum_{\substack{0\le k\le ab\\\text{$k$: \rm{even}}}}
  \sqrt{-1}^{k}\sin\left(\frac{k\pi}{a}\right)\sin\left(\frac{k\pi}{b}\right)
  \exp\left(-n\frac{k^2}{4ab}\pi\sqrt{-1}\right).
\end{split}
\end{equation*}
On the other hand, since we have
\begin{equation*}
\begin{split}
  &
  \sum_{\substack{0\le k\le 2ab\\\text{$k$: \rm{even}}}}
  \sqrt{-1}^{k}\sin\left(\frac{k\pi}{a}\right)\sin\left(\frac{k\pi}{b}\right)
  \exp\left(-n\frac{k^2}{4ab}\pi\sqrt{-1}\right)
  \\
  =&
  \sum_{\substack{0\le k\le ab\\\text{$k$: \rm{even}}}}
  \sqrt{-1}^{k}\sin\left(\frac{k\pi}{a}\right)\sin\left(\frac{k\pi}{b}\right)
  \exp\left(-n\frac{k^2}{4ab}\pi\sqrt{-1}\right)
  \\
  &+
  \sum_{\substack{0\le k\le ab\\\text{$k$: \rm{even}}}}
  \sqrt{-1}^{2ab-k}\sin\left(\frac{(2ab-k)\pi}{a}\right)\sin\left(\frac{(2ab-k)\pi}{b}\right)
  \exp\left(-n\frac{(2ab-k)^2}{4ab}\pi\sqrt{-1}\right)
  \\
  =&
  2
  \sum_{\substack{0\le k\le ab\\\text{$k$: \rm{even}}}}
  \sqrt{-1}^{k}\sin\left(\frac{k\pi}{a}\right)\sin\left(\frac{k\pi}{b}\right)
  \exp\left(-n\frac{k^2}{4ab}\pi\sqrt{-1}\right)
  \\
  =&
  2
  \sum_{l=0}^{ab}(-1)^{l}\sin\left(\frac{2l\pi}{a}\right)\sin\left(\frac{2l\pi}{b}\right)
  \exp\left(-n\frac{l^2}{ab}\pi\sqrt{-1}\right),
\end{split}
\end{equation*}
\eqref{eq:sin_sin_exp_odd} becomes
\begin{equation*}
  \exp\left(\frac{-abn\pi\sqrt{-1}}{4}\right)
  \sum_{l=0}^{ab}(-1)^{l}\sin\left(\frac{2l\pi}{a}\right)\sin\left(\frac{2l\pi}{b}\right)
  \exp\left(-n\frac{l^2}{ab}\pi\sqrt{-1}\right).
\end{equation*}
In a similar way, the summation above becomes
\begin{equation*}
\begin{split}
  &
  \sum_{l=0}^{ab}
  (-1)^{l}
  \sin\left(\frac{2l\pi}{a}\right)\sin\left(\frac{2l\pi}{b}\right)
  \exp\left(-n\frac{l^2}{ab}\pi\sqrt{-1}\right)
  \\
  =&
  \sum_{\substack{0\le l\le ab\\\text{$l$: even}}}
  \sin\left(\frac{2l\pi}{a}\right)\sin\left(\frac{2l\pi}{b}\right)
  \exp\left(-n\frac{l^2}{ab}\pi\sqrt{-1}\right)
  \\
  &
  -
  \sum_{\substack{0\le l\le ab\\\text{$l$: even}}}
  \sin\left(\frac{2(ab-l)\pi}{a}\right)\sin\left(\frac{2(ab-l)\pi}{b}\right)
  \exp\left(-n\frac{(ab-l)^2}{ab}\pi\sqrt{-1}\right)
  \\
  &\text{(since $ab$ and $n$ are odd)}
  \\
  =&
  2\sum_{\substack{0\le l\le ab\\\text{$l$: even}}}
  \sin\left(\frac{2l\pi}{a}\right)\sin\left(\frac{2l\pi}{b}\right)
  \exp\left(-n\frac{l^2}{ab}\pi\sqrt{-1}\right)
  \\
  =&
  \sum_{\substack{0\le l\le ab\\\text{$l$: even}}}
  \sin\left(\frac{2l\pi}{a}\right)\sin\left(\frac{2l\pi}{b}\right)
  \exp\left(-n\frac{l^2}{ab}\pi\sqrt{-1}\right)
  \\
  &+
  \sum_{\substack{0\le l\le ab\\\text{$l$: even}}}
  \sin\left(\frac{2(2ab-l)\pi}{a}\right)\sin\left(\frac{2(2ab-l)\pi}{b}\right)
  \exp\left(-n\frac{(2ab-l)^2}{ab}\pi\sqrt{-1}\right)
  \\
  =&
  \sum_{\substack{0\le l\le 2ab\\\text{$l$: even}}}
  \sin\left(\frac{2l\pi}{a}\right)\sin\left(\frac{2l\pi}{b}\right)
  \exp\left(-n\frac{l^2}{ab}\pi\sqrt{-1}\right)
  \\
  =&
  \sum_{h=0}^{ab}
  \sin\left(\frac{4h\pi}{a}\right)\sin\left(\frac{4h\pi}{b}\right)
  \exp\left(-4n\frac{h^2}{ab}\pi\sqrt{-1}\right).
\end{split}
\end{equation*}
We denote the right-hand side by $W_{\rm{iii}}$, and show that it vanishes.
\par
Using the equality
\begin{equation*}
  \sum_{k=1}^{ab}
  \left(e^{4k\pi\sqrt{-1}/a}-e^{4k\pi\sqrt{-1}/a}\right)
  \left(e^{4k\pi\sqrt{-1}/b}-e^{4k\pi \sqrt{-1}/b}\right)
  e^{-4n\frac{k^2}{ab}\pi\sqrt{-1}}
  =
  -4W_{\rm{iii}},
\end{equation*}
we apply the Gauss sum reciprocity formula.
Putting $c:=ab$, $d:=4n$, and $w:=\frac{2(\pm a\pm b)}{ab}$ in Theorem~\ref{thm:Gauss}, we have
\begin{equation*}
\begin{split}
  &-4W_{\rm{iii}}
  \\
  =&
  \sum_{k=1}^{ab}
  \left(
    e^{\left(4k\frac{a+b}{ab}-4n\frac{k^2}{ab}\right)\pi\sqrt{-1}}
    -
    e^{\left(4k\frac{-a+b}{ab}-4n\frac{k^2}{ab}\right)\pi\sqrt{-1}}
    -
    e^{\left(4k\frac{a-b}{ab}-4n\frac{k^2}{ab}\right)\pi\sqrt{-1}}
    +
    e^{\left(4k\frac{-a-b}{ab}-4n\frac{k^2}{ab}\right)\pi\sqrt{-1}}
  \right)
  \\
  =&
  \frac{\sqrt{ab}}{e^{\pi\sqrt{-1}/4}\sqrt{4n}}
  \\
  &\times
  \sum_{l=1}^{4n}
  \left(
    e^{\frac{ab}{4n}(l+\frac{2(a+b)}{ab})^2\pi\sqrt{-1}}
    -
    e^{\frac{ab}{4n}(l+\frac{2(-a+b)}{ab})^2\pi\sqrt{-1}}
    -
    e^{\frac{ab}{4n}(l+\frac{2(a-b)}{ab})^2\pi\sqrt{-1}}
    +
    e^{\frac{ab}{4n}(l+\frac{2(-a-b)}{ab})^2\pi\sqrt{-1}}
  \right)
  \\
  =&
  \sqrt{\frac{ab}{4n}}e^{-\pi\sqrt{-1}/4}
  \\
  &\times
  \sum_{l=1}^{4n}
  \left(
    e^{\frac{\pi\sqrt{-1}}{4abn}(abl+2a+2b)^2\pi\sqrt{-1}}
    -
    e^{\frac{\pi\sqrt{-1}}{4abn}(abl-2a+2b)^2\pi\sqrt{-1}}
    -
    e^{\frac{\pi\sqrt{-1}}{4abn}(abl+2a-2b)^2\pi\sqrt{-1}}
    +
    e^{\frac{\pi\sqrt{-1}}{4abn}(abl-2a-2b)^2\pi\sqrt{-1}}
  \right).
\end{split}
\end{equation*}
Replacing $l$ with $4n-l$, we have
\begin{align*}
  \sum_{l=1}^{4n}e^{\frac{\pi\sqrt{-1}}{4abn}(abl-2a-2b)^2\pi\sqrt{-1}}
  &=
  \sum_{l=1}^{4n}e^{\frac{\pi\sqrt{-1}}{4abn}(abl+2a+2b)^2\pi\sqrt{-1}},
  \\
  \sum_{l=1}^{4n}e^{\frac{\pi\sqrt{-1}}{4abn}(abl+2a-2b)^2\pi\sqrt{-1}}
  &=
  \sum_{l=1}^{4n}e^{\frac{\pi\sqrt{-1}}{4abn}(abl+2a-2b)^2\pi\sqrt{-1}}.
\end{align*}
Therefore we have
\begin{equation*}
\begin{split}
  -4W_{\rm{iii}}
  &=
  \sqrt{\frac{ab}{n}}e^{-\pi\sqrt{-1}/4}
  \sum_{l=1}^{4n}
  \left(
    e^{\frac{\pi\sqrt{-1}}{4abn}(abl+2a+2b)^2\pi\sqrt{-1}}
    -
    e^{\frac{\pi\sqrt{-1}}{4abn}(abl-2a+2b)^2\pi\sqrt{-1}}
  \right).
  \\
  =&
  \sqrt{\frac{ab}{n}}e^{-\pi\sqrt{-1}/4}
  e^{\frac{a\pi\sqrt{-1}}{bn}+\frac{b\pi\sqrt{-1}}{an}}
  \sum_{l=1}^{4n}
  e^{\frac{bl(al+4)\pi\sqrt{-1}}{4n}}
  \left(e^{\frac{(al+2)\pi\sqrt{-1}}{n}}-e^{-\frac{(al+2)\pi\sqrt{-1}}{n}}\right).
\end{split}
\end{equation*}
We denote the summation in the right-hand side by $\tilde{W}_{\rm{iii}}$.
Note that
\begin{equation*}
\begin{split}
  &
  \sum_{\substack{1\le l\le4n\\\text{$l$: even}}}
  e^{\frac{bl(al+4)\pi\sqrt{-1}}{4n}}
  \left(e^{\frac{(al+2)\pi\sqrt{-1}}{n}}-e^{-\frac{(al+2)\pi\sqrt{-1}}{n}}\right)
  \\
  =&
  \sum_{\substack{1\le l\le 2n\\\text{$l$: even}}}
  e^{\frac{bl(al+4)\pi\sqrt{-1}}{4n}}
  \left(e^{\frac{(al+2)\pi\sqrt{-1}}{n}}-e^{-\frac{(al+2)\pi\sqrt{-1}}{n}}\right)
  \\
  &
  +
  \sum_{\substack{1\le l\le 2n\\\text{$l$: even}}}
  e^{\frac{b(l+2n)(a(l+2n)+4)\pi\sqrt{-1}}{4n}}
  \left(e^{\frac{(a(l+2n)+2)\pi\sqrt{-1}}{n}}-e^{-\frac{(a(l+2n)+2)\pi\sqrt{-1}}{n}}\right)
  \\
  =&
  \sum_{\substack{1\le l\le2n\\\text{$l$: even}}}
  \left(1+(-1)^{b(2+al+an)}\right)
  e^{\frac{bl(al+4)\pi\sqrt{-1}}{4n}}
  \left(e^{\frac{(al+2)\pi\sqrt{-1}}{n}}-e^{-\frac{(al+2)\pi\sqrt{-1}}{n}}\right)
  \\
  =&
  0.
\end{split}
\end{equation*}
In a similar way we can prove
\begin{equation*}
\begin{split}
  &
  \sum_{\substack{1\le l\le4n\\\text{$l$: odd}}}
  e^{\frac{bl(al+4)\pi\sqrt{-1}}{4n}}
  \left(e^{\frac{(al+2)\pi\sqrt{-1}}{n}}-e^{-\frac{(al+2)\pi\sqrt{-1}}{n}}\right)
  \\
  =&
  \sum_{\substack{1\le l\le2n\\\text{$l$: odd}}}
  \left(1+(-1)^{b(2+al+an)}\right)
  e^{\frac{bl(al+4)\pi\sqrt{-1}}{4n}}
  \left(e^{\frac{(al+2)\pi\sqrt{-1}}{n}}-e^{-\frac{(al+2)\pi\sqrt{-1}}{n}}\right)
  \\
  =&
  2
  \sum_{\substack{1\le l\le2n\\\text{$l$: odd}}}
  e^{\frac{bl(al+4)\pi\sqrt{-1}}{4n}}
  \left(e^{\frac{(al+2)\pi\sqrt{-1}}{n}}-e^{-\frac{(al+2)\pi\sqrt{-1}}{n}}\right).
\end{split}
\end{equation*}
Hence we have the following four equalities:
\begin{equation}\label{eq:W_iii_1}
  \tilde{W}_{\rm{iii}}
  =
  2
  \sum_{\substack{1\le l\le2n\\\text{$l$: odd}}}
  e^{\frac{bl(al+4)\pi\sqrt{-1}}{4n}}
  \left(e^{\frac{(al+2)\pi\sqrt{-1}}{n}}-e^{-\frac{(al+2)\pi\sqrt{-1}}{n}}\right),
\end{equation}
\begin{equation}\label{eq:W_iii_2}
\begin{split}
  \tilde{W}_{\rm{iii}}
  &=
  2
  \sum_{\substack{1\le l\le 2n\\\text{$l$: odd}}}
  e^{\frac{b(2n-l)(a(2n-l)+4)\pi\sqrt{-1}}{4n}}
  (e^{\frac{(a(2n-l)+2)\pi\sqrt{-1}}{n}}-e^{\frac{-(a(2n-l)+2)\pi\sqrt{-1}}{n}})
  \\
  &=
  2
  \sum_{\substack{1\le l\le 2n\\\text{$l$: odd}}}
  e^{\frac{bl(al-4)\pi\sqrt{-1}}{4n}}
  (e^{\frac{-(al-2)\pi\sqrt{-1}}{n}}-e^{\frac{(al-2)\pi\sqrt{-1}}{n}}),
\end{split}
\end{equation}
\begin{equation}\label{eq:W_iii_3}
\begin{split}
  &
  \sum_{\substack{n+1\le l\le 2n-1\\\text{$l$: odd}}}
  e^{\frac{bl(al+4)\pi\sqrt{-1}}{4n}}
  \left(e^{\frac{(al+2)\pi\sqrt{-1}}{n}}-e^{-\frac{(al+2)\pi\sqrt{-1}}{n}}\right)
  \\
  =&
  \sum_{\substack{1\le l\le n-1\\\text{$l$: odd}}}
  e^{\frac{b(2n-l)(a(2n-l)+4)\pi\sqrt{-1}}{4n}}
  \left(e^{\frac{(a(2n-l)+2)\pi\sqrt{-1}}{n}}-e^{-\frac{(a(2n-l)+2)\pi\sqrt{-1}}{n}}\right)
  \\
  =&
  \sum_{\substack{1\le l\le n-1\\\text{$l$: odd}}}
  e^{\frac{bl(al-4)\pi\sqrt{-1}}{4n}}
  \left(e^{-\frac{(al-2)\pi\sqrt{-1}}{n}}-e^{\frac{(al-2)\pi\sqrt{-1}}{n}}\right),
\end{split}
\end{equation}
\begin{equation}\label{eq:W_iii_4}
\begin{split}
  &
  \sum_{\substack{n+1\le l\le 2n-1\\\text{$l$: odd}}}
  e^{\frac{bl(al-4)\pi\sqrt{-1}}{4n}}
  \left(e^{-\frac{(al-2)\pi\sqrt{-1}}{n}}-e^{\frac{(al-2)\pi\sqrt{-1}}{n}}\right)
  \\
  =&
  \sum_{\substack{1\le l\le n-1\\\text{$l$: odd}}}
  e^{\frac{b(2n-l)(a(2n-l)-4)\pi\sqrt{-1}}{4n}}
  \left(e^{-\frac{(a(2n-l)-2)\pi\sqrt{-1}}{n}}-e^{\frac{(a(2n-l)-2)\pi\sqrt{-1}}{n}}\right)
  \\
  =&
  \sum_{\substack{1\le l\le n-1\\\text{$l$: odd}}}
  e^{\frac{bl(al+4)\pi\sqrt{-1}}{4n}}
  \left(e^{\frac{(al+2)\pi\sqrt{-1}}{n}}-e^{-\frac{(al+2)\pi\sqrt{-1}}{n}}\right)
\end{split}
\end{equation}
since $n$ is odd.
\par
Add \eqref{eq:W_iii_1} and \eqref{eq:W_iii_2}, and divide it by two, we have
\begin{equation}\label{eq:W_iii_5}
\begin{split}
  \tilde{W}_{\rm{iii}}
  &=
  \sum_{\substack{1\le l\le 2n\\\text{$l$: odd}}}
  e^{\frac{bl(al+4)\pi\sqrt{-1}}{4n}}
  \left(e^{\frac{(al+2)\pi\sqrt{-1}}{n}}-e^{-\frac{(al+2)\pi\sqrt{-1}}{n}}\right)
  \\
  &\quad+
  \sum_{\substack{1\le l\le 2n\\\text{$l$: odd}}}
  e^{\frac{bl(al-4)\pi\sqrt{-1}}{4n}}
  \left(e^{-\frac{(al-2)\pi\sqrt{-1}}{n}}-e^{\frac{(al-2)\pi\sqrt{-1}}{n}}\right).
\end{split}
\end{equation}
From \eqref{eq:W_iii_3}, the first summation in the right-hand side of \eqref{eq:W_iii_5} becomes
\begin{multline*}
  \sum_{\substack{1\le l\le n-1\\\text{$l$: odd}}}
  e^{\frac{bl(al+4)\pi\sqrt{-1}}{4n}}
  \left(e^{\frac{(al+2)\pi\sqrt{-1}}{n}}-e^{-\frac{(al+2)\pi\sqrt{-1}}{n}}\right)
  \\
  +
  e^{\frac{bn(an-4)\pi\sqrt{-1}}{4n}}
  \left(e^{\frac{(an-2)\pi\sqrt{-1}}{n}}-e^{-\frac{(an-2)\pi\sqrt{-1}}{n}}\right)
  \\
  +
  \sum_{\substack{1\le l\le n-1\\\text{$l$: odd}}}
  e^{\frac{bl(al-4)\pi\sqrt{-1}}{4n}}
  \left(e^{\frac{(al-2)\pi\sqrt{-1}}{n}}-e^{-\frac{(al-2)\pi\sqrt{-1}}{n}}\right).
\end{multline*}
From \eqref{eq:W_iii_4}, the second summation in the right-hand side of \eqref{eq:W_iii_5} becomes
\begin{multline*}
  \sum_{\substack{1\le l\le n-1\\\text{$l$: odd}}}
  e^{\frac{bl(al-4)\pi\sqrt{-1}}{4n}}
  \left(e^{-\frac{(al-2)\pi\sqrt{-1}}{n}}-e^{\frac{(al-2)\pi\sqrt{-1}}{n}}\right)
  \\
  +
  e^{\frac{bn(an-4)\pi\sqrt{-1}}{4n}}
  \left(e^{-\frac{(an-2)\pi\sqrt{-1}}{n}}-e^{\frac{(an-2)\pi\sqrt{-1}}{n}}\right)
  \\
  +
  \sum_{\substack{1\le l\le n-1\\\text{$l$: odd}}}
  e^{\frac{bl(al+4)\pi\sqrt{-1}}{4n}}
  \left(e^{-\frac{(al+2)\pi\sqrt{-1}}{n}}-e^{\frac{(al+2)\pi\sqrt{-1}}{n}}\right).
\end{multline*}
Therefore \eqref{eq:W_iii_5} turns out to be
\begin{equation*}
\begin{split}
  \tilde{W}_{\rm{iii}}
  &=
  2\sum_{\substack{1\le l\le n-1\\\text{$l$: odd}}}
  e^{\frac{bl(al+4)\pi\sqrt{-1}}{4n}}
  \left(e^{\frac{(al+2)\pi\sqrt{-1}}{n}}-e^{-\frac{(al+2)\pi\sqrt{-1}}{n}}\right)
  \\
  &\quad+
  2\sum_{\substack{1\le l\le n-1\\\text{$l$: odd}}}
  e^{\frac{bl(al-4)\pi\sqrt{-1}}{4n}}
  \left(e^{-\frac{(al-2)\pi\sqrt{-1}}{n}}-e^{\frac{(al-2)\pi\sqrt{-1}}{n}}\right)
  \\
  &\quad-
  2e^{abn\pi\sqrt{-1}/4}
  \left(e^{\frac{2\pi\sqrt{-1}}{n}}-e^{-\frac{2\pi \sqrt{-1}}{n}}\right)
  \\
  &=
  2\sum_{\substack{1\le l\le n-1\\\text{$l$: even}}}
  e^{\frac{b(n-l)(a(n-l)+4)\pi\sqrt{-1}}{4n}}
  \left(e^{\frac{(a(n-l)+2)\pi\sqrt{-1}}{n}}-e^{-\frac{(a(n-l)+2)\pi\sqrt{-1}}{n}}\right)
  \\
  &\quad+
  2\sum_{\substack{1\le l\le n-1\\\text{$l$: even}}}
  e^{\frac{b(n-l)(a(n-l)-4)\pi\sqrt{-1}}{4n}}
  \left(e^{-\frac{(a(n-l)-2)\pi\sqrt{-1}}{n}}-e^{\frac{(a(n-l)-2)\pi\sqrt{-1}}{n}}\right)
  \\
  &\quad-
  2e^{abn\pi\sqrt{-1}/4}
  \left(e^{\frac{2\pi\sqrt{-1}}{n}}-e^{-\frac{2\pi\sqrt{-1}}{n}}\right)
  \\
  &=
  2\sum_{\substack{1\le l\le n-1\\\text{$l$: even}}}
  (-1)^{l/2}e^{abn\pi\sqrt{-1}/4}e^{\frac{bl(al-4)\pi\sqrt{-1}}{4n}}
  \left(e^{-\frac{(al-2)\pi\sqrt{-1}}{n}}-e^{\frac{(al-2)\pi\sqrt{-1}}{n}}\right)
  \\
  &\quad+
  2\sum_{\substack{1\le l\le n-1\\\text{$l$: even}}}
  (-1)^{l/2}e^{abn\pi\sqrt{-1}/4}e^{\frac{bl(al+4)\pi\sqrt{-1}}{4n}}
  \left(e^{\frac{(al+2)\pi\sqrt{-1}}{n}}-e^{-\frac{(al+2)\pi\sqrt{-1}}{n}}\right)
  \\
  &\quad-
  2e^{abn\pi\sqrt{-1}/4}
  \left(e^{\frac{2\pi\sqrt{-1}}{n}}-e^{-\frac{2\pi\sqrt{-1}}{n}}\right)
  \\
  &=
  2e^{abn\pi\sqrt{-1}/4}
  \sum_{l=1}^{(n-1)/2}(-1)^{l}e^{\frac{bl(al-2)\pi\sqrt{-1}}{n}}
  \left(e^{-\frac{2(al-1)\pi\sqrt{-1}}{n}}-e^{\frac{2(al-1)\pi\sqrt{-1}}{n}}\right)
  \\
  &\quad+
  2e^{abn\pi\sqrt{-1}/4}
  \sum_{l=1}^{(n-1)/2}(-1)^{l}e^{\frac{bl(al+2)\pi\sqrt{-1}}{n}}
  \left(e^{\frac{2(al+1)\pi\sqrt{-1}}{n}}-e^{-2\frac{(al+1)\pi\sqrt{-1}}{n}}\right)
  \\
  &\quad-
  2e^{abn\pi\sqrt{-1}/4}
  \left(e^{\frac{2\pi\sqrt{-1}}{n}}-e^{-\frac{2\pi\sqrt{-1}}{n}}\right).
\end{split}
\end{equation*}
\par
Now we will calculate $\tilde{J}_n\left(e^{4\pi\sqrt{-1}/4}\right)$.
Since $ab$ is odd, we have
\begin{equation*}
\begin{split}
  \tilde{J}_n\left(e^{4\pi\sqrt{-1}/n}\right)
  =&
  \sum_{j=-(n-1)/2}^{(n-1)/2}
  e^{4bj(aj+1)\pi\sqrt{-1}/n}
  \left(e^{4(aj+1/2)\pi\sqrt{-1}/n}-e^{-4(aj+1/2)\pi\sqrt{-1}/n}\right)
  \\
  &\text{(Put $l=2j$, noting that $n$ is odd)}
  \\
  =&
  \sum_{\substack{1-n\le l\le n-1\\\text{$l$: even}}}
  e^{bl(al+2)\pi\sqrt{-1}/n}
  \left(e^{2(al+1)\pi\sqrt{-1}/n}-e^{-2(al+1)\pi\sqrt{-1}/n}\right)
  \\
  =&
  \sum_{\substack{0\le l\le n-1\\\text{$l$: even}}}
  e^{bl(al+2)\pi\sqrt{-1}/n}
  \left(e^{2(al+1)\pi\sqrt{-1}/n}-e^{-2(al+1)\pi\sqrt{-1}/n}\right)
  \\
  &\quad+
  \sum_{\substack{1\le l\le n-2\\\text{$l$: odd}}}
  e^{b(l-n)(a(l-n)+2)\pi\sqrt{-1}/n}
  \left(e^{2(a(l-n)+1)\pi\sqrt{-1}/n}-e^{-2(a(l-n)+1)\pi \sqrt{-1}/n}\right)
  \\
  &\text{(since $ab$ and $n$ are odd)}
  \\
  =&
  \sum_{\substack{0\le l\le n-1\\\text{$l$: even}}}
  e^{bl(al+2)\pi\sqrt{-1}/n}
  \left(e^{2(al+1)\pi\sqrt{-1}/n}-e^{-2(al+1)\pi\sqrt{-1}/n}\right)
  \\
  &-
  \sum_{\substack{1\le l\le n-2\\\text{$l$: odd}}}
  e^{bl(al+2)\pi\sqrt{-1}/n}
  \left(e^{2(al+1)\pi\sqrt{-1}/n}-e^{-2(al+1)\pi\sqrt{-1}/n}\right)
  \\
  =&
  \sum_{l=0}^{n-1}
  (-1)^le^{bl(al+2)\pi\sqrt{-1}/n}
  \left(e^{2(al+1)\pi\sqrt{-1}/n}-e^{-2(al+1)\pi\sqrt{-1}/n}\right).
\end{split}
\end{equation*}
Since we have
\begin{equation*}
\begin{split}
  &
  \sum_{l=0}^{n-1}
  (-1)^le^{bl(al+2)\pi\sqrt{-1}/n}
  \left(e^{2(al+1)\pi\sqrt{-1}/n}-e^{-2(al+1)\pi\sqrt{-1}/n}\right)
  \\
  =&
  \sum_{l=1}^{n}
  (-1)^{n-l}e^{b(n-l)(a(n-l)+2)\pi\sqrt{-1}/n}
  \left(e^{2(a(n-l)+1)\pi\sqrt{-1}/n}-e^{-2(a(n-l)+1)\pi\sqrt{-1}/n}\right)
  \\
  =&
  \sum_{l=0}^{n-1}
  (-1)^le^{bl(al-2)\pi\sqrt{-1}/n}
  \left(e^{-2(al-1)\pi\sqrt{-1}/n}-e^{2(al-1)\pi\sqrt{-1}/n}\right),
\end{split}
\end{equation*}
we conclude that
\begin{equation}\label{eq:W_iii_6}
\begin{split}
  \tilde{J}_n\left(e^{4\pi\sqrt{-1}/n}\right)
  =&
  \frac{1}{2}
  \sum_{l=0}^{n-1}
  (-1)^le^{bl(al+2)\pi\sqrt{-1}/n}
  \left(e^{2(al+1)\pi\sqrt{-1}/n}-e^{-2(al+1)\pi \sqrt{-1}/n}\right)
  \\
  &+
  \frac{1}{2}
  \sum_{l=0}^{n-1}
  (-1)^le^{bl(al-2)\pi\sqrt{-1}/n}
  \left(e^{-2(al-1)\pi\sqrt{-1}/n}-e^{2(al-1)\pi\sqrt{-1}/n}\right).
\end{split}
\end{equation}
\par
Now we have
\begin{equation}\label{eq:W_iii_7}
\begin{split}
  &
  \sum_{l=0}^{n-1}
  (-1)^le^{bl(al+2)\pi\sqrt{-1}/n}
  \left(e^{2(al+1)\pi\sqrt{-1}/n}-e^{-2(al+1)\pi \sqrt{-1}/n}\right)
  \\
  =&
  \sum_{l=0}^{(n-1)/2}
  (-1)^le^{bl(al+2)\pi\sqrt{-1}/n}
  \left(e^{2(al+1)\pi\sqrt{-1}/n}-e^{-2(al+1)\pi\sqrt{-1}/n}\right)
  \\
  &+
  \sum_{l=1}^{(n-1)/2}
  (-1)^{n-l}e^{b(n-l)(a(n-l)+2)\pi\sqrt{-1}/n}
  \left(e^{2(a(n-l)+1)\pi\sqrt{-1}/n}-e^{-2(a(n-l)+1)\pi\sqrt{-1}/n}\right)
  \\
  =&
  \sum_{l=0}^{(n-1)/2}
  (-1)^le^{bl(al+2)\pi\sqrt{-1}/n}
  \left(e^{2(al+1)\pi\sqrt{-1}/n}-e^{-2(al+1)\pi\sqrt{-1}/n}\right)
  \\
  &+
  \sum_{l=1}^{(n-1)/2}
  (-1)^le^{bl(al-2)\pi\sqrt{-1}/n}
  \left(e^{-2(al-1)\pi\sqrt{-1}/n}-e^{2(al-1)\pi\sqrt{-1}/n}\right)
\end{split}
\end{equation}
and
\begin{equation}\label{eq:W_iii_8}
\begin{split}
  &
  \sum_{l=0}^{n-1}
  (-1)^le^{bl(al-2)\pi\sqrt{-1}/n}
  \left(e^{-2(al-1)\pi\sqrt{-1}/n}-e^{2(al-1)\pi\sqrt{-1}/n}\right)
  \\
  =&
  \sum_{l=0}^{(n-1)/2}
  (-1)^le^{bl(al-2)\pi\sqrt{-1}/n}
  \left(e^{-2(al-1)\pi\sqrt{-1}/n}-e^{2(al-1)\pi\sqrt{-1}/n}\right)
  \\
  &+
  \sum_{l=1}^{(n-1)/2}
  (-1)^le^{bl(al+2)\pi\sqrt{-1}/n}
  \left(e^{2(al+1)\pi\sqrt{-1}/n}-e^{-2(al+1)\pi\sqrt{-1}/n}\right).
\end{split}
\end{equation}
\par
Adding \eqref{eq:W_iii_7} and \eqref{eq:W_iii_8}, and dividing it by two, we obtain from \eqref{eq:W_iii_6}
\begin{equation*}
\begin{split}
  \tilde{J}_n\left(e^{4\pi\sqrt{-1}/n}\right)
  &=
  \sum_{l=0}^{(n-1)/2}
  (-1)^le^{bl(al+2)\pi\sqrt{-1}/n}
  \left(e^{2(al+1)\pi\sqrt{-1}/n}-e^{-2(al+1)\pi\sqrt{-1}/n}\right)
  \\
  &\quad+
  \sum_{l=0}^{(n-1)/2}
  (-1)^le^{bl(al-2)\pi\sqrt{-1}/n}
  \left(e^{-2(al-1)\pi\sqrt{-1}/n}-e^{2(al-1)\pi\sqrt{-1}/n}\right)
  \\
  &\quad-
  \left(e^{2\pi\sqrt{-1}/n}-e^{-2\pi\sqrt{-1}/n}\right).
\end{split}
\end{equation*}
Therefore we finally have
\begin{equation*}
  \tilde{W}_{\rm{iii}}
  =
  2e^{abn\pi\sqrt{-1}/4}
  \widetilde{J}_n\left(e^{4\pi \sqrt{-1}/n}\right)
  =
  0
\end{equation*}
and
\begin{equation*}
  -4W_{\rm{iii}}
  =
  \sqrt{\frac{ab}{n}}e^{-\pi\sqrt{-1}/4}
  e^{\frac{a\pi\sqrt{-1}}{bn}+\frac{b\pi\sqrt{-1}}{an}}
  \tilde{W}_{\rm{iii}}
  =
  0.
\end{equation*}
This completes the proof.
\end{proof}

\bibliography{mrabbrev,hitoshi}

\def\cprime{$'$}
\providecommand{\bysame}{\leavevmode\hbox to3em{\hrulefill}\thinspace}
\providecommand{\MR}{\relax\ifhmode\unskip\space\fi MR }
\providecommand{\MRhref}[2]{%
  \href{http://www.ams.org/mathscinet-getitem?mr=#1}{#2}
}
\providecommand{\href}[2]{#2}
\begin{thebibliography}{10}

\bibitem{Andersen:problem}
J.~E. Andersen, \emph{The asymptotic expansion conjecture, {S}ection 7.2 of
  ``{P}roblems on invariants of knots and 3-manifolds''}, Invariants of knots
  and 3-manifolds ({K}yoto, 2001) (T.~Ohtsuki, ed.), Geom. Topol. Monogr.,
  vol.~4, Geom. Topol. Publ., Coventry, 2002, pp.~474--481. \MR{2065029}

\bibitem{Andersen:JREIA2013}
\bysame, \emph{The {W}itten-{R}eshetikhin-{T}uraev invariants of finite order
  mapping tori {I}}, J. Reine Angew. Math. \textbf{681} (2013), 1--38.
  \MR{3181488}

\bibitem{Andersen/Hansen:JKNOT2006}
J.~E. Andersen and S.~K. Hansen, \emph{Asymptotics of the quantum invariants
  for surgeries on the figure 8 knot}, J. Knot Theory Ramifications \textbf{15}
  (2006), no.~4, 479--548. \MR{2221531}

\bibitem{Andersen/Himpel:QT2012}
J.~E. Andersen and B.~Himpel, \emph{The {W}itten-{R}eshetikhin-{T}uraev
  invariants of finite order mapping tori {II}}, Quantum Topol. \textbf{3}
  (2012), no.~3-4, 377--421. \MR{2928090}

\bibitem{Andersen/Petersen:2018}
J.~E. Andersen and W.~E. Petersen, \emph{Resurgence analysis of quantum
  invariants: Seifert manifolds and surgeries on the figure eight knot},
  arXiv:1811.05376, 2018.

\bibitem{Benedetti/Petronio:JKNOT1996}
R.~Benedetti and C.~Petronio, \emph{On {R}oberts' proof of the
  {T}uraev-{W}alker theorem}, J. Knot Theory Ramifications \textbf{5} (1996),
  no.~4, 427--439. \MR{1406714}

\bibitem{Blanchet/Habegger/Masbaum/Vogel:TOPOL1992}
C.~Blanchet, N.~Habegger, G.~Masbaum, and P.~Vogel, \emph{Three-manifold
  invariants derived from the {K}auffman bracket}, Topology \textbf{31} (1992),
  no.~4, 685--699. \MR{1191373}

\bibitem{Chandrasekharan:1985}
K.~Chandrasekharan, \emph{Elliptic functions}, Grundlehren der Mathematischen
  Wissenschaften [Fundamental Principles of Mathematical Sciences], vol. 281,
  Springer-Verlag, Berlin, 1985. \MR{808396}

\bibitem{Cheeger/Simons:LNM1167}
J.~Cheeger and J.~Simons, \emph{Differential characters and geometric
  invariants}, Geometry and topology ({C}ollege {P}ark, {M}d., 1983/84),
  Lecture Notes in Math., vol. 1167, Springer, Berlin, 1985, pp.~50--80.
  \MR{827262}

\bibitem{Chen/Yang:QT2018}
Q.~Chen and T.~Yang, \emph{Volume conjectures for the {R}eshetikhin-{T}uraev
  and the {T}uraev-{V}iro invariants}, Quantum Topol. \textbf{9} (2018), no.~3,
  419--460. \MR{3827806}

\bibitem{Chern/Simons:ANNMA21974}
S.~S. Chern and J.~Simons, \emph{Characteristic forms and geometric
  invariants}, Ann. of Math. (2) \textbf{99} (1974), 48--69. \MR{0353327}

\bibitem{Cochran/Gompf:TOPOL1988}
T.~D. Cochran and R.~E. Gompf, \emph{Applications of {D}onaldson's theorems to
  classical knot concordance, homology {$3$}-spheres and property {$P$}},
  Topology \textbf{27} (1988), no.~4, 495--512. \MR{976591}

\bibitem{Detcherry/Kalfagianni:ANNSE12019}
R.~Detcherry and E.~Kalfagianni, \emph{Gromov norm and turaev-viro invariants
  of 3-manifolds}, arXiv:1705.09964, to appear in Ann. Sci. {\'E}cole Norm.
  Sup., 2019.

\bibitem{Detcherry/Kalfagianni/Yang:QT2018}
R.~Detcherry, E.~Kalfagianni, and T.~Yang, \emph{Turaev-{V}iro invariants,
  colored {J}ones polynomials, and volume}, Quantum Topol. \textbf{9} (2018),
  no.~4, 775--813. \MR{3874003}

\bibitem{Dubois:CANMB2006}
J.~Dubois, \emph{Non abelian twisted {R}eidemeister torsion for fibered knots},
  Canad. Math. Bull. \textbf{49} (2006), no.~1, 55--71. \MR{2198719}

\bibitem{Freed/Gompf:COMMP1991}
D.~S. Freed and R.~E. Gompf, \emph{Computer calculation of {W}itten's
  {$3$}-manifold invariant}, Comm. Math. Phys. \textbf{141} (1991), no.~1,
  79--117. \MR{1133261}

\bibitem{Gang/Romo/Yamazaki:COMMP2018}
D.~Gang, M.~Romo, and M.~Yamazaki, \emph{All-order volume conjecture for closed
  3-manifolds from complex {C}hern-{S}imons theory}, Comm. Math. Phys.
  \textbf{359} (2018), no.~3, 915--936. \MR{3784535}

\bibitem{Garoufalidis/Thurston/Zickert:DUKMJ2015}
S.~Garoufalidis, D.~P. Thurston, and C.~K. Zickert, \emph{The complex volume of
  {${\rm SL}(n,\mathbb{C})$}-representations of 3-manifolds}, Duke Math. J.
  \textbf{164} (2015), no.~11, 2099--2160. \MR{3385130}

\bibitem{Gromov:INSHE82}
M.~Gromov, \emph{Volume and bounded cohomology}, Inst. Hautes \'{E}tudes Sci.
  Publ. Math. (1982), no.~56, 5--99 (1983). \MR{686042}

\bibitem{Hansen2005}
S.~K. Hansen, \emph{{Analytic asymptotic expansions of the Reshetikhin--Turaev
  invariants of Seifert 3-manifolds for \rm{SU(2)}}}, arXiv:math/0510549
  [math.QA], 2005.

\bibitem{HikamiCOMMP2006}
K.~Hikami, \emph{On the quantum invariants for the spherical {S}eifert
  manifolds}, Comm. Math. Phys. \textbf{268} (2006), no.~2, 285--319.
  \MR{2259197}

\bibitem{Ireland/Rosen:GTM84}
K.~Ireland and M.~Rosen, \emph{A classical introduction to modern number
  theory}, second ed., Graduate Texts in Mathematics, vol.~84, Springer-Verlag,
  New York, 1990. \MR{1070716}

\bibitem{Jaco/Shalen:MEMAM1979}
W.~H. Jaco and P.~B. Shalen, \emph{Seifert fibered spaces in {$3$}-manifolds},
  Mem. Amer. Math. Soc. \textbf{21} (1979), no.~220, viii+192. \MR{539411}

\bibitem{Jeffrey:COMMP1992}
L.~C. Jeffrey, \emph{Chern-{S}imons-{W}itten invariants of lens spaces and
  torus bundles, and the semiclassical approximation}, Comm. Math. Phys.
  \textbf{147} (1992), no.~3, 563--604. \MR{1175494}

\bibitem{Johannson:1979}
K.~Johannson, \emph{Homotopy equivalences of {$3$}-manifolds with boundaries},
  Lecture Notes in Mathematics, vol. 761, Springer, Berlin, 1979. \MR{551744}

\bibitem{Jones:BULAM31985}
V.~F.~R. Jones, \emph{A polynomial invariant for knots via von {N}eumann
  algebras}, Bull. Amer. Math. Soc. (N.S.) \textbf{12} (1985), no.~1, 103--111.
  \MR{766964}

\bibitem{Kashaev:MODPLA95}
R.~M. Kashaev, \emph{A link invariant from quantum dilogarithm}, Modern Phys.
  Lett. A \textbf{10} (1995), no.~19, 1409--1418. \MR{1341338}

\bibitem{Kashaev:LETMP97}
\bysame, \emph{The hyperbolic volume of knots from the quantum dilogarithm},
  Lett. Math. Phys. \textbf{39} (1997), no.~3, 269--275. \MR{1434238}

\bibitem{Kauffman:TOPOL1987}
L.~H. Kauffman, \emph{State models and the {J}ones polynomial}, Topology
  \textbf{26} (1987), no.~3, 395--407. \MR{88f:57006}

\bibitem{Kirby/Melvin:INVEM1991}
R.~Kirby and P.~Melvin, \emph{The {$3$}-manifold invariants of {W}itten and
  {R}eshetikhin-{T}uraev for {${\rm sl}(2,{\bf C})$}}, Invent. Math.
  \textbf{105} (1991), no.~3, 473--545. \MR{1117149}

\bibitem{Kirk/Klassen:COMMP1993}
P.~Kirk and E.~Klassen, \emph{Chern-{S}imons invariants of {$3$}-manifolds
  decomposed along tori and the circle bundle over the representation space of
  {$T^2$}}, Comm. Math. Phys. \textbf{153} (1993), no.~3, 521--557.
  \MR{1218931}

\bibitem{Kirk/Klassen:MATHA1990}
P.~A. Kirk and E.~P. Klassen, \emph{Chern-{S}imons invariants of
  {$3$}-manifolds and representation spaces of knot groups}, Math. Ann.
  \textbf{287} (1990), no.~2, 343--367. \MR{1054574}

\bibitem{Klassen:TRAAM1991}
E.~P. Klassen, \emph{Representations of knot groups in {${\rm SU}(2)$}}, Trans.
  Amer. Math. Soc. \textbf{326} (1991), no.~2, 795--828. \MR{1008696}

\bibitem{Lam/Ang:SunZi}
L.~Y. Lam and T.~S. Ang, \emph{Fleeting footsteps}, revised ed., World
  Scientific Publishing Co., Inc., River Edge, NJ, 2004, Tracing the conception
  of arithmetic and algebra in ancient China, With a foreword by Joseph W.
  Dauben. \MR{2092881}

\bibitem{Le/Tran:JKNOT2010}
T.~T.~Q. Le and A.~T. Tran, \emph{On the volume conjecture for cables of
  knots}, J. Knot Theory Ramifications \textbf{19} (2010), no.~12, 1673--1691.
  \MR{2755495}

\bibitem{Lickorish:JKNOT1993}
W.~B.~R. Lickorish, \emph{The skein method for three-manifold invariants}, J.
  Knot Theory Ramifications \textbf{2} (1993), no.~2, 171--194. \MR{1227009}

\bibitem{Lickorish:1997}
\bysame, \emph{An introduction to knot theory}, Graduate Texts in Mathematics,
  vol. 175, Springer-Verlag, New York, 1997. \MR{1472978}

\bibitem{Maria/Rouille:arXiv2020}
C.~Maria and O.~Rouill{\'e}, \emph{Computation of large asymptotics of
  3-manifold quantum invariants}, arXiv:2010.14316, 2020.

\bibitem{Meyerhoff:LMSLN112}
R.~Meyerhoff, \emph{Density of the {C}hern-{S}imons invariant for hyperbolic
  {$3$}-manifolds}, Low-dimensional topology and {K}leinian groups
  ({C}oventry/{D}urham, 1984), London Math. Soc. Lecture Note Ser., vol. 112,
  Cambridge Univ. Press, Cambridge, 1986, pp.~217--239. \MR{903867}

\bibitem{Milnor:ANNMA21962}
J.~Milnor, \emph{A duality theorem for {R}eidemeister torsion}, Ann. of Math.
  (2) \textbf{76} (1962), 137--147. \MR{0141115}

\bibitem{Morton:MATPC1995}
H.~R. Morton, \emph{The coloured {J}ones function and {A}lexander polynomial
  for torus knots}, Math. Proc. Cambridge Philos. Soc. \textbf{117} (1995),
  no.~1, 129--135. \MR{1297899}

\bibitem{Moser:PACJM1971}
L.~Moser, \emph{Elementary surgery along a torus knot}, Pacific J. Math.
  \textbf{38} (1971), 737--745. \MR{0383406}

\bibitem{Munoz:REVMC2009}
V.~Mu{\~n}oz, \emph{The {${\rm SL}(2,\mathbb C)$}-character varieties of torus
  knots}, Rev. Mat. Complut. \textbf{22} (2009), no.~2, 489--497. \MR{2553945}

\bibitem{Murakami:SURIK2000}
H.~Murakami, \emph{Optimistic calculations about the
  {W}itten-{R}eshetikhin-{T}uraev invariants of closed three-manifolds obtained
  from the figure-eight knot by integral {D}ehn surgeries},
  S\={u}rikaisekikenky\={u}sho K\={o}ky\={u}roku (2000), no.~1172, 70--79,
  Recent progress towards the volume conjecture (Japanese) (Kyoto, 2000).
  \MR{1805729}

\bibitem{Murakami/Murakami:ACTAM12001}
H.~Murakami and J.~Murakami, \emph{The colored {J}ones polynomials and the
  simplicial volume of a knot}, Acta Math. \textbf{186} (2001), no.~1, 85--104.
  \MR{1828373}

\bibitem{Murakami/Murakami/Okamoto/Takata/Yokota:EXPMA02}
H.~Murakami, J.~Murakami, M.~Okamoto, T.~Takata, and Y.~Yokota, \emph{Kashaev's
  conjecture and the {C}hern-{S}imons invariants of knots and links},
  Experiment. Math. \textbf{11} (2002), no.~3, 427--435. \MR{1959752}

\bibitem{Murakami/Yokota:2018}
H.~Murakami and Y.~Yokota, \emph{Volume conjecture for knots}, SpringerBriefs
  in Mathematical Physics, vol.~30, Springer, Singapore, 2018. \MR{3837111}

\bibitem{Ohtsuki:2002}
T.~Ohtsuki, \emph{Problems on invariants of knots and 3-manifolds}, Invariants
  of knots and 3-manifolds ({K}yoto, 2001), Geom. Topol. Monogr., vol.~4, Geom.
  Topol. Publ., Coventry, 2002, With an introduction by J. Roberts, pp.~i--iv,
  377--572. \MR{2065029}

\bibitem{Ohtsuki:QT2016}
\bysame, \emph{On the asymptotic expansion of the {K}ashaev invariant of the
  {$5_2$} knot}, Quantum Topol. \textbf{7} (2016), no.~4, 669--735.
  \MR{3593566}

\bibitem{Ohtsuki:AGT2018}
\bysame, \emph{On the asymptotic expansion of the quantum {$\rm SU(2)$}
  invariant at {$q=\exp(4\pi\sqrt{-1}/N)$} for closed hyperbolic 3-manifolds
  obtained by integral surgery along the figure-eight knot}, Algebr. Geom.
  Topol. \textbf{18} (2018), no.~7, 4187--4274. \MR{3892244}

\bibitem{Ohtsuki/Takata::COMMP2019}
T.~Ohtsuki and T.~Takata, \emph{On the quantum {${\rm SU}(2)$} invariant at
  {$q={\rm exp}(4\pi\sqrt{-1}/N)$} and the twisted {R}eidemeister torsion for
  some closed 3-manifolds}, Comm. Math. Phys. \textbf{370} (2019), no.~1,
  151--204. \MR{3982693}

\bibitem{Ohtsuki/Yokota:MATPC2018}
T.~Ohtsuki and Y.~Yokota, \emph{On the asymptotic expansions of the {K}ashaev
  invariant of the knots with 6 crossings}, Math. Proc. Cambridge Philos. Soc.
  \textbf{165} (2018), no.~2, 287--339. \MR{3834003}

\bibitem{Porti:MAMCAU1997}
J.~Porti, \emph{Torsion de {R}eidemeister pour les vari\'{e}t\'{e}s
  hyperboliques}, Mem. Amer. Math. Soc. \textbf{128} (1997), no.~612, x+139.
  \MR{1396960}

\bibitem{Reshetikhin/Turaev:INVEM1991}
N.~Reshetikhin and V.~G. Turaev, \emph{Invariants of {$3$}-manifolds via link
  polynomials and quantum groups}, Invent. Math. \textbf{103} (1991), no.~3,
  547--597. \MR{1091619}

\bibitem{Roberts:TOPOL1995}
J.~Roberts, \emph{Skein theory and {T}uraev-{V}iro invariants}, Topology
  \textbf{34} (1995), no.~4, 771--787. \MR{1362787}

\bibitem{Rolfsen:PACJM1984}
D.~Rolfsen, \emph{Rational surgery calculus: extension of {K}irby's theorem},
  Pacific J. Math. \textbf{110} (1984), no.~2, 377--386. \MR{726496}

\bibitem{Rolfsen:1990}
\bysame, \emph{Knots and links}, Mathematics Lecture Series, vol.~7, Publish or
  Perish, Inc., Houston, TX, 1990, Corrected reprint of the 1976 original.
  \MR{1277811 (95c:57018)}

\bibitem{Rosso/Jones:JKNOT1993}
M.~Rosso and V.~Jones, \emph{On the invariants of torus knots derived from
  quantum groups}, J. Knot Theory Ramifications \textbf{2} (1993), no.~1,
  97--112. \MR{1209320}

\bibitem{Rozansky:COMMP1995}
L.~Rozansky, \emph{A large {$k$} asymptotics of {W}itten's invariant of
  {S}eifert manifolds}, Comm. Math. Phys. \textbf{171} (1995), no.~2, 279--322.
  \MR{1344728}

\bibitem{Rozansky:COMMP1996_2}
\bysame, \emph{Residue formulas for the large {$k$} asymptotics of {W}itten's
  invariants of {S}eifert manifolds. {T}he case of {${\rm SU}(2)$}}, Comm.
  Math. Phys. \textbf{178} (1996), no.~1, 27--60. \MR{1387940}

\bibitem{Seifert:ACTAM11933}
H.~Seifert, \emph{Topologie {D}reidimensionaler {G}efaserter {R}\"{a}ume}, Acta
  Math. \textbf{60} (1933), no.~1, 147--238. \MR{1555366}

\bibitem{Seifert/Threlfall:Topology}
H.~Seifert and W.~Threlfall, \emph{Seifert and {T}hrelfall: a textbook of
  topology}, Pure and Applied Mathematics, vol.~89, Academic Press, Inc.
  [Harcourt Brace Jovanovich, Publishers], New York-London, 1980, Translated
  from the German edition of 1934 by Michael A. Goldman, With a preface by Joan
  S. Birman, With ``Topology of $3$-dimensional fibered spaces'' by Seifert,
  Translated from the German by Wolfgang Heil. \MR{575168}

\bibitem{Thurston:GT3M}
W.~P. Thurston, \emph{{The Geometry and Topology of Three-Manifolds}},
  Electronic version 1.1 - March 2002,
  http://www.msri.org/publications/books/gt3m/.

\bibitem{Turaev:USPMN1986}
V.~Turaev, \emph{Reidemeister torsion in knot theory}, Uspekhi Mat. Nauk
  \textbf{41} (1986), no.~1(247), 97--147, 240. \MR{832411}

\bibitem{Turaev/Viro:TOPOL1992}
V.~G. Turaev and O.~Ya. Viro, \emph{State sum invariants of {$3$}-manifolds and
  quantum {$6j$}-symbols}, Topology \textbf{31} (1992), no.~4, 865--902.
  \MR{1191386}

\bibitem{Wenzl:CRMAR1987}
H.~Wenzl, \emph{On sequences of projections}, C. R. Math. Rep. Acad. Sci.
  Canada \textbf{9} (1987), no.~1, 5--9. \MR{88k:46070}

\bibitem{Witten:COMMP1989}
E.~Witten, \emph{Quantum field theory and the {J}ones polynomial}, Comm. Math.
  Phys. \textbf{121} (1989), no.~3, 351--399. \MR{990772}

\bibitem{Zheng:CHIAM22007}
H.~Zheng, \emph{Proof of the volume conjecture for {W}hitehead doubles of a
  family of torus knots}, Chin. Ann. Math. Ser. B \textbf{28} (2007), no.~4,
  375--388. \MR{MR2348452}

\end{thebibliography}
\bibliographystyle{amsplain}

\end{document}